\documentclass{amsart}
\usepackage{amsmath}
\usepackage{amsfonts}
\usepackage{amsthm}
\usepackage{amsbsy}
\usepackage{amssymb}
\usepackage{amsthm}
\usepackage{enumerate}
\usepackage[margin=1in]{geometry}
\usepackage{hyperref}
\usepackage{mathrsfs}
\usepackage{mathtools}
\usepackage{microtype}
\usepackage[nobysame]{amsrefs}

\newtheorem{thm}{Theorem}[section]
\newtheorem{lem}[thm]{Lemma}
\newtheorem{prop}[thm]{Proposition}
\newtheorem{cor}[thm]{Corollary}

\theoremstyle{definition}
\newtheorem{defn}[thm]{Definition}

\newtheorem{rem}[thm]{Remark}
\newtheorem{exam}[thm]{Example}

\newcommand{\bC}{{\mathbb{C}}}

\newcommand{\bR}{{\mathbb{R}}}

\newcommand{\A}{{\mathcal{A}}}
\newcommand{\B}{{\mathcal{B}}}
\newcommand{\C}{{\mathcal{C}}}
\newcommand{\D}{{\mathcal{D}}}

\newcommand{\F}{{\mathcal{F}}}

\newcommand{\K}{{\mathcal{K}}}
\renewcommand{\L}{{\mathcal{L}}}

\newcommand{\N}{{\mathcal{N}}}

\newcommand{\R}{{\mathcal{R}}}

\newcommand{\X}{{\mathcal{X}}}

\newcommand{\rc}{{\mathrm{c}}}

\newcommand{\fB}{{\mathfrak{B}}}

\newcommand{\qand}{\quad\text{and}\quad}

\newcommand{\lat}{\mathrm{lat}}
\newcommand{\capp}{\mathrm{cap}}

\newcommand{\BNC}{{\mathcal{BNC}}}
\newcommand{\NC}{{\mathcal{NC}}}

\usepackage{tikz}
\usetikzlibrary{shapes,snakes,calc,arrows}
\usetikzlibrary{decorations.pathreplacing,shapes.geometric}
\usetikzlibrary{calc,positioning}
\tikzset{Box/.style={very thick, rounded corners}}
\tikzset{marked/.style={star, star point height = .75mm, star points =5, fill=black,minimum size=2mm, inner sep=0mm} }
\tikzset{verythickline/.style = {line width=7pt}}
\tikzset{thickline/.style = {line width=5pt}}
\tikzset{medthick/.style = {line width=3pt}}
\tikzset{med/.style = {line width=2pt}}
\tikzset{count/.style = {fill=white,circle,draw,thin, inner sep=2pt}}
\tikzset{rcount/.style = {fill=white,rectangle,draw,thin,inner sep=2pt, rounded corners}}
\tikzset{cpr/.style = {draw,fill=white,rectangle,thin, rounded corners}}

\definecolor{ggreen}{HTML}{00BB33}

\begin{document}

\nocite{*}

\title[Conditionally bi-free independence for pairs of algebras]{Conditionally bi-free independence for pairs of algebras}

\author{Yinzheng Gu and Paul Skoufranis}

\address{Department of Mathematics and Statistics, Queen's University, Jeffery Hall, Kingston, Ontario, K7L 3N6, Canada}
\email{gu.y@queensu.ca}

\address{Department of Mathematics and Statistics, York University, 4700 Keele Street, Toronto, Ontario, M3J 1P3, Canada}
\email{pskoufra@yorku.ca}

\date{\today}
\subjclass[2010]{Primary 46L54; Secondary 46L53, 60E07, 60F05.}
\keywords{conditionally bi-free independence, conditional $(\ell, r)$-cumulants, infinite divisibility, limit theorems.}
\thanks{The work of Yinzheng Gu was partially supported by CIMI (Centre International de Math\'{e}matiques et d'Informatique) Excellence Program while visiting the Institute of Mathematics of Toulouse. He would like to thank the institute for the generous hospitality and Serban Belinschi for his constant support and valuable advice when this research was conducted.}

\begin{abstract}
In this paper, the notion of conditionally bi-free independence for pairs of algebras is introduced. The notion of conditional $(\ell, r)$-cumulants are introduced and it is demonstrated that conditionally bi-free independence is equivalent to mixed cumulants. Furthermore, limit theorems for the additive conditionally bi-free convolution are studied using both combinatorial and analytic techniques. In particular, a conditionally bi-free partial $\mathcal{R}$-transform is constructed and a conditionally bi-free analogue of the L\'{e}vy-Hin\v{c}in formula for planar Borel probability measures is derived.
\end{abstract}

\maketitle

\section{Introduction}

The basic framework in non-commutative probability theory is a pair $(\A, \psi)$, called a non-commutative probability space, where $\A$ is a (complex) unital algebra and $\psi$ is a unital linear functional on $\A$. Subalgebras of $\A$ are said have a certain independence with respect to $\psi$ if there is a specific rule of calculating the joint distributions.  There are several important notions of independence in the literature. According to \cites{M2002, M2003} there are exactly five notions of universal/natural independence: classical, free, Boolean, monotone, and anti-monotone. These notions of independence have very similar theories such as the combinatorics of cumulants and the analytic aspects of convolutions on probability measures. On the other hand, Bo\.{z}ejko, Leinert, and Speicher \cite{BLS1996} introduced conditionally free independence as a notion of independence with respect to a pair of unital linear functionals $(\varphi, \psi)$ on a unital algebra $\A$. Although mainly intended as a generalization of free independence, it turned out (see \cites{BLS1996, F2006}) that Boolean and monotone independences, especially their relative convolutions, can also be unified in terms of conditionally free independence.

Free probability for pairs of faces, or bi-free probability for short, is a generalization of free probability introduced by Voiculescu \cite{V2014} in order to study the non-commutative left and right actions of algebras on a reduced free product space simultaneously. Again, the basic framework is a non-commutative probability space $(\A, \psi)$, but the corresponding independence, called bi-free independence, is defined for pairs of subalgebras of $\A$ instead. Since its inception, bi-free probability has received a lot of attention as many old results from free probability have been extended to the bi-free setting and new results have been developed. In particular, it was noticed in \cite{V2014} that both classical and free independences can be viewed as specific cases of bi-free independence and it was noticed in \cite{S2016} that Boolean and monotone independences also occur in bi-free probability.  Thus bi-free probability is in a certain sense another unifying theory. It is then natural to combine the two mentioned unifying theories together and develop a notion of conditionally bi-free independence, which is the main focus of this paper.

This paper contains six sections, including this introduction, which are structured as follows. In Section \ref{sec:prelims}, basic notions and results from bi-free and conditionally free probability theories are recalled, with an emphasis on the combinatorial aspects.

In Section \ref{sec:c-bi-free-defns}, two notions of conditionally bi-free independence are provided.  The first arises naturally by combining the constructions of bi-free and conditionally free independences.  The second is defined as the vanishing of certain cumulants. More precisely, as bi-free independence can be characterized by the vanishing of mixed $(\ell, r)$-cumulants, and as conditionally free independence can be characterized by the vanishing of mixed free and c-free cumulants, we introduce the family of c-$(\ell, r)$-cumulants and define combinatorially c-bi-free independence as the vanishing of mixed $(\ell, r)$- and c-$(\ell, r)$-cumulants. 

In Section \ref{sec:equivalence-of-defn}, it is demonstrated that a collection of pairs of algebras is conditionally bi-free independent if and only if it is combinatorially conditionally bi-free independent.  To achieve this result, moment formulae for joint distributions for conditionally bi-free independent pairs are obtained.  These formulae are along similar lines to those obtained in \cite{CNS2015-1} and make use of a diagrammatic approach.  It is then shown that combinatorially c-bi-free independence implies a priori different moment formulae for joint distributions which are then shown to coincide with those for conditionally bi-free independence.

In Section \ref{sec:R-transform}, a conditionally bi-free partial $\R$-transform is constructed which, along with Voiculescu's bi-free partial $\R$-transform \cite{V2016}, linearize the additive c-bi-free convolution. A functional equation is also derived relating said $\R$-transform with the Cauchy transform.

Finally, Section \ref{sec:additive-limit-theorems} studies various limit theorems using both combinatorial (which relies heavily on the relations between moments and cumulants) and analytic (which uses complex analysis methods to deal with measures without any assumption on finite moments) techniques. In particular, infinite divisibility with respect to the additive c-bi-free convolution is defined and studied, and a conditionally bi-free L\'{e}vy-Hin\v{c}in formula is presented.

\section{Preliminaries}
\label{sec:prelims}

In this section, we briefly review bi-free independence and conditionally free independence, and develop notation that will be used throughout the paper.

\subsection{Bi-free independence}

Recall a \emph{pair of algebras} in a non-commutative probability space $(\A, \psi)$ is an ordered pair $(\A_\ell, \A_r)$ of unital subalgebras of $\A$. We call $\mathcal{A}_\ell$ the \emph{left algebra} and $\A_r$ the \emph{right algebra}.  A family of pairs of algebras is said to be \emph{bi-freely independent} with respect to $\psi$ if the joint distributions can be realized using non-commutative left and right actions of the algebras on reduced free product spaces (see \cite{V2014}*{Section 2} for more precision).  Below we recall the combinatorial theory of bi-free probability developed in \cites{MN2015, CNS2015-1, CNS2015-2}.

\begin{defn}
For $n \geq 1$, denote by $[n]$ the set $\{1, \dots, n\}$. Given a map $\chi: [n] \to \{\ell, r\}$ with 
\[
\chi^{-1}(\{\ell\}) = \{i_1 < \cdots < i_p\} \qand \chi^{-1}(\{r\}) = \{i_{p + 1} > \cdots > i_n\},
\]
define a permutation $s_\chi$ on $[n]$ via $s_\chi(k) = i_k$ and a total order $\prec_\chi$ on $[n]$ by
\[
i_1 \prec_\chi \cdots \prec_\chi i_p \prec_\chi i_{p + 1} \prec_\chi \cdots \prec_\chi i_n.
\]
Equivalently, for $a, b \in [n]$, $a \prec_\chi b$ if and only if $s_\chi^{-1}(a) < s_\chi^{-1}(b)$.

A partition $\pi$ of $[n]$ is said to be \emph{bi-non-crossing} with respect to $\chi$ if $s_\chi^{-1} \cdot \pi \in \NC(n)$ (the set of non-crossing partitions of $[n]$); that is, $\pi$ is a non-crossing partition on $[n]$ under the $\prec_\chi$-ordering. The set of all bi-non-crossing partitions with respect to $\chi$ is denoted by $\BNC(\chi)$ and the minimum and maximum (with respect to the order that $\pi \leq \sigma$ if $\pi$ is a refinement of $\sigma$) elements of $\BNC(\chi)$ are denoted by $0_\chi$ and $1_\chi$ respectively.
\end{defn}

Bi-non-crossing partitions corresponding to a given $\chi: [n] \to \{\ell, r\}$ can be represented diagrammatically by placing $n$ nodes labelled $1$ to $n$ on two parallel vertical transparent lines from top to bottom in increasing order with node $k$ on the left or right depending on whether $\chi(k) = \ell$ or $\chi(k) = r$, and drawing the partition in a non-crossing way between the vertical lines on the $n$ nodes. Moreover, given a bi-non-crossing partition $\pi \in \BNC(\chi)$, it is possible to draw the diagram of $\pi$ using only horizontal and vertical lines.  The vertical segment of a block $V \in \pi$ will be referred to as the \emph{spine} of $V$ and the horizontal segments connecting the nodes to the spine of $V$ will be referred to as the \emph{ribs} of $V$. We refer to \cite{CNS2015-1}*{Section 2} for more details.

\begin{defn}[\cite{MN2015}]
Let $(\A, \psi)$ be a non-commutative probability space. The family of \emph{$(\ell, r)$-cumulants} with respect to $\psi$ is the family of multilinear functionals
\[
\{\kappa_\chi: \A^n \to \mathbb{C}\}_{n \geq 1, \chi: [n] \to \{\ell, r\}}
\]
uniquely determined by the requirement that
\[
\psi(a_1\cdots a_n) = \sum_{\pi \in \BNC(\chi)}\left(\prod_{V \in \pi}\kappa_{\chi|_V}((a_1, \dots, a_n)|_V)\right)
\]
for all $n \geq 1$, $\chi: [n] \to \{\ell, r\}$, and $a_1, \dots, a_n \in \A$.
\end{defn}

Since $\BNC(\chi)$ inherits a special lattice structure from the set of all partitions of $[n]$, an equivalent formulation of the above moment-cumulant formula is given for $\pi \in \BNC(\chi)$ by
\[
\kappa_\pi(a_1, \dots, a_n) = \sum_{\substack{\sigma \in \BNC(\chi)\\\sigma \leq \pi}}\psi_\sigma(a_1, \dots, a_n)\mu_{\BNC}(\sigma, \pi)
\]
where $\psi_\sigma(a_1, \dots, a_n) = \prod_{V \in \pi} \psi\left(\prod_{i \in V} a_i\right)$ (the product in increasing order of the elements of $V$) and $\mu_{\BNC}$ is the bi-non-crossing M\"{o}bius function on the lattice of bi-non-crossing partitions. Due to the similar lattice structures, one obtains $\mu_{\BNC}(\sigma, \pi) = \mu_{\NC}(s^{-1}_\chi \cdot \sigma, s^{-1}_\chi \cdot \pi)$. See \cite{CNS2015-1}*{Section 3} for more details.

Inspired by the characterization of free independence in terms of the vanishing of mixed free cumulants, Mastnak and Nica defined a family $\{(\A_{k, \ell}, \A_{k, r})\}_{k \in K}$ of pairs of algebras in a non-commutative probability space $(\A, \psi)$ to be \emph{combinatorially bi-free} with respect to $\psi$ if for all $n \geq 2$, $\chi: [n] \to \{\ell, r\}$, $\omega: [n] \to K$, and $a_1, \dots, a_n \in \A$ with $a_i \in \A_{\omega(i), \chi(i)}$ for $1 \leq i \leq n$,
\[
\kappa_\chi(a_1, \dots, a_n) = 0
\]
whenever $\omega$ is not constant.

It was proved by Charlesworth, Nelson, and Skoufranis in \cite{CNS2015-1} that the two notions of bi-free independence are equivalent via a diagrammatic argument. Since the diagrams constructed in \cite{CNS2015-1} will be used again later, we briefly review their results.

\begin{defn}
Given a set $K$, assign a shade (or colour) to each $k \in K$. For $n \geq 1$, $\chi: [n] \to \{\ell, r\}$, and $\omega: [n] \to K$, the set $\L\R(\chi, \omega)$ of \emph{shaded $\L\R$-diagrams} is recursively constructed as follows.
\begin{enumerate}[$\qquad(1)$]
\item For $n = 1$, $\L\R(\chi, \omega)$ consists of two parallel vertical transparent lines with a single node shaded $\omega(1)$ on the left or right depending on whether $\chi(1) = \ell$ or $\chi(1) = r$. Then either this node remains isolated or a rib and spine shaded $\omega(1)$ are drawn connecting to the top of the diagram.

\item For $n \geq 2$, let $\chi_0 = \chi|_{\{2, \dots, n\}}$ and $\omega_0 = \omega|_{\{2, \dots, n\}}$. Each diagram $D \in \L\R(\chi_0, \omega_0)$ extends to two diagrams in $\L\R(\chi, \omega)$ via the following process: Add to the top of $D$ a node shaded $\omega(1)$ on the side corresponding to $\chi(1)$ and extend all spines of $D$ to the top. If at least one spine was extended and the one nearest the new node is shaded $\omega(1)$, then connect it to the node with a rib and choose to either extend the spine to the top or not. Otherwise leave the new node isolated, or connect the new node with a rib to a new spine shaded $\omega(1)$ to the top.
\end{enumerate}
For $0 \leq t \leq n$, let $\L\R_t(\chi, \omega)$ denote the subset of $\L\R(\chi, \omega)$ with exactly $t$ spines reaching the top.
\end{defn}

\begin{defn}
Given $\chi: [n] \to \{\ell, r\}$ and $\pi, \sigma \in \BNC(\chi)$ such that $\pi \leq \sigma$, the partition $\pi$ is said to be a \emph{lateral refinement} of $\sigma$, denoted $\pi \leq_\lat \sigma$, if the bi-non-crossing diagram of $\pi$ can be obtained from that of $\sigma$ by making lateral cuts along the spines of blocks of $\sigma$ between their ribs; that is, by removing some portion of the vertical lines and then any horizontal lines that are no longer attached to a vertical line in the diagram of $\sigma$.
\end{defn}

For a bi-non-crossing partition $\pi$, let $|\pi|$ denote the number of blocks of $\pi$.  Given $\omega: [n] \to K$, we may view $\omega$ as a partition of $[n]$ with blocks $\{\omega^{-1}(\{k\})\}_{k \in K}$. Thus $\sigma \leq \omega$ denotes $\sigma$ is a refinement of the partition induced by $\omega$.

The following combinatorial result and moment type characterization were established in \cite{CNS2015-1}*{Section 4} as a crucial step in connecting bi-free independence with $(\ell, r)$-cumulants.

\begin{prop}
For $n \geq 1$, $\chi: [n] \to \{\ell, r\}$, $\omega: [n] \to K$, and $\pi \in \BNC(\chi)$ such that $\pi \leq \omega$,  
\[
\sum_{\substack{\sigma \in \L\R_0(\chi, \omega)\\\sigma \geq_\lat \pi}}(-1)^{|\pi| - |\sigma|} = \sum_{\substack{\sigma \in \BNC(\chi)\\\pi \leq \sigma \leq \omega}}\mu_{\BNC}(\pi, \sigma).
\]
\end{prop}

\begin{thm}\label{BFMomentChar}
A family $\{(\A_{k, \ell}, \A_{k, r})\}_{k \in K}$ of pairs of algebras in a non-commutative probability space $(\A, \psi)$ is bi-free with respect to $\psi$ if and only if for all $n \geq 1$, $\chi: [n] \to \{\ell, r\}$, $\omega: [n] \to K$, and $a_1, \dots, a_n \in \A$ with $a_i \in \A_{\omega(i), \chi(i)}$ for $1 \leq i \leq n$,
\[
\psi(a_1\cdots a_n) = \sum_{\pi \in \BNC(\chi)}\left[\sum_{\substack{\sigma \in \BNC(\chi)\\\pi \leq \sigma \leq \omega}}\mu_{\BNC}(\pi, \sigma)\right]\psi_\pi(a_1, \dots, a_n).
\]

Equivalently, the family $\{(\A_{k, \ell}, \A_{k, r})\}_{k \in K}$ is bi-free with respect to $\psi$ if and only if it is combinatorially bi-free with respect to $\psi$.
\end{thm}

\subsection{Conditionally free independence}\label{CfreeConst}

The notion of conditionally free independence was introduced in \cite{BLS1996}.  Given a family of unital $*$-algebras $\{\A_k\}_{k \in K}$ such that each $\A_k$ is equipped with a pair of states $(\varphi_k, \psi_k)$ and $\A_k$ is decomposes as $\A_k = \mathbb{C}1 \oplus \A_k^\circ$ with $\A_k^\circ = \ker(\psi_k)$, consider the algebraic free product $\A = *_{k \in K}\A_k$, which can be identified as a vector space with $\mathbb{C}1 \oplus \A^\circ$ where
\[
\A^\circ = \bigoplus_{n \geq 1}\left(\bigoplus_{k_1 \neq \cdots \neq k_n}\A_{k_1}^\circ \otimes \cdots \otimes \A_{k_n}^\circ\right).
\]
The \emph{conditionally free product} (or \emph{c-free product} for short) of the pairs of states $\{(\varphi_k, \psi_k)\}_{k \in K}$ is defined as $(\varphi, \psi) = *_{k \in K}(\varphi_k, \psi_k)$ where $\psi = *_{k \in K}\psi_k$ is the free product state of the states $\{\psi_k\}_{k \in K}$ and $\varphi = *_{k \in K}\{\varphi_k, \psi_k\}$ is the linear functional on $\A$ such that $\varphi(1) = 1$ and
\[
\varphi(a_1 \otimes \cdots \otimes a_n) = \varphi_{k_1}(a_1)\cdots\varphi_{k_n}(a_n)
\]
for $a_1 \otimes \cdots \otimes a_n \in \A_{k_1}^\circ \otimes \cdots \otimes \A_{k_n}^\circ$ with $k_1 \neq \cdots \neq k_n$. The triple $(\A, \varphi, \psi)$ is called the \emph{c-free product} of the family $\{(\A_k, \varphi_k, \psi_k)\}_{k \in K}$. As shown in \cite{BLS1996}*{Theorem 2.2}, the unital linear functional $\varphi$ is also a state on $\A$.  Thus $(\A, \varphi, \psi)$ is referred to as a \emph{two-state non-commutative probability space}. 

In the general case where each $\A_k$ is simply a unital algebra and $\varphi_k$, $\psi_k$ are unital linear functionals on $\A_k$, we can still construct the c-free product $(\A, \varphi, \psi) = *_{k \in K}(\A_k, \varphi_k, \psi_k)$ except now $\varphi$ and $\psi$ are just unital linear functionals on $\A$. By an abuse of terminology, any triple $(\A, \varphi, \psi)$ such that $\A$ is a unital algebra and $\varphi$, $\psi$ are unital linear functionals on $\A$ will be referred to as a two-state non-commutative probability space in the sequel. 

\begin{defn}
Let $(\A, \varphi, \psi)$ be a two-state non-commutative probability space. A family $\{\A_k\}_{k \in K}$ of unital subalgebras of $\A$ is said to be \emph{conditionally free} (or \emph{c-free} for short) with respect to $(\varphi, \psi)$ if
\begin{enumerate}[$\qquad(1)$]
\item $\psi(a_1\cdots a_n) = 0$,

\item $\varphi(a_1\cdots a_n) = \varphi(a_1)\cdots\varphi(a_n)$
\end{enumerate}
whenever $a_i \in \A_{k_i}$, $k_i \in K$, $k_1 \neq \cdots \neq k_n$, and $\psi(a_i) = 0$ for all $1 \leq i \leq n$.
\end{defn}

It is well-known that a family is c-free if and only if it can be represented using left regular representations on a conditionally free product space.  Observe also that c-free independence implies free independence and the two notions of independence coincide when $\varphi = \psi$.  Furthermore, if $\psi|_{\A^\circ_k} = 0$ for all $k$, then c-free independence with respect to $(\varphi, \psi)$ is equivalent to Boolean independence (a notion of independence introduced by Speicher and Woroudi in \cite{SW1997}) with respect to $\varphi$.

On the level of cumulants, it is clear that c-free independence implies the vanishing of mixed free cumulants, but these two assertions are not equivalent because the free cumulants, which are defined using $\psi$ only, provide no information about the moments with respect to $\varphi$. Thus for a full characterization of c-free independence another family of cumulants is required (see \cite{BLS1996}*{Section 3}). These cumulants use non-crossing partitions that are divided into two types.

\begin{defn}
Given $\pi \in \NC(n)$, a block $V$ of $\pi$ is said to be \emph{inner} if there exists another block $W$ of $\pi$ and $a, b \in W$ such that $a < v < b$ for some (hence all) $v \in V$. A block of $\pi$ is said to be \emph{outer} if it is not inner.
\end{defn}

\begin{defn}
Let $(\A, \varphi, \psi)$ be a two-state non-commutative probability space. The family of \emph{conditionally free cumulants} (or \emph{c-free cumulants} for short) with respect to $(\varphi, \psi)$ is the family of multilinear functionals
\[
\{\K_n: \A^n \to \mathbb{C}\}_{n \geq 1}
\]
uniquely determined by the requirement that
\[
\varphi(a_1\cdots a_n) = \sum_{\pi \in \NC(n)}\left(\prod_{\substack{V \in \pi\\V\,\mathrm{inner}}}\kappa_{|V|}((a_1, \dots, a_n)|_V)\right)\left(\prod_{\substack{V \in \pi\\V\,\mathrm{outer}}}\K_{|V|}((a_1, \dots, a_n)|_V)\right)
\]
for all $n \geq 1$ and $a_1, \dots, a_n \in \A$ where $\left\{\kappa_n: \A^n \to \mathbb{C}\right\}_{n \geq 1}$ denotes the family of free cumulants with respect to $\psi$.
\end{defn}

\begin{thm}[\cite{BLS1996}*{Theorem 3.1}]
A family $\{\A_k\}_{k \in K}$ of unital subalgebras in a two-state non-commutative probability space $(\A, \varphi, \psi)$ is c-free with respect to $(\varphi, \psi)$ if and only if
\[
\kappa_n(a_1, \dots, a_n) = \K_n(a_1, \dots, a_n) = 0
\]
whenever $n \geq 2$, $a_i \in \A_{k_i}$, $k_i \in K$, and there exist $i$ and $j$ such that $k_i \neq k_j$.
\end{thm}

\section{Definitions of conditionally bi-free independence}
\label{sec:c-bi-free-defns}

In this section,  the notion of conditionally bi-free independence for pairs of algebras in a two-state non-commutative probability space and conditional $(\ell, r)$-cumulants are introduced. 

\subsection{Free products of two-state vector spaces}  We being with a modification of Voiculescu's construction for bi-free independence in terms of actions on reduced free product spaces.

\begin{defn}
A \emph{two-state vector space with a specified state-vector} is a quadruple $(\X, \X^\circ, \xi, \varphi)$ where $\X$ is a vector space, $\X^\circ \subset \X$ is a subspace of co-dimension $1$, $0 \neq \xi \in \X$ is a vector such that $\X = \mathbb{C}\xi \oplus \X^\circ$ , and $\varphi: \X \to \mathbb{C}$ is a linear functional such that $\varphi(\xi) = 1$.
\end{defn}

If $(\X, \X^\circ, \xi, \varphi)$ is a two-state vector space with a specified state-vector, then the triple $(\X, \X^\circ, \xi)$ consisting of the first three elements is referred to as a \emph{vector space with a specified state-vector}. For such a triple, there is another linear functional $\psi: \X \to \mathbb{C}$ defined by $\psi(x)\xi = \mathfrak{p}(x)$ for every $x \in \X$ where $\mathfrak{p}: \X \to \X$ is the projection such that $\mathfrak{p}(\xi) = \xi$ and $\ker(\mathfrak{p}) = \X^\circ$. Note that $\psi(\xi) = 1$, $\ker(\psi) = \X^\circ$, and it is possible that $\ker(\varphi) \neq \X^\circ$. 

Given a vector space $\X$, we denote by $\L(\X)$ the algebra of linear operators on $\X$. In the case of a two-state vector space with a specified state-vector $(\X, \X^\circ, \xi, \varphi)$, there are two states $\varphi_\xi, \psi_\xi: \L(\X) \to \mathbb{C}$ defined by $\varphi_\xi(T) = \varphi(T\xi)$ and $\psi_\xi(T) = \psi(T\xi)$ for all $T \in \L(\X)$. Since $\varphi_\xi(I) = \psi_\xi(I) = 1$ for the identity operator $I$ on $\X$, the triple $(\L(\X), \varphi_\xi, \psi_\xi)$ is a two-state non-commutative probability space.

Given a family of vector spaces with specified state-vectors $\{(\X_k, \X^{\circ}_k, \xi_k)\}_{k \in K}$, the \emph{free product} $(\X, \X^\circ, \xi) = *_{k \in K}(\X_k, \X^{\circ}_k, \xi_k)$ is defined by
\[
\X = \mathbb{C}\xi \oplus \X^\circ\quad\mathrm{with}\quad\X^\circ = \bigoplus_{n \geq 1}\left(\bigoplus_{k_1 \neq \cdots \neq k_n}\X^{\circ}_{k_1} \otimes \cdots \otimes \X^{\circ}_{k_n}\right).
\]
In the case of a family of two-state vector spaces with specified state-vectors $\{(\X_k, \X^\circ_k, \xi_k, \varphi_k)\}_{k \in K}$, we define their \emph{c-free product} as in Subsection \ref{CfreeConst} by $(\X, \X^\circ, \xi, \varphi) = *_{k \in K}(\X_k, \X^\circ_k, \xi_k, \varphi_k)$ where $(\X, \X^\circ, \xi)$ is same as above and $\varphi: \X \to \bC$ is the linear functional uniquely determined by the requirement that $\varphi(\xi) = 1$ and $\varphi(x_1 \otimes \cdots \otimes x_n) = \varphi_{k_1}(x_1)\cdots\varphi_{k_n}(x_n)$ for $x_1 \otimes \cdots \otimes x_n \in \X_{k_1}^\circ \otimes \cdots \otimes \X_{k_n}^\circ$ with $k_1 \neq \cdots \neq k_n$. 

For every $k \in K$, let
\[
\X(\ell, k) = \mathbb{C}\xi \oplus \bigoplus_{n \geq 1}\left(\bigoplus_{\substack{k_1 \neq \cdots \neq k_n\\k_1 \neq k}}\X^{\circ}_{k_1} \otimes \cdots \otimes \X^{\circ}_{k_n}\right)\quad\mathrm{and}\quad\X(r, k) = \mathbb{C}\xi \oplus \bigoplus_{n \geq 1}\left(\bigoplus_{\substack{k_1 \neq \cdots \neq k_n\\k_n \neq k}}\X^{\circ}_{k_1} \otimes \cdots \otimes \X^{\circ}_{k_n}\right).
\]
There are natural identifications $V_k: \X_k \otimes \X(\ell, k) \to \X$ and $W_k: \X(r, k) \otimes \X_k \to \X$. Consequently, the algebra $\L(\X_k)$ has a left representation $\lambda_k$ and a right representation $\rho_k$ on $\L(\X)$ given by 
\[
\lambda_k(T) = V_k(T \otimes I_{\X(\ell, k)})V_k^{-1}\quad\mathrm{and}\quad\rho_k(T) = W_k(I_{\X(r, k)} \otimes T)W_k^{-1}
\]
for every $T \in \L(\X_k)$.

\subsection{Conditionally bi-free independence}

Like Voiculescu's definition of bi-free independence, conditionally bi-free independence is defined via an equality of joint distributions.

\begin{defn}
If $\Gamma = \{(\A_{k, \ell}, \A_{k, r})\}_{k \in K}$ is a family of pairs of algebras in a two-state non-commutative probability space $(\A, \varphi, \psi)$, then its \emph{joint pair-distribution} $(\mu_\Gamma, \nu_\Gamma)$ consists of the unital linear functionals
\[
\mu_\Gamma, \nu_\Gamma: *_{k \in K}(\A_{k, \ell} * \A_{k, r}) \to \mathbb{C}
\]
defined by $\mu_\Gamma = \varphi \circ \tau$ and $\nu_\Gamma = \psi \circ \tau$ where $\tau: *_{k \in K}(\A_{k, \ell} * \A_{k, r}) \to \A$ is the unital homomorphism such that $\tau|_{\A_{k, \ell}}(x) = x$ and $\tau|_{\A_{k, r}}(y) = y$ for all $k \in K$, $x \in \A_{k, \ell}$, and $y \in \A_{k, r}$.
\end{defn}

\begin{defn}
If $\hat{a} = ((a_i)_{i \in I}, (a_j)_{j \in J})$ is a two-faced family in a two-state non-commutative probability space $(\A, \varphi, \psi)$, then its \emph{pair-distribution} $(\mu_{\hat{a}}, \nu_{\hat{a}})$ consists of the unital linear functionals
\[
\mu_{\hat{a}}, \nu_{\hat{a}}: \mathbb{C}\langle X_k : k \in I \sqcup J\rangle \to \mathbb{C}
\]
defined by $\mu_{\hat{a}} = \varphi \circ \tau$ and $\nu_{\hat{a}} = \psi \circ \tau$ where $\tau: \mathbb{C}\langle X_k : k \in I \sqcup J\rangle \to \A$ is the unital homomorphism such that $\tau(X_k) = a_k$ for every $k \in I \sqcup J$.
\end{defn}

\begin{defn}\label{CBFDefn}
A family $\Gamma = \{(\A_{k, \ell}, \A_{k, r})\}_{k \in K}$ of pairs of algebras in a two-state non-commutative probability space $(\A, \varphi, \psi)$ is said to be \emph{conditionally bi-freely independent} (or \emph{c-bi-free} for short) with respect to $(\varphi, \psi)$ if there is a family of two-state vector spaces with specified state-vectors $\{(\X_k, \X^\circ_k, \xi_k, \varphi_k)\}_{k \in K}$ and unital homomorphisms
\[
\ell_k: \A_{k, \ell} \to \L(\X_k) \qand r_k: \A_{k, r} \to \L(\X_k)
\]
such that the joint pair-distribution $(\mu_\Gamma, \nu_\Gamma)$ of $\Gamma$ is equal to the joint pair-distribution of the family
\[
\widetilde{\Gamma} = \{(\lambda_k \circ \ell_k(\A_{k, \ell}), \rho_k \circ r_k(\A_{k, r}))\}_{k \in K}
\]
in $(\L(\X), \varphi_\xi, \psi_\xi)$ where $(\X, \X^\circ, \xi, \varphi) = *_{k \in K}(\X_k, \X^\circ_k, \xi_k, \varphi_k)$.
\end{defn}

A priori the definition of c-bi-free independence may not be well-defined as it must be demonstrated that the joint pair-distribution does not depend on the particular choice of $\{(\X_k, \X^\circ_k, \xi_k, \varphi_k)\}_{k \in K}$. One direct way to achieve this is to use morphisms of conditionally reduced free product spaces along the lines used to show bi-free independence is well-defined in \cite{V2014}.  Instead, the fact that Definition \ref{CBFDefn} is well-defined follows directly from Theorem \ref{thm:distributions-c-bi-free} which explicitly computes the joint pair-distributions independent of which representations are used.

\begin{exam}
For the same motivation as in \cite{BLS1996}, let $\{G_k\}_{k \in K}$ be a family of discrete groups, $G = *_{k \in K}G_k$ be their free product, and $\{\mathbb{C}[G_k]\}_{k \in K}$ and $\mathbb{C}[G]$ be the group algebras of $\{G_k\}_{k \in K}$ and $G$ respectively. Suppose each $\mathbb{C}[G_k]$ is endowed with a pair of states $(\varphi_k, \psi_k)$ such that $\varphi_k(e_k) = 1$ and $\psi_k$ is the von Neumann trace on $\mathbb{C}[G_k]$; that is,
\[
\psi_k\left(\sum_{g \in G_k}\alpha_g\cdot g\right) = \alpha_{e_k}
\]
where $e_k$ denotes the identity element of $G_k$. Let $(\varphi, \psi) = *_{k \in K}(\varphi_k, \psi_k)$ be the c-free product of $\{(\varphi_k, \psi_k)\}_{k \in K}$ on $\mathbb{C}[G]$ and define $\varphi_e, \psi_e: \L(\mathbb{C}[G]) \to \mathbb{C}$ by $\varphi_e(T) = \varphi(Te)$ and $\psi_e(T) = \psi(Te)$ for $T \in \L(\mathbb{C}[G])$ where $e$ denotes the identity element of $G$. Let further $L: \mathbb{C}[G] \to \L(\mathbb{C}[G])$ and $R: \mathbb{C}[G]^{\mathrm{op}} \to \L(\mathbb{C}[G])$ be the left and right regular representations of $\mathbb{C}[G]$ and $\mathbb{C}[G]^{\mathrm{op}}$ into $\L(\mathbb{C}[G])$ respectively, and $L_k = L|_{\mathbb{C}[G_k]}$, $R_k = R|_{\mathbb{C}[G_k]^{\mathrm{op}}}$ for $k \in K$. The family $\{(L_k(\mathbb{C}[G_k]), R_k(\mathbb{C}[G_k]^{\mathrm{op}}))\}_{k \in K}$ is c-bi-free in $(\L(\mathbb{C}[G]), \varphi_e, \psi_e)$ with respect to $(\varphi_e, \psi_e)$. Indeed, for $k \in K$, we can choose $\X_k = \mathbb{C}[G_k]$, $\X_k^\circ = \ker(\psi_k)$, $\xi_k = e_k$, and the rest of the arguments are exactly the same as those presented in \cite{V2014}*{Example 6.1}.
\end{exam}

Since Theorem \ref{thm:distributions-c-bi-free} shows the joint pair-distribution of a c-bi-free family is completely determined by the pair-distributions of the individual pairs, it is possible to define the following.

\begin{defn}
If $(\hat{a}, \hat{b}) = ((a_i)_{i \in I}, (b_j)_{j \in J})$ and $(\hat{c}, \hat{d}) = ((c_i)_{i \in I}, (d_j)_{j \in J})$ are c-bi-free two-faced families in a two-state non-commutative probability space $(\A, \varphi, \psi)$ with pair-distributions $(\mu_{(\hat{a}, \hat{b})}, \nu_{(\hat{a}, \hat{b})})$ and $(\mu_{(\hat{c}, \hat{d})}, \nu_{(\hat{c}, \hat{d})})$ respectively, then the \emph{additive}, \emph{multiplicative}, and \emph{additive-multiplicative c-bi-free convolutions} of $(\mu_{(\hat{a}, \hat{b})}, \nu_{(\hat{a}, \hat{b})})$ and $(\mu_{(\hat{c}, \hat{d})}, \nu_{(\hat{c}, \hat{d})})$ are defined to be the pair-distributions of the two-faced families
\begin{align*}
(\hat{a} + \hat{c}, \hat{b} + \hat{d}) & = ((a_i + c_i)_{i \in I}, (b_j + d_j)_{j \in J}), \\
(\hat{a}\hat{c}, \hat{b}\hat{d}) & = ((a_ic_i)_{i \in I}, (b_jd_j)_{j \in J}), \text{ and} \\
(\hat{a} + \hat{c}, \hat{b}\hat{d}) &= ((a_i + c_i)_{i \in I}, (b_jd_j)_{j \in J})
\end{align*}
respectively. These operations are denoted $\boxplus\boxplus_{\rc}$, $\boxtimes\boxtimes_{\rc}$, and $\boxplus\boxtimes_{\rc}$ respectively.
\end{defn}

The additive c-bi-free convolution will be studied extensively in Section \ref{sec:additive-limit-theorems}.

\subsection{Combinatorial c-bi-free independence and c-$(\ell, r)$-cumulants}

Due to the combinatorial structures for the c-free cumulants and $(\ell, r)$-cumulants, it is natural to hypothesize that the desired cumulants should be given by summing over bi-non-crossing partitions with a distinction on the blocks. To this end, we need the following analogues of inner and outer blocks.

\begin{defn}
Given a $\chi : [n] \to \{\ell, r\}$ and a $\pi \in \BNC(\chi)$, a block $V$ of $\pi$ is said to be \emph{interior} if there exists another block $W$ of $\pi$ such that
\[
\min_{\prec_\chi}(W) \prec_\chi \min_{\prec_\chi}(V)\quad \mathrm{and}\quad \max_{\prec_\chi}(V) \prec_\chi \max_{\prec_\chi}(W),
\]
where $\min_{\prec_\chi}$ and $\max_{\prec_\chi}$ denote the minimum and maximum elements with respect to $\prec_\chi$ respectively. A block of $\pi$ is said to be \emph{exterior} if it is not interior.
\end{defn}

Note that if $\chi$ is constant, then every bi-non-crossing partition is a non-crossing partition on $[n]$ and interior and exterior blocks are inner and outer blocks respectively. The c-$(\ell, r)$-cumulants can now be recursively defined as follows using both $\varphi$ and $\psi$.

\begin{prop}
Let $(\A, \varphi, \psi)$ be a two-state non-commutative probability space. There exists a family of multilinear functionals
\[
\left\{\K_\chi: \A^n \to \mathbb{C}\right\}_{n \geq 1, \chi: [n] \to \{\ell, r\}}
\]
uniquely determined by the requirement that
\begin{equation}\label{CBiFreeMomentCumulant}
\varphi(a_1\cdots a_n) = \sum_{\pi \in \BNC(\chi)}\left(\prod_{\substack{V \in \pi\\V\,\mathrm{interior}}}\kappa_{\chi|_V}((a_1, \dots, a_n)|_V)\right)\left(\prod_{\substack{V \in \pi\\V\,\mathrm{exterior}}}\K_{\chi|_V}((a_1, \dots, a_n)|_V)\right)
\end{equation}
for all $n \geq 1$, $\chi: [n] \to \{\ell, r\}$, and $a_1, \dots, a_n \in \A$ where $\left\{\kappa_\chi: \A^n \to \mathbb{C}\right\}_{n \geq 1, \chi: [n] \to \{\ell, r\}}$ denotes the family of $(\ell, r)$-cumulants with respect to $\psi$.
\end{prop}

\begin{proof}
For every $n \geq 1$ and $\chi: [n] \to \{\ell, r\}$, the partition $1_\chi \in \BNC(\chi)$ contains only one block, which is exterior. For $\chi_\ell: [1] \to \{\ell\}$ and $\chi_r: [1] \to \{r\}$ define $\K_{\chi_\ell} = \K_{\chi_r} = \varphi$, and recursively define
\[
\K_\chi(a_1, \dots, a_n) = \varphi(a_1\cdots a_n) - \sum_{\substack{\pi \in \BNC(\chi)\\\pi \neq 1_\chi}}\left(\prod_{\substack{V \in \pi\\V\,\mathrm{interior}}}\kappa_{\chi|_V}((a_1, \dots, a_n)|_V)\right)\left(\prod_{\substack{V \in \pi\\V\,\mathrm{exterior}}}\K_{\chi|_V}((a_1, \dots, a_n)|_V)\right)
\]
for all $n \geq 2$, $\chi: [n] \to \{\ell, r\}$, and $a_1, \dots, a_n \in \A$.
\end{proof}

\begin{defn}
The functionals from the family $\left\{\K_\chi: \A^n \to \mathbb{C}\right\}_{n \geq 1, \chi: [n] \to \{\ell, r\}}$ defined above will be referred to as the \emph{conditional $(\ell, r)$-cumulants} (or \emph{c-$(\ell, r)$-cumulants} for short) with respect to $(\varphi, \psi)$. For notational simplicity, for $\pi \in \BNC(\chi)$ define
\[
\K_{\pi}(a_1, \dots, a_n) = \left(\prod_{\substack{V \in \pi\\V\,\mathrm{interior}}}\kappa_{\chi|_V}((a_1, \dots, a_n)|_V)\right)\left(\prod_{\substack{V \in \pi\\V\,\mathrm{exterior}}}\K_{\chi|_V}((a_1, \dots, a_n)|_V)\right)
\]
so that equation \eqref{CBiFreeMomentCumulant} becomes
\[
\varphi(a_1\cdots a_n) = \sum_{\pi \in \BNC(\chi)}\K_{\pi}(a_1, \dots, a_n).
\]
\end{defn}

\begin{defn}
A family $\Gamma = \{(\A_{k, \ell}, \A_{k, r})\}_{k \in K}$ of pairs of algebras in a two-state non-commutative probability space $(\A, \varphi, \psi)$ is said to be \emph{combinatorially c-bi-free} with respect to $(\varphi, \psi)$ if for all $n \geq 2$, $\chi: [n] \to \{\ell, r\}$, $\omega: [n] \to K$, and $a_1, \dots, a_n \in \A$ with $a_i \in \A_{\omega(i), \chi(i)}$ for $1 \leq i \leq n$, we have
\[
\kappa_\chi(a_1, \dots, a_n) = \K_\chi(a_1, \dots, a_n) = 0
\]
whenever $\omega$ is not constant.
\end{defn}

\section{Equivalence of c-bi-free and combinatorial c-bi-free independence}
\label{sec:equivalence-of-defn}

The main goal of this section is to prove the following.
\begin{thm}\label{EquivCBF}
A family $\Gamma = \{(\A_{k, \ell}, \A_{k, r})\}_{k \in K}$ of pairs of algebras in a two-state non-commutative probability space $(\A, \varphi, \psi)$ is c-bi-free with respect to $(\varphi, \psi)$ if and only if it is combinatorially c-bi-free with respect to $(\varphi, \psi)$.
\end{thm}

\subsection{A moment formula for c-bi-free independence}
Our first goal is to explicitly describe the joint pair-distributions of c-bi-free families.  To do so, we note \cite{CNS2015-2} generalized bi-free probability theory to an amalgamated setting over a unital algebra $\B$ from which we will make use of the following definitions and results for the special case $\B = \bC$.

\begin{defn}
Let $n \geq 1$, $\chi: [n] \to \{\ell, r\}$, and $\omega: [n] \to K$. For $0 \leq t \leq n$, let $\L\R_t^\lat(\chi, \omega)$ denote all diagrams that can be obtained from $\L\R_t(\chi, \omega)$ under lateral refinement (i.e., cutting spines that do not reach the top). Note every diagram in $\L\R_t^\lat(\chi, \omega)$ still has $t$ spines reaching the top. For $D \in \L\R_t^\lat(\chi, \omega)$ and $D' \in \L\R_t(\chi, \omega)$, write $D' \geq_\lat D$ if $D$ can be obtained by laterally refining $D'$. Moreover, let
\[
\L\R^\lat(\chi, \omega) = \bigcup_{t = 0}^n\L\R_t^\lat(\chi, \omega).
\]
\end{defn}

\begin{defn}\label{DefnED}
Let $\{(\X_k, \X^\circ_k, \xi_k, \varphi_k)\}_{k \in K}$ be a family of two-state vector spaces with specified state-vectors, let $(\X, \X^\circ, \xi, \varphi) = *_{k \in K}(\X_k, \X^\circ_k, \xi_k, \varphi_k)$, and let $\lambda_k$ and $\rho_k$ be the left and right representations of $\L(\X_k)$ on $\L(\X)$ respectively. Moreover, let $n \geq 1$, $\chi: [n] \to \{\ell, r\}$, $\omega: [n] \to K$, and $a_i \in \L(\X_{\omega(i)})$ for $1 \leq i \leq n$. Define $\mu_i(a_i) = \lambda_{\omega(i)}(a_i)$ if $\chi(i) = \ell$ and $\mu_i(a_i) = \rho_{\omega(i)}(a_i)$ if $\chi(i) = r$.

For each $D \in \L\R^\lat(\chi, \omega)$, define $E_D(\mu_1(a_1), \dots, \mu_n(a_n))$ as follows. View $D$ as a partition of $[n]$ with blocks $V_1, \dots, V_p, W_1, \dots, W_q$ where $V_1, \dots, V_p$ are blocks with spines that do not reach the top and $W_1, \dots, W_q$ are blocks with spines that reach the top ordered from left to right. Then $E_D(\mu_1(a_1), \dots, \mu_n(a_n))$ will be a product of scalar terms with one vector term; one scalar from each $V_i$ and the vector from all of the $W_j$. For each $V_i$ write $V_i = \{s_{1, i} < s_{2, i} < \cdots < s_{r_i, i}\}$.  Then $V_i$ contributes the scalar
\[
\psi_\xi(\mu_{s_{1,i}}(a_{s_{1, i}})\cdots \mu_{s_{r_i,i}}(a_{s_{r_i, i}})) = \psi(a_{s_{1, i}}\cdots a_{s_{r_i, i}})
\]
to the product (as all elements of $V_i$ share the same colour). On the other hand, write each $W_j$ as $W_j = \{s_{1, j} < s_{2, j} < \cdots < s_{r_j, j}\}$.  Then $W_1, \dots, W_q$ contribute the vector
\[
[(1 - \mathfrak{p}_{\omega(s_{1, 1})})a_{s_{1, 1}}\cdots a_{s_{r_1, 1}}\xi_{\omega(s_{1, 1})}] \otimes \cdots \otimes [(1 - \mathfrak{p}_{\omega(s_{1, q})})a_{s_{1, q}}\cdots a_{s_{r_q, q}}\xi_{\omega(s_{1, q})}]
\]
to the product. If $q = 0$ (that is, $D \in \L\R_0^\lat(\chi, \omega)$), multiply the product by $\xi$.
\end{defn}

Under the above assumptions and notation, it was demonstrated in \cite{CNS2015-2}*{Lemma 7.1.3} that
\begin{equation}\label{FreeProdMoment}
\mu_1(a_1)\cdots\mu_n(a_n)\xi = \sum_{t = 0}^n\sum_{D \in \L\R_t^\lat(\chi, \omega)}\left[\sum_{\substack{D' \in \L\R_t(\chi, \omega)\\D' \geq_\lat D}}(-1)^{|D| - |D'|}\right]E_D(\mu_1(a_1), \dots, \mu_n(a_n)),
\end{equation}
where $|D|$ is the number of blocks in the partition induced by $D$. For later purposes, note that we can re-define $|D|$ as
\[
|D| = (\text{number of blocks of }D) + t
\]
where $t$ denotes the number of spines of $D$ that reach the top.

In \cite{CNS2015-2} only the terms with $D \in \L\R_0^\lat(\chi, \omega)$ in equation \eqref{FreeProdMoment} mattered as the focus was on $\psi$.  To obtain the necessary information for $\varphi$, further diagrams and notation will be required.

\begin{defn}
Let $n \geq 1$, $\chi: [n] \to \{\ell, r\}$, $\omega: [n] \to K$, $0 \leq t \leq n$, and $D \in \L\R^\mathrm{lat}_t(\chi, \omega)$. A diagram $D'$ is said to be a \emph{capping} of $D$, denoted $D \geq_\mathrm{cap} D'$, if $D' = D$ or $D'$ can be obtained from $D$ by removing spines from $D$ that reach the top. Let $\L\R^{\mathrm{latcap}}_m(\chi, \omega)$ denote the set of all diagrams with $m$ spines reaching the top that can be obtained by capping some $D \in \L\R^\mathrm{lat}_t(\chi, \omega)$ with $t \geq m$.
\end{defn}

Note that capping elements of $\L\R(\chi, \omega)$ need not produce elements of $\L\R(\chi, \omega)$ nor $\L\R^\lat(\chi, \omega)$. However, if we denote
\[
\L\R^{\lat\capp}(\chi, \omega) = \bigcup_{t = 0}^n\L\R^{\lat\capp}_t(\chi, \omega),
\]
then $\L\R^{\lat\capp}(\chi, \omega)$ is closed under both lateral refinement and capping. Consequently, we can extend the partial orders $\geq_{\lat}$ and $\geq_{\capp}$ to this set.
 
\begin{defn}
Let $t \geq m$, $D \in \L\R_t(\chi, \omega)$, and $D' \in \L\R^{\lat\capp}_m(\chi, \omega)$. We say that $D$ \emph{laterally caps to} $D'$, denoted $D \geq_{\lat\capp} D'$ if there exists $D'' \in \L\R^\lat_t(\chi, \omega)$ such that $D \geq_\lat D''$ and $D'' \geq_\capp D'$.
\end{defn}

Note that an alternative approach to the above definitions is to permit lateral refinements to cutting spines that reach the top.  We will use the above approach as our techniques will involve first laterally refining, and then capping.

For $D \in \L\R^{\lat\capp}_t(\chi, \omega)$ let
\[
|D| = (\text{number of blocks of }D) + t
\]
and define $E_D(\mu_1(a_1), \dots, \mu_n(a_n))$ via Definition \ref{DefnED}. Note that, unlike $E_{D'}(\mu_1(a_1), \dots, \mu_n(a_n))$ for $D' \in \L\R^{\lat}_t(\chi, \omega)$, it is not necessarily true that $E_D(\mu_1(a_1), \dots, \mu_n(a_n)) \in \X$ as such diagrams may have spines reaching the top which do not alternate in colour.  Furthermore, if $E_D(\mu_1(a_1), \dots, \mu_n(a_n)) = X_1 \otimes \cdots \otimes X_q$, let
\[
\varphi(E_D(\mu_1(a_1), \dots, \mu_n(a_n))) = \varphi(X_1)\cdots\varphi(X_q).
\]
Observe that although it is possible $X_1 \otimes \cdots \otimes X_q \notin \X^\circ$, it is still true that every $X_j$ belongs to some $\X_{k_j}^\circ$, and thus the above expression makes sense.

Finally for $D \in \L\R^{\lat\capp}(\chi, \omega)$, define $\varphi_D(a_1, \dots, a_n)$ as follows.   View $D$ as a partition of $[n]$ with blocks $V_1, \dots, V_p, W_1, \dots, W_q$ where $V_1, \dots, V_p$ are blocks with spines that do not reach the top and $W_1, \dots, W_q$ are blocks with spines that reach the top ordered from left to right. Then $\varphi_D(a_1, \dots, a_n)$ is a product of scalar terms; one from each $V_i$ and $W_j$. Each $V_i = \{s_{1, i} < s_{2, i} < \cdots < s_{r_i, i}\}$ contributes the scalar
\[
\psi_\xi(\mu_{s_{1,i}}(a_{s_{1, i}})\cdots \mu_{s_{r_i,i}}(a_{s_{r_i, i}})) = \psi(a_{s_{1, i}}\cdots a_{s_{r_i, i}})
\]
to the product and each $W_j = \{s_{1, j} < s_{2, j} < \cdots < s_{r_j, j}\}$ contributes the scalar
\[
\varphi_\xi(\mu_{s_{1,i}}(a_{s_{1, i}})\cdots \mu_{s_{r_i,i}}(a_{s_{r_i, i}})) = \varphi(a_{s_{1, j}}\cdots a_{s_{r_j, j}})
\]
(as all elements of the same block share the same colour).

In that which follows, the elements of a c-bi-free family will be identified as operators acting on the appropriate spaces via some fixed representation in Definition \ref{CBFDefn}.

\begin{lem}
\label{lem:recursive-moment-formula}
Let $\Gamma = \{(\A_{k, \ell}, \A_{k, r})\}_{k \in K}$ be a c-bi-free family of pairs of algebras in a two-state non-commutative probability space $(\A, \varphi, \psi)$. If $n \geq 1$, $\chi : [n] \to \{\ell, r\}$, $\omega : [n] \to K$, and $a_i \in \A_{\omega(i), \chi(i)}$ for all $1 \leq i \leq n$, then
\[
\varphi(a_1\cdots a_n) = \sum^n_{t = 0} \sum_{D \in \L\R^{\lat\capp}_t(\chi, \omega)}C'_D\varphi_D(a_1, \dots, a_n),
\]
where $C'_D$ is an integer-valued coefficient recursively defined as follows: For $D \in \L\R^{\lat\capp}_t(\chi, \omega)$, define
\[
C_D = \begin{cases}
\displaystyle\sum_{\substack{D' \in \L\R_t(\chi, \omega)\\D' \geq_\lat D}}(-1)^{|D| - |D'|} &\text{if } D \in \L\R^{\lat}_t(\chi, \omega)\\
0 &\text{otherwise }
\end{cases}.
\]
Recursively, starting with $t = n$, define
\[
C'_D = C_D - \sum^n_{m = t + 1}\sum_{\substack{D' \in \L\R^{\lat\capp}_m(\chi, \omega)\\D' \geq_\capp D}}C'_{D'}.
\]
\end{lem}

\begin{proof}
Note that each $C'_D$ is a well-defined integer. Since $\Gamma$ is a c-bi-free family,
\[
\varphi(a_1\cdots a_n) = \varphi_\xi(\mu_1(a_1)\cdots\mu_n(a_n)) = \varphi(\mu_1(a_1)\cdots\mu_n(a_n)\xi),
\]
where the first $\varphi$ represents the unital linear functional on $\A$ and the last $\varphi$ represents the c-free product state on $\X$. Note that $\varphi(\mu_1(a_1)\cdots\mu_n(a_n)\xi)$ is obtained by applying $\varphi$ to equation \eqref{FreeProdMoment}, thus to complete the lemma, we only show that we can correctly modify the right-hand side of equation \eqref{FreeProdMoment} after applying $\varphi$ to it.

First notice if $D \in \L\R^{\lat}_t(\chi, \omega)$ then
\[
\sum_{\substack{D' \in \L\R^{\lat\capp}(\chi, \omega)\\D' \leq_{\capp} D}} \varphi(E_{D'}(\mu_1(a_1), \dots, \mu_n(a_n))) = \varphi_D(a_1, \dots, a_n). 
\]
Indeed for each spine of $D$ that reaches the top, in half of the cappings of $D$ a factor of $\psi_\xi(\bullet)$ will appear while in the other half a factor of $\varphi((1 - \mathfrak{p})\bullet\xi) = \varphi_\xi(\bullet) - \psi_\xi(\bullet)$ will appear. Adding up these terms produces the product of all necessary $\varphi_\xi(\bullet)$ yielding $\varphi_D$.  Consequently, if $C_D\varphi(E_D(\mu_1(a_1), \dots, \mu_n(a_n)))$ occurs in the sum, then we can replace it with $C_D\varphi_D(a_1, \dots, a_n)$ provided we subtract
\[
C_D \sum_{\substack{D' \in \L\R^{\lat\capp}(\chi, \omega)\\D' \leq_{\capp} D\\D' \neq D}} \varphi(E_{D'}(\mu_1(a_1), \dots, \mu_n(a_n)))
\]
from the current expression. Note that all of the $D'$ in the above sum have fewer spines that reach the top.

To change the right-hand side of equation \eqref{FreeProdMoment} after applying $\varphi$ to the expression in this lemma, modify all of the $t$ terms starting with $t = n$ and working downwards. For $t = n$ and $D \in \L\R^{\lat\capp}_t(\chi, \omega)$ (provided such a diagram exists), the coefficient of $\varphi(E_{D}(\mu_1(a_1), \dots, \mu_n(a_n)))$ is $C_D = C'_D$.  Thus to change $\varphi(E_{D}(\mu_1(a_1), \dots, \mu_n(a_n)))$ to $\varphi_D(a_1, \ldots, a_n)$, subtract $C'_D$ from the coefficient of $\varphi(E_{D'}(\mu_1(a_1), \dots, \mu_n(a_n)))$ for all $D' \in \L\R^{\lat\capp}_m(\chi, \omega)$ with $m < t$ and $D' \leq_\capp D$.   For $t = n - 1$ and $D \in \L\R^{\lat\capp}_{t}(\chi, \omega)$  (provided such a diagram exists), the coefficient of $\varphi(E_{D}(\mu_1(a_1), \dots, \mu_n(a_n)))$ is now $C'_D$.  Thus to change $\varphi(E_{D}(\mu_1(a_1), \dots, \mu_n(a_n)))$ to $\varphi_D(a_1, \ldots, a_n)$, subtract $C'_D$ from the coefficient of $\varphi(E_{D'}(\mu_1(a_1), \dots, \mu_n(a_n)))$ for all $D' \in \L\R^{\lat\capp}_m(\chi, \omega)$ with $m < t$ and $D' \leq_\capp D$, and continue. Repeating this process yields the claimed expression by noting that if $D \in \L\R^{\lat\capp}_0(\chi, \omega)$, then $\varphi(E_{D}(\mu_1(a_1), \dots, \mu_n(a_n))) = \psi_D(a_1, \dots, a_n) = \varphi_D(a_1, \dots, a_n)$.
\end{proof}

Fortunately there is a nicer expression for $C'_D$.

\begin{lem}
Under the assumptions of Lemma \ref{lem:recursive-moment-formula}, for $D \in \L\R^{\lat \capp}_t(\chi, \omega)$ 
\[
C'_D = \sum_{\substack{ D' \in \L\R(\chi, \omega)\\D' \geq_{\lat\capp} D}}(-1)^{|D| - |D'|} = \sum^n_{r = t}\sum_{\substack{D' \in \L\R_r(\chi, \omega)\\D' \geq_{\lat\capp} D}}(-1)^{|D| - |D'|}.
\]
\end{lem}

\begin{proof}
Note the two sums in the assertion are trivially equal.  We proceed by induction on the number of spines of $D$ that reach the top, starting with $n$ spines where the result is trivial as if $D \in \L\R^{\lat \capp}_n(\chi, \omega)$ then $C'_D = C_D$.

To proceed, suppose $D \in \L\R^{\lat\capp}_t(\chi, \omega)$ and the formula holds for all $D' \in \L\R^{\lat\capp}_m(\chi, \omega)$ with $m > t$.  Then
\begin{align*}
C'_D &= C_D - \sum^n_{m = t + 1}\sum_{\substack{D' \in \L\R^{\lat\capp}_m(\chi, \omega)\\D' \geq_\capp D}}C'_{D'}\\
&= C_D - \sum^n_{m = t + 1}\sum_{\substack{D' \in \L\R^{\lat\capp}_m(\chi, \omega)\\D' \geq_\capp D}}\left[\sum^n_{q = m}\sum_{\substack{D'' \in \L\R_q(\chi, \omega)\\D'' \geq_{\lat\capp} D'}}(-1)^{|D'| - |D''|}\right]\\
&= C_D - \sum^n_{q = t + 1}\sum_{\substack{D'' \in \L\R_q(\chi, \omega)\\D'' \geq_{\lat\capp} D}}\left[\sum^q_{m = t + 1}\sum_{\substack{D' \in \L\R^{\lat\capp}_m(\chi, \omega)\\D \leq_{\capp} D' \leq_{\lat\capp} D''}}(-1)^{|D'| - |D''|}\right].
\end{align*}

Notice that the $C_D$ term in this expression gives the $r = t$ term in the assertion of the lemma since if $D' \in \L\R_t(\chi, \omega)$, then the only way that $D' \geq_{\lat\capp} D$ is if $D' \geq_{\lat} D$. Therefore if we have a fixed $D'' \in \L\R_q(\chi, \omega)$ with $q \geq t + 1$ and $D'' \geq_{\lat\capp} D$, and if we can show that
\[
\sum^q_{m = t}\sum_{\substack{D' \in \L\R^{\lat\capp}_m(\chi, \omega)\\D \leq_{\capp} D' \leq_{\lat\capp} D''}}(-1)^{|D'| - |D''|} = 0,
\]
then the proof will be complete as we can replace the sum
\[
-\sum^q_{m = t + 1}\sum_{\substack{D' \in \L\R^{\lat\capp}_m(\chi, \omega)\\D \leq_{\capp} D' \leq_{\lat\capp} D''}}(-1)^{|D'| - |D''|}
\]
with $(-1)^{|D| - |D''|}$ in the expression (i.e., the only $D' \in \L\R^{\lat\capp}_t(\chi, \omega)$ with $D \leq_{\capp} D' \leq_{\lat\capp} D''$ is $D' = D$ since $D$ is a capping of $D'$ yet $D$ and $D'$ have the same number of spines that reach the top).

Note that the desired sum is clearly zero if the sum is empty.  Hence assume the sum is not empty.  Thus there exists a $D''' \in \L\R^{\lat}_{q}(\chi, \omega)$ such that $D''' \leq_{\lat} D''$ and $D \leq_{\capp} D'''$. Then for all $D'$ such that $D \leq_{\capp} D' \leq_{\lat\capp} D''$, we must have that $D \leq_{\capp} D' \leq_{\capp} D''' \leq_{\lat} D''$ (i.e., $D$ and $D''$ determine which spines not reaching the top are cut and then the only options for $D'$ are which spines reaching the top to cap). Hence
\[
\sum^q_{m = t}\sum_{\substack{D' \in \L\R^{\lat\capp}_m(\chi, \omega)\\D \leq_{\capp} D' \leq_{\lat\capp} D''}}(-1)^{|D'| - |D''|} = (-1)^{|D'''| - |D''|}\sum^q_{m = t}\sum_{\substack{D' \in \L\R^{\lat\capp}_m(\chi, \omega)\\D \leq_{\capp} D' \leq_{\capp} D'''}}(-1)^{|D'| - |D'''|}.
\]
However, the sum on the right is clearly zero as it is the binomial expansion of 
\[
(1 + (-1))^{\text{number of spines to cap to make } D \text{ from }D'''}
\]
(i.e., the complete set of options for $D'$ is to cap or not cap each spine that reaches the top in $D'''$ but not in $D$; each spine that is capped corresponds to a $(-1)$ in the product and if one caps $b$ spines, then there are $\binom{s}{b}$ ways to do this where $s$ denotes the number of spines to cap to make $D$ from $D'''$).
\end{proof}

Combining these results, we have the following moment type characterization of c-bi-free independence.  Note that due to the nature of $D \in \L\R^{\mathrm{latcap}}(\chi, \omega)$ and $\varphi_D$, the right-hand side of equation (\ref{CBFPhiMoment}) only involves $\varphi$ applied to elements of $\mathrm{alg}(\A_{k, \ell}, \A_{k, r})$ for exactly one $k$ at a time.   Hence, as the following is proved independent of the choice representation, Definition \ref{CBFDefn} is well-defined.

\begin{thm}\label{thm:distributions-c-bi-free}
A family $\Gamma = \{(\A_{k, \ell}, \A_{k, r})\}_{k \in K}$ of pairs of algebras in a two-state non-commutative probability space $(\A, \varphi, \psi)$ is c-bi-free with respect to $(\varphi, \psi)$ if and only if for all $n \geq 1$, $\chi: [n] \to \{\ell, r\}$, $\omega: [n] \to K$, and $a_1, \dots, a_n \in \A$ with $a_i \in \A_{\omega(i), \chi(i)}$ for all $1\leq i \leq n$,
\begin{equation}\label{CBFPsiMoment}
\psi(a_1\cdots a_n) = \sum_{\pi \in \BNC(\chi)}\left[\sum_{\substack{\sigma \in \BNC(\chi, \omega)\\\pi \leq \sigma \leq \omega}}\mu_{\BNC}(\pi, \sigma)\right]\psi_\pi(a_1, \dots, a_n)
\end{equation}
and
\begin{equation}\label{CBFPhiMoment}
\varphi(a_1\cdots a_n) = \sum_{D \in \L\R^{\mathrm{latcap}}(\chi, \omega)}\left[\sum_{\substack{D' \in \L\R(\chi, \omega)\\D' \geq_{\mathrm{latcap}} D}}(-1)^{|D| - |D'|}\right]\varphi_D(a_1, \dots, a_n).
\end{equation}
\end{thm}

\begin{proof}
The fact that $\Gamma$ is bi-free with respect to $\psi$ if and only if equation \eqref{CBFPsiMoment} holds was obtained in \cite{CNS2015-1}*{Section 4}. On the other hand, if $\Gamma$ is c-bi-free with respect to $(\varphi, \psi)$, then equation \eqref{CBFPhiMoment} follows immediately from the previously two lemmata.

Conversely, suppose equations \eqref{CBFPsiMoment} and \eqref{CBFPhiMoment} hold.  Consider the universal representations of $\Gamma$; that is, for every $k \in K$, let $\X_k = \A_{k, \ell} * \A_{k, r}$, $\X_k^\circ = \ker(\psi|_{\X_k})$, $\xi_k = 1$, $\varphi_k = \varphi|_{\X_k}$, and define $\ell_k: \A_{k, \ell} \to \L(\X_k)$ and $r_k: \A_{k, r} \to \L(\X_k)$ by the left actions of $\A_{k, \ell}$ and $\A_{k, r}$ on $\X_k$ respectively.  Consequently, by the above work, the joint pair-distribution of $\{(\lambda_k \circ \ell_k(\A_{k, \ell}), \rho_k \circ r_k(\A_{k, r}))\}_{k \in K}$ satisfy equations \eqref{CBFPsiMoment} and \eqref{CBFPhiMoment} and thus agree with the joint pair-distribution of $\Gamma$.  Hence $\Gamma$ is a c-bi-free family by definition.
\end{proof}

\subsection{Equivalence with combinatorial c-bi-freeness}

Suppose $\Gamma = \{(\A_{k, \ell}, \A_{k, r})\}_{k \in K}$ is a family of pairs of algebras in a two-state non-commutative probability space $(\A, \varphi, \psi)$, $n \geq 1$, $\chi: [n] \to \{\ell, r\}$, $\omega: [n] \to K$, and $a_1, \dots, a_n \in \A$ with $a_i \in \A_{\omega(i), \chi(i)}$ for $1 \leq i \leq n$.  Using equation \eqref{CBiFreeMomentCumulant}, we obtain that
\[
\K_\chi(a_1, \ldots, a_n) = \varphi(a_1 \cdots a_n) - \sum_{\substack{\pi \in \BNC(\chi) \\ \pi \neq 1_\chi}} \K_\pi(a_1, \ldots, a_n).
\]
Since every $\kappa_{\pi|_V}((a_1, \dots, a_n)|_V)$ can be written as a sum involving products of $\psi$-moment expressions indexed by bi-non-crossing partitions with respect to $\chi|_V$, and every $\K_{\pi|_V}((a_1, \dots, a_n)|_V)$ can be written as a sum involving products of both $\psi$-moment and $\varphi$-moment expressions indexed by bi-non-crossing partitions with respect to $\chi|_V$, an expression for $\K_\chi(a_1, \dots, a_n)$ can be written (independent of the choice of $a_1, \ldots, a_n$) as a sum involving products of both $\psi$-moment and $\varphi$-moment expressions indexed by bi-non-crossing partitions with respect to $\chi$. However, for each bi-non-crossing partition $\pi \in \BNC(\chi)$ with $V$ being a block of $\pi$, it is possible that both $\psi_{\pi|_V}((a_1, \dots, a_n)|_V)$ and $\varphi_{\pi|_V}((a_1, \dots, a_n)|_V)$ appear in different products in the sum. 

In order to write the final sum in a unified way, we introduce the following notation. Let $\BNC(\chi, ie)$ denote the set of all pairs $(\pi, \iota)$ where $\pi \in \BNC(\chi)$ is a bi-non-crossing partition and $\iota: \pi \to \{i,e\}$ is a function on the blocks of $\pi$.  Then, independent of $\Gamma$ and the choice of $a_1, \ldots, a_n$, there exist integer coefficients $d(\chi; \pi, \iota)$ such that
\begin{equation}
\K_\chi(a_1, \dots, a_n) = \sum_{\substack{(\pi, \iota) \in \BNC(\chi, ie)}} d(\chi; \pi, \iota)\phi_{(\pi, \iota)}(a_1, \dots, a_n) \label{eqn:mob-inversion}
\end{equation}
where
\[
\phi_{(\pi, \iota)}(a_1, \dots, a_n) = \left(\prod_{\substack{V \in \pi\\\iota(V) = i}}\psi_{\pi|_V}((a_1, \dots, a_n)|_V)\right)\left(\prod_{\substack{V \in \pi\\\iota(V) = e}}\varphi_{\pi|_V}((a_1, \dots, a_n)|_V)\right).
\]

\begin{rem}\label{rem:pi-coming-from-D}
Notice that $\phi_{(\pi, \iota)}(a_1, \dots, a_n)$ and $\varphi_D(a_1,\dots, a_n)$ agree for certain $(\pi, \iota) \in \BNC(\chi, ie)$ and $D \in \L\R^{\lat \capp}(\chi, \omega)$.    Indeed, given $D \in \L\R^{\lat\capp}(\chi, \omega)$, defining $\pi$ via the blocks of $D$ and $\iota$ via $\iota(V) = e$ if the spine of $V$ reaches the top and $\iota(V) = i$ otherwise will produce such an equality.
\end{rem}

Note that the coefficients $d(\chi; \pi, \iota)$ play a similar role to that of the bi-non-crossing M\"{o}bius function, but less is known about their structure and properties.  However, since the expansion of the above formulae depended only on the lattice structure of $\BNC(\chi)$, we obtain that
\[
d(\chi; \pi, \iota) = d(\chi_\ell; s^{-1}_\chi \cdot \pi, \iota \circ s_\chi)
\]
where $\chi_\ell : [n] \to \{\ell\}$ is the constant map (that is, the tuple $(\chi_\ell; s^{-1}_\chi \cdot \pi, \iota \circ s_\chi)$ which corresponds to the non-crossing partition with the same selection of $\{i,e\}$ on left nodes obtained by using the $\prec_\chi$-ordering on $\pi$ must produce the same coefficient).

Consequently, if $\Gamma$ is combinatorially c-bi-free with respect to $(\varphi, \psi)$, then
\[
\varphi(a_1\cdots a_n) = \sum_{ \pi \in \BNC(\chi) }\K_\pi(a_1, \dots, a_n) = \sum_{\substack{\pi \in \BNC(\chi)\\\pi \leq \omega}}\K_\pi(a_1, \dots, a_n).
\]
Hence, by equation \eqref{eqn:mob-inversion}, we obtain that
\begin{equation}\label{CombCBF}
\varphi(a_1\cdots a_n) = \sum_{\substack{(\pi, \iota) \in \BNC(\chi, ie) \\\pi \leq \omega}}c(\chi, \omega; \pi, \iota)\phi_{(\pi, \iota)}(a_1, \dots, a_n)
\end{equation}
where $c(\chi, \omega; \pi, \iota)$ is an integer-valued coefficient.   As only the lattice structure affects the expansions of the above formulae, we obtain the following.

\begin{lem}\label{ChangeCoeff}
Let $n \geq 1$, $\chi: [n] \to \{\ell, r\}$, $\omega: [n] \to K$, and $(\pi, \iota) \in \BNC(\chi, ie)$. If $\chi_\ell: [n] \to \{\ell\}$ is the constant map, then 
\[
c(\chi, \omega; \pi, \iota) = c(\chi_\ell, \omega \circ s_\chi; s_\chi^{-1}\cdot\pi, \iota \circ s_\chi).
\]
\end{lem}

As combinatorial c-bi-free independence implies equation \eqref{CombCBF}, to show combinatorial c-bi-free independence implies c-bi-free independence, our goal is to show that equation \eqref{CombCBF} is equation \eqref{CBFPhiMoment}.    Due to Lemma \ref{ChangeCoeff}, we will follow an idea of \cite{CNS2015-1} and try to `reduce to the case that every node is on the left' via the following two operations:
\begin{defn}
\label{defn:transition-operations}
Let $\chi: [n] \to \{\ell, r\}$, $\omega: [n] \to K$, and $(\pi, \iota) \in \BNC(\chi, ie)$ be such that $\pi \leq \omega$.
\begin{enumerate}[$\qquad(1)$]
\item Suppose $\chi(n) = \ell$.  Define $\hat \chi: [n] \to \{\ell, r\}$ by
\[
\hat \chi(i) = \begin{cases}
r &\text{if } i = n\\
\chi(i) &\text{otherwise }
\end{cases},
\]
and let $\hat \pi \in \BNC(\hat \chi)$ be the unique bi-non-crossing partition with the same blocks as $\pi$ and $\hat \iota: \hat \pi \to \{i, e\}$. The operation of changing $(\pi, \iota, \omega)$ to $(\hat \pi, \hat \iota, \omega)$ is called a \emph{changing} (from left to right).

\item Suppose $\chi(i_0) = \ell$ and $\chi(i_0 + 1) = r$ for some $i_0 \in [n - 1]$.  Define $\hat \chi: [n] \to \{\ell, r\}$ and $\hat \omega: [n] \to K$ by
\[
\hat \chi(i) = \begin{cases}
r &\text{if } i = i_0\\
\ell &\text{if } i = i_0 + 1\\
\chi(i) &\text{otherwise }
\end{cases}
\qand
\hat \omega(i) = \begin{cases}
\omega(i_0 + 1) &\text{if } i = i_0\\
\omega(i_0) &\text{if } i = i_0 + 1\\
\omega(i) &\text{otherwise }
\end{cases},
\]
and let $\hat \pi \in \BNC(\hat \chi)$ be the unique bi-non-crossing partition and $\hat \iota: \hat \pi \to \{i, e\}$ be the unique function obtained by swapping $i_0$ and $i_0 + 1$. The operation of changing $(\pi, \iota, \omega)$ to $(\hat \pi, \hat \iota, \hat \omega)$ is called a \emph{swapping} (a left and a right).
\end{enumerate}
Note these same operations may be applied to elements of $\L\R^{\lat \capp}(\chi, \omega)$, but may produce diagrams outside of $\L\R^{\lat \capp}(\chi, \omega)$.
\end{defn}

To implement these operations on the $\L\R$-diagrams, we will require some terminology from \cite{CNS2015-1}.  
   
\begin{defn}
Two blocks $V$ and $W$ of the induced partition of some element $D$ from $\L\R^{\lat \capp}(\chi, \omega)$ are said to be \emph{piled} if $\max\left(\min(V), \min(W)\right) \leq \min\left(\max(V), \max(W)\right)$.  In terms of the diagram of $D$, there is some horizontal level at which both the spines of $V$ and $W$ are present.

Given blocks $V$ and $W$, a third block $U$ \emph{separates} $V$ from $W$ if it is piled with both, and its spine lies in-between the spines of $V$ and $W$. Note that $V$ and $W$ need not be piled with each other to have a separator and given any three piled blocks, one always separates the other two.

Finally, piled blocks $V$ and $W$ are said to be \emph{tangled} if there is no block which separates them.
\end{defn}

\begin{lem}
\label{lem:c-c-b-free-implies-c-b-free}
Let $\Gamma = \{(\A_{k, \ell}, \A_{k, r})\}_{k \in K}$ be a family of pairs of algebras in a two-state non-commutative probability space $(\A, \varphi, \psi)$. If $\Gamma$ is combinatorially c-bi-free with respect to $(\varphi, \psi)$, then $\Gamma$ is c-bi-free with respect to $(\varphi, \psi)$.
\end{lem}

\begin{proof}
As $\Gamma$ is combinatorially c-bi-free with respect to $(\varphi, \psi)$, $\Gamma$ is bi-free with respect to $\psi$ so equation \eqref{CBFPsiMoment} holds.  Furthermore, equation \eqref{CombCBF}  holds.  To conclude $\Gamma$ is c-bi-free with respect to $(\varphi, \psi)$, it suffices by Theorem \ref{thm:distributions-c-bi-free} to show that equation (\ref{CBFPhiMoment}) holds for any choice of $n \geq 1$, $\chi : [n] \to \{\ell, r\}$, $\omega : [n] \to K$, and $a_1, \ldots, a_n \in\A$ with $a_i \in \A_{\omega(i), \chi(i)}$ for $1 \leq i \leq n$.

Suppose first that $\chi = \chi_\ell: [n] \to \{\ell\}$ is the constant map. Since all random variables are from left algebras, the c-$(\ell, r)$-cumulants are the c-free cumulants.  Hence the vanishing of mixed c-$(\ell, r)$-cumulants implies that $\{\A_{k,\ell}\}_{k\in K}$ are c-free with respect to $(\varphi, \psi)$.  Consequently the conclusions of Theorem \ref{thm:distributions-c-bi-free} must hold for these particular $\chi$ and $\omega$ by the same arguments and thus equation (\ref{CBFPhiMoment}) holds in this setting.  Therefore, combining equations \eqref{CBFPhiMoment} and \eqref{CombCBF} produces
\[
\sum_{\substack{(\pi, \iota) \in \BNC(\chi, ie)\\\pi \leq \omega}}c(\chi, \omega; \pi, \iota)\phi_{(\pi, \iota)}(a_1, \dots, a_n) = \sum_{D \in \L\R^{\lat\capp}(\chi, \omega)}\left[\sum_{\substack{D' \in \L\R(\chi, \omega)\\D' \geq_{\lat\capp} D}}(-1)^{|D| - |D'|}\right]\varphi_D(a_1, \dots, a_n)
\]
in this setting.  However, as the above must hold for any selection of $\{\A_{k,\ell}\}_{k\in K}$ (independent of $\A_{k,r}$), it is possible using non-commutative polynomials in $n$ determinates to force at most one $\phi_{(\pi, \iota)}(a_1, \dots, a_n) $ to be non-zero at a time.  Consequently, the equality of the above sums, for the case that $\chi = \chi_\ell: [n] \to \{\ell\}$ is the constant map, implies that the only $(\pi, \iota) \in \BNC(\chi, ie)$ with $\pi \leq \omega$ and $c(\chi, \omega; \pi, \iota) \neq 0$ corresponds to some $D \in \L\R^{\lat\capp}(\chi, \omega)$ as in Remark \ref{rem:pi-coming-from-D} and, in this case, 
\begin{equation}\label{Coeff}
c(\chi, \omega; \pi, \iota) =  \sum_{\substack{D' \in \L\R(\chi, \omega)\\D' \geq_{\lat\capp} D}}(-1)^{|D| - |D'|}.
\end{equation}

To complete the proof, it suffices to verify that the previous sentence holds for arbitrary $\chi$.  Consequently, as any such tuple $(\chi, \omega; \pi, \iota)$ can be obtained using the operations of changing from left to right and swapping a left and a right, the proof will be complete provided this sentence is preserved under these operations.

As observed in Lemma \ref{ChangeCoeff}, the coefficients $c(\chi, \omega; \pi, \iota)$ are invariant under the two operations of changing and swapping. On the other hand, the property that $(\pi, \iota)$ corresponds to a $D \in \L\R^{\lat\capp}(\chi, \omega)$ and the value of the sum in equation \eqref{Coeff} are invariant under the operation of changing since lateral refinements and cappings are not affected by this operation.

For the swapping operation, suppose for a fixed $\chi$ for which there exists a $q \in [n-1]$ with $\chi(q) = \ell$ and $\chi(q + 1) = r$ that the only $(\pi, \iota) \in \BNC(\chi, ie)$ with $\pi \leq \omega$ and $c(\chi, \omega; \pi, \iota) \neq 0$ corresponds to some $D \in \L\R^{\lat\capp}(\chi, \omega)$ and equation \eqref{Coeff} holds for such $(\pi, \iota)$.  We desire to show the analogous statement for $\hat \chi$ (as in Definition \ref{defn:transition-operations}) holds.   The proof will be divided into several cases and follow along the lines of the proof of \cite{CNS2015-1}*{Lemma 4.2.4} (which has pretty pictures).

First suppose $\omega(q) \neq \omega(q+1)$.  In this case, the swapping operation is a bijection which preserves lateral refinements followed by cappings. Consequently, this swapping is a bijection from $\L\R^{\lat\capp}(\chi, \omega)$ to $\L\R^{\lat\capp}(\hat \chi, \hat \omega)$. Therefore the property that $(\pi, \iota)$ corresponds to a $D \in \L\R^{\lat\capp}(\chi, \omega)$ and the value of sum in equation \eqref{Coeff} are invariant in this case.

Suppose $D \in \L\R^{\lat\capp}(\chi, \omega)$ has the properties that $\omega(q) = \omega(q+1)$ and that $q$ and $q+1$ are in the same block of $D$.  In this case, the swapping operation is a bijection which preserves lateral refinements followed by cappings. Therefore the property that $(\pi, \iota)$ corresponds to a $D \in \L\R^{\lat\capp}(\chi, \omega)$ and the value of sum in equation \eqref{Coeff} are invariant in this case.

Suppose $D \in \L\R^{\lat\capp}(\chi, \omega)$ has the properties that $\omega(q) = \omega(q+1)$ and that $q$ and $q+1$ are not in the same block of $D$.  We require some observations about the sum in equation \eqref{Coeff} in this case.   Let $V_1$ and $V_2$ be the blocks in $D$ of $q$ and $q+1$ respectively.  Note that $V_1$ contains a left node and $V_2$ contains a right node and the sum in equation \eqref{Coeff} becomes
\[
\sum_{\substack{D' \in \L\R(\chi, \omega)\\D' \geq_{\lat\capp} D \\ q, q+1 \text{ in separated blocks of } D'}}(-1)^{|D| - |D'|}  + \sum_{\substack{D' \in \L\R(\chi, \omega)\\D' \geq_{\lat\capp} D \\ q, q+1 \text{ not in separated blocks of } D'}}(-1)^{|D| - |D'|}.
\]
We claim that
\begin{equation}\label{eqn:zero-sum-technical}
\sum_{\substack{D' \in \L\R(\chi, \omega)\\D' \geq_{\lat\capp} D \\ q, q+1 \text{ not in separated blocks of } D'}}(-1)^{|D| - |D'|}   = 0.
\end{equation}
To see this, the discussion will be divided into two cases: when $V_1$ and $V_2$ are piled and when they are not.  
 
If $V_1$ and $V_2$ are piled it is easy to see that any $D' \in \L\R(\chi, \omega)$ such that $D' \geq_{\lat\capp} D$ and $q$ and $q+1$ are not in separated blocks of $D'$ must be such that $V_1$ and $V_2$ are contained in the same block of $D'$.  This implies that $D'$ cannot produce $D$ via a lateral refinement followed by a capping as joining piled blocks cannot be undone by a lateral refinement.  Hence the sum is zero in this case.

Otherwise $V_1$ and $V_2$ are not piled.   This implies $q$ is the lowest element of $V_1$ in $D$ and $q+1$ is the highest element of $V_2$.  If $D' \in \L\R(\chi, \omega)$ is such that  $D' \geq_{\lat\capp} D$ and $q$ and $q+1$ are not in separated blocks of $D'$, then if $q$ and $q+1$ are in the same block of $D'$ it must be the case that the diagram $D''$ obtained by 	cutting the spine between $q$ and $q+1$ is an element of $\L\R(\chi, \omega)$ with $D'' \geq_{\lat\capp} D$.  Similarly if $D' \in \L\R(\chi, \omega)$ is such that  $D' \geq_{\lat\capp} D$ and $q$ and $q+1$ are not in separated blocks of $D'$, then if $q$ and $q+1$ are not in the same block of $D'$ it must be the case that the diagram $D''$ obtained by 	drawing a spine between $q$ and $q+1$ is an element of $\L\R(\chi, \omega)$ with $D'' \geq_{\lat\capp} D$.  In either case, combining the $D'$ and $D''$ terms yields $(-1)^{|D| - |D'|} + (-1)^{|D| - |D''|} = 0$.  As we may pair up diagrams in this fashion, the sum in equation \eqref{eqn:zero-sum-technical} is zero.

Let $\hat{D}$ be the diagram obtained from $D$ by swapping nodes $q$ and $q+1$.  A moment's thought shows such a diagram exists, but may not be an element of $\L\R^{\lat \capp}(\hat \chi, \hat \omega)$.  First, suppose $\hat{D}$ is an element of $\L\R^{\lat \capp}(\hat \chi, \hat \omega)$.  Then $\hat{D}$ also has the properties that  $q$ and $q+1$ are in different blocks of $\hat{D}$ and $\hat\omega(q) = \hat \omega(q+1)$.  Hence repeating the same argument above yields
\[
\sum_{\substack{\hat{D'} \in \L\R(\hat\chi, \hat\omega)\\ \hat{D'} \geq_{\lat\capp} \hat{D}}}(-1)^{|\hat{D}| - |\hat{D'}|} = \sum_{\substack{\hat{D'} \in \L\R(\hat\chi, \hat\omega)\\ \hat{D'} \geq_{\lat\capp} \hat{D} \\ q, q+1 \text{ in separated blocks of } \hat{D'}}}(-1)^{|\hat{D}| - |\hat{D'}|}.
\]
As the map taking $D' \in \L\R(\chi, \omega)$ with $q$ and $q+1$ in separated blocks of $D'$ and $D' \geq_{\lat\capp} D$ to $\hat{D'} \in \L\R(\hat\chi, \hat\omega)$ with $q$ and $q+1$ in separated blocks of $\hat{D'}$ and $\hat{D'} \geq_{\lat\capp} \hat{D}$ is a bijection, the value of the sum in equation \eqref{eqn:zero-sum-technical} is preserved in this case.

Otherwise $\hat{D}$ is not an element of $\L\R^{\lat \capp}(\hat \chi, \hat \omega)$ so there cannot exist a $D' \in \L\R(\chi, \omega)$ such that $q$ and $q+1$ are in separated blocks of $D'$ and $D' \geq_{\lat\capp} D$ for otherwise $\hat{D'}$ would be an element of $\L\R(\hat\chi, \hat \omega)$ that can be laterally refined and capped to $\hat{D}$.  Consequently, we obtain that the sum in equation \eqref{Coeff} is zero in this case.

To complete the proof, suppose $(\pi, \iota) \in \BNC(\chi, ie)$ with $\pi \leq \omega$ corresponds to some $D \in \L\R^{\lat\capp}(\chi, \omega)$ as in Remark \ref{rem:pi-coming-from-D}.  If $(\hat \pi, \hat \iota) \in \BNC(\hat \chi, ie)$ corresponds to some element of $\L\R^{\lat\capp}(\hat \chi, \hat\omega)$ (which then must be $\hat{D}$), the above work implies the sum in equation \eqref{Coeff} is preserved under the operation of swapping in this situation as desired.  If $(\hat \pi, \hat \iota) \in \BNC(\hat \chi, ie)$ does not correspond to some element of $\L\R^{\lat\capp}(\hat \chi, \hat\omega)$, then $\hat{D}$ is not an element of $\L\R^{\lat\capp}(\chi, \omega)$.  Hence it must be the case that $q$ and $q+1$ are in different blocks of $D$ and that $\omega(q) = \omega(q+1)$.  The above work demonstrates that the sum in equation \eqref{Coeff} is zero for $D$.  Therefore as we are assuming the result for $\chi$, equation \eqref{Coeff} and Lemma \ref{ChangeCoeff} yield $0 = c(\chi, \omega; \pi, \iota) = c(\hat\chi, \hat\omega; \hat\pi, \hat\iota)$, which was the desired value for $c(\hat\chi, \hat\omega; \hat\pi, \hat\iota)$.  Finally, suppose there exists a $(\pi, \iota) \in \BNC(\chi, ie)$ that does not correspond to some element of $\L\R^{\lat\capp}(\chi, \omega)$.  Hence $c(\chi, \omega; \pi, \iota) = 0$ by assumption and thus $c(\hat\chi, \hat\omega; \hat\pi, \hat\iota) = 0$ by  Lemma \ref{ChangeCoeff}.  If $(\hat\pi, \hat\iota)$ corresponds to some $\hat{D} \in \L\R^{\lat\capp}(\hat \chi, \hat\omega)$, then, by reversing the above proofs, it must be the case that $q$ and $q+1$ are in different blocks of $\hat{D}$ and that $\hat{\omega}(q) = \hat{\omega}(q+1)$ as if $D$ is the diagram obtained from $\hat{D}$ by swapping $q$ and $q+1$, then $D$ is not an element of $\L\R^{\lat \capp}(\hat \chi, \hat \omega)$ for otherwise it would correspond to $(\pi, \iota)$.  Applying a mirror to the above work then implies
\[
\sum_{\substack{\hat{D'} \in \L\R(\hat \chi, \hat \omega)\\  \hat{D'} \geq_{\lat\capp} \hat{D}}}(-1)^{|\hat{D}| - |\hat{D'}|}  = 0 = c(\hat\chi, \hat\omega; \hat\pi, \hat\iota) 
\]
as desired. 
\end{proof}

\begin{proof}[Proof of Theorem \ref{EquivCBF}]
If $\Gamma$ is combinatorially c-bi-free with respect to $(\varphi, \psi)$, then $\Gamma$ is c-bi-free with respect to $(\varphi, \psi)$ by Lemma \ref{lem:c-c-b-free-implies-c-b-free}.

Suppose $\Gamma$ is c-bi-free with respect to $(\varphi, \psi)$.  Thus equations \eqref{CBFPsiMoment} and \eqref{CBFPhiMoment} hold by Theorem \ref{thm:distributions-c-bi-free}. As shown in \cite{CNS2015-1}*{Theorem 4.3.1}, equation \eqref{CBFPsiMoment} is equivalent to the vanishing of mixed $(\ell, r)$-cumulants.  Thus we need only show that mixed c-$(\ell, r)$-cumulants vanish.

For fixed $a_1, \dots, a_n \in \A$ with $a_i \in \A_{\omega(i), \chi(i)}$ for $1 \leq i \leq n$, construct a two-state non-commutative probability space $(\A', \varphi', \psi')$, pairs of algebras $\{(\A'_{k, \ell}, \A'_{k, r})\}_{k \in K}$, and elements $a'_i \in \A'_{\omega(i), \chi(i)}$ for $1 \leq i \leq n$ such that 
\begin{itemize}
\item for each $i \in [n]$, $\{a'_j \, \mid \, \omega(j) = \omega(i), \chi(j) = \chi(i)\}$ generated $\A_{\omega(i), \chi(i)}$,
\item any joint c-$(\ell, r)$-cumulant involving $a'_1, \ldots, a'_n$ containing a pair $a'_i, a'_j$ with $\omega(i) \neq \omega(j)$ is zero, and
\item for each $i \in [n]$, the joint distribution of $\{a'_j \, \mid \, \omega(j) = \omega(i)\}$ with respect to $(\varphi', \psi')$ equals the joint distribution of $\{a_j \, \mid \, \omega(j) = \omega(i)\}$ with respect to $(\varphi, \psi)$.
\end{itemize}
The above is possible by using an algebra of non-commutative polynomials in $n$ determinates and defining $\varphi'$ and $\psi'$ using the moment-cumulant formulae. 

By the second part of the construction, the proof of Lemma \ref{lem:c-c-b-free-implies-c-b-free} implies $a'_1, \ldots, a'_n$ satisfy equations \eqref{CBFPsiMoment} and \eqref{CBFPhiMoment}.   However, since for each $i \in [n]$ the joint distribution of $\{a'_j \, \mid \, \omega(j) = \omega(i)\}$ with respect to $(\varphi', \psi')$ equals the joint distribution of $\{a_j \, \mid \, \omega(j) = \omega(i)\}$ with respect to $(\varphi, \psi)$, equations \eqref{CBFPsiMoment} and \eqref{CBFPhiMoment} imply the joint distribution of $a_1, \dots, a_n$ with respect to $(\varphi, \psi)$ equals the joint distribution of $a'_1, \dots, a'_n$ with respect to $(\varphi', \psi')$. Hence, the moment-cumulant formulae imply that $a_1, \dots, a_n$ and $a'_1, \dots, a'_n$ have the same $(\ell, r)$- and c-$(\ell, r)$-cumulants. Consequently, $\Gamma$ is combinatorially c-bi-free with respect to $(\varphi, \psi)$.
\end{proof}

\subsection{Additional properties}

There are several additional properties of the c-bi-free independence and c-$(\ell, r)$-cumulants, some of which will be used later when studying limit theorems and infinite divisibility. 

\begin{defn}[\cite{V2014}*{Proposition 2.16}]
Let $(\A, \phi)$ be a non-commutative probability space.  Two unital subalgebras $\B$ and $\C$ of $\A$ are said to be \emph{classically independent} with respect $\phi$ if $\phi(b_1c_1\cdots b_nc_n) = \phi(b_1\cdots b_n)\phi(c_1\cdots c_n)$ for all $n \geq 1$, $b_1, \dots, b_n \in \B$, and $c_1, \dots, c_n \in \C$.
\end{defn}

\begin{prop}
Let $\Gamma = \{(\A_{k, \ell}, \A_{k, r})\}_{k \in K}$ be a family of pairs of algebras in a two-state non-commutative probability space $(\A, \varphi, \psi)$. If $\Gamma$ is c-bi-free with respect to $(\varphi, \psi)$, then $\A_{k_1, \ell}$ and $\A_{k_2, r}$ are classically independent with respect to both $\varphi$ and $\psi$ for all $k_1, k_2 \in K$ such that $k_1 \neq k_2$.
\end{prop}

\begin{proof}
Let $n \geq 1$, $b_1, \dots, b_n \in \A_{k_1, \ell}$, and $c_1, \dots, c_n \in \A_{k_2, r}$ for some $k_1, k_2 \in K$ with $k_1 \neq k_2$. The fact that $\A_{k_1, \ell}$ and $\A_{k_2, r}$ are classically independent with respect to $\psi$ was shown in \cite{V2014}*{Proposition 2.16}. On the other hand, since $\Gamma$ has vanishing mixed $(\ell, r)$- and c-$(\ell, r)$-cumulants,
\[
\varphi(b_1c_1\cdots b_nc_n) = \sum_{\pi \in \BNC(\chi_{\mathrm{alt}, 2n}, \omega_{\mathrm{alt}, 2n})}\K_\pi(b_1, c_1, \dots, b_n, c_n),
\]
where $\chi_{\mathrm{alt}, 2n}: [2n] \to \{\ell, r\}$ and $\omega_{\mathrm{alt}, 2n}: [2n] \to \{k_1, k_2\}$ are such that $\chi_{\mathrm{alt}, 2n}^{-1}(\{\ell\}) = \omega_{\mathrm{alt}, 2n}^{-1}(\{k_1\}) = \{2m - 1\}_{m = 1}^n$. Since every partition $\pi \in \BNC(\chi_{\mathrm{alt}, 2n}, \omega_{\mathrm{alt}, 2n})$ has the property that if $V$ is a block of $\pi$, then either $V \subset \chi_{\mathrm{alt}, 2n}^{-1}(\{\ell\})$ or $V \subset \chi_{\mathrm{alt}, 2n}^{-1}(\{r\})$, and since any pair $(\pi_1, \pi_2)$ of non-crossing partitions on the odd and even numbers respectively produces a $\pi \in \BNC(\chi_{\mathrm{alt}, 2n}, \omega_{\mathrm{alt}, 2n})$ via $\pi = \pi_1 \cup \pi_2$, we obtain that
\[
\sum_{\pi \in \BNC(\chi_{\mathrm{alt}, 2n}, \omega_{\mathrm{alt}, 2n})}\K_\pi(b_1, c_1, \dots, b_n, c_n) = \varphi(b_1\cdots b_n)\varphi(c_1\cdots c_n) \qedhere
\]
\end{proof}

\begin{prop}
Let $(\A, \varphi, \psi)$ be a two-state non-commutative probability space. If $n \geq 2$, $a_1, \dots, a_n \in \A$, $\chi: [n] \to \{\ell, r\}$, and $a_i = 1$ for some $1 \leq i \leq n$, then
\[
\kappa_\chi(a_1, \dots, a_n) = \K_\chi(a_1, \dots, a_n) = 0.
\]
\end{prop}

\begin{proof}
The assertion that $\kappa_\chi(a_1, \dots, a_n) = 0$ is an immediate consequence of \cite{CNS2015-2}*{Proposition 6.4.1} applied to the scalar-valued setting.  The other assertion will be proved by induction. The base case $n = 2$ holds as $\K_\chi(a_1, a_2) = \varphi(a_1a_2) - \varphi(a_1)\varphi(a_2)$ for all $\chi: [2] \to \{\ell, r\}$. Assume the assertion is true for all $2 \leq m \leq n - 1$ and $\chi: [m] \to \{\ell, r\}$. If $\chi: [n] \to \{\ell, r\}$ and $a_i = 1$ for some $1 \leq i \leq n$, then 
\begin{align*}
\varphi(a_1\cdots a_{i - 1}1a_{i + 1}\cdots a_n) &= \varphi(a_1\cdots a_n)\\
&= \K_\chi(a_1, \dots, a_n) + \sum_{\substack{\pi \in \BNC(\chi)\\\pi \neq 1_\chi}}\K_\pi(a_1, \dots, a_n)\\
&= \K_\chi(a_1, \dots, a_n) + \sum_{\substack{\pi \in \BNC(\chi)\\\{i\} \in \pi}}\K_\pi(a_1, \dots, a_n)
\end{align*}
by the first assertion and the induction hypothesis. Since $\psi(a_i) = \psi(1) = 1 = \varphi(1) = \varphi(a_i)$, if $\chi' = \chi|_{[n]\,\backslash\,\{i\}}$ then
\[
\sum_{\substack{\pi \in \BNC(\chi)\\\{i\} \in \pi}}\K_\pi(a_1, \dots, a_n) = \sum_{\pi \in \BNC(\chi')}\K_\pi(a_1, \dots, a_{i - 1}, a_{i + 1}, \dots, a_n) = \varphi(a_1\cdots a_{i - 1}a_{i + 1}\cdots a_n),
\]
Hence $\K_\chi(a_1, \dots, a_n) =0$.
\end{proof}

The following demonstrate how the swapping and changing operations affect c-$(\ell, r)$-cumulants under certain settings.  The same effects for the $(\ell, r)$-cumulants were observed in \cite{S2015}*{Lemmata 2.16 and 2.17}.

\begin{lem}\label{lem:interchange-cumulant}
Let $\chi : [n] \to \{\ell, r\}$ be such that $\chi(i_0) = \ell$ and $\chi(i_0 + 1) = r$ for some $i_0 \in [n - 1]$, and let $a, b \in \A$ be such that $\varphi(cabc') = \varphi(cbac')$ and $\psi(cabc') = \psi(cbac')$ for all $c, c' \in \A$. If $\hat \chi: [n] \to \{\ell, r\}$ is as in part (2) of Definition \ref{defn:transition-operations}, then 
\[
\K_{\chi}(c_1, \dots, c_{i_0 - 1}, a, b, c_{i_0 + 2}, \dots, c_n) = \K_{\hat \chi}(c_1, \dots, c_{i_0 - 1}, b, a, c_{i_0 + 2}, \dots, c_n)
\]
for all $c_1, \dots, c_{i_0 - 1}, c_{i_0 + 2}, \dots, c_n \in \A$.
\end{lem}

\begin{proof}
By equation \eqref{eqn:mob-inversion}, the c-$(\ell, r)$-cumulants
\[
\K_{\chi}(c_1, \dots, c_{i_0 - 1}, a, b, c_{i_0 + 2}, \dots, c_n)\quad \mathrm{and}\quad \K_{\hat \chi}(c_1, \dots, c_{i_0 - 1}, b, a, c_{i_0 + 2}, \dots, c_n)
\]
can be written as expressions involving 
\[
\phi_{(\pi, \iota)}(c_1, \dots, c_{i_0 - 1}, a, b, c_{i_0 + 2}, \dots, c_n)\qand
\phi_{(\hat \pi, \hat \iota)}(c_1, \dots, c_{i_0 - 1}, b, a, c_{i_0 + 2}, \dots, c_n),
\]
respectively where the coefficients of $\phi_{(\pi, \iota)}$ and $\phi_{(\hat \pi, \hat \iota)}$ agree if $\hat \pi$ is obtained from $\pi$ by swapping $i_0$ and $i_0 + 1$ due to the lattice structure. Since there is a bijection from $\BNC(\chi)$ to $\BNC(\hat \chi)$ which sends a partition $\pi$ to the partition $\hat \pi$ obtained by swapping $i_0$ and $i_0 + 1$, and since the coefficients are the same, it suffices to show that 
\[
\phi_{(\pi, \iota)}(c_1, \dots, c_{i_0 - 1}, a, b, c_{i_0 + 2}, \dots, c_n) = \phi_{(\hat \pi, \hat \iota)}(c_1, \dots, c_{i_0 - 1}, b, a, c_{i_0 + 2}, \dots, c_n)
\]
for all $\pi \in \BNC(\chi)$. If $i_0$ and $i_0 + 1$ are in the same block of $\pi$, then, by definitions, one may reduce 
\[
\phi_{(\pi, \iota)}(c_1, \dots, c_{i_0 - 1}, a, b, c_{i_0 + 2}, \dots, c_n)
\]
to an expression involving $\varphi(cabc')$ or $\psi(cabc')$ for some $c, c' \in \A$, commute $a$ and $b$ to get $\varphi(cbac')$ (respectively, $\psi(cbac')$), and undo the reduction to obtain 
\[
\phi_{(\hat \pi, \hat \iota)}(c_1, \dots, c_{i_0 - 1}, b, a, c_{i_0 + 2}, \dots, c_n).
\]
On the other hand, if  $i_0$ and $i_0 + 1$ are in different blocks of $\pi$, then the definitions of
\[
\phi_{(\pi, \iota)}(c_1, \dots, c_{i_0 - 1}, a, b, c_{i_0 + 2}, \dots, c_n)\qand \phi_{(\hat \pi, \hat \iota)}(c_1, \dots, c_{i_0 - 1}, b, a, c_{i_0 + 2}, \dots, c_n) 
\]
agree. Consequently, the proof is complete.
\end{proof}

\begin{lem}\label{lem:swap-cumulant}
Let $\chi: [n] \to \{\ell, r\}$ be such that $\chi(n) = \ell$ and let $a, b \in \A$ be such that $\varphi(ca) = \varphi(cb)$ and $\psi(ca) = \psi(cb)$ for all $c \in \A$. If $\hat \chi: [n] \to \{\ell, r\}$ is as in part (1) of Definition \ref{defn:transition-operations}, then
\[
\K_{\chi}(c_1, \dots, c_{n - 1}, a) = \K_{\hat \chi}(c_1, \dots, c_{n - 1}, b)
\]
for all $c_1, \dots, c_{n - 1} \in \A$.
\end{lem}

\begin{proof}
By the same arguments as the previous lemma, we have
\[
\phi_{(\pi, \iota)}(c_1, \dots, c_{n - 1}, a) = \phi_{(\hat \pi, \hat \iota)}(c_1, \dots, c_{n - 1}, b)
\]
for all $\pi \in \BNC(\chi)$ where $\hat \pi \in \BNC(\hat \chi)$ is obtained from $\pi$ by changing the last node from a left node to a right node. Consequently, the proof is complete.
\end{proof}

As an immediate consequence of Lemma \ref{lem:interchange-cumulant}, we have the following result which shows that, like with the bi-free case, the family of ordered c-free cumulants of a commuting two-faced pair contains all the information about its c-$(\ell, r)$-cumulants. Consequently, when studying (pairs of) planar Borel probability measures later in Section \ref{SecLimitThm}, it is enough to know their free and c-free cumulants.

\begin{cor}
Let $(a_\ell, a_r)$ be a commuting two-faced pair in a two-state non-commutative probability space $(\A, \varphi, \psi)$. For $m, n \geq 0$ with $m + n \geq 1$ and $\chi: [m + n]  \to \{\ell, r\}$ such that $|\chi^{-1}(\{\ell\})| = m$,
\begin{align*}
\kappa_{m + n}(\underbrace{a_\ell, \dots, a_\ell}_{m\,\mathrm{times}}, \underbrace{a_r, \dots, a_r}_{n\,\mathrm{times}}) &= \kappa_\chi(a_{\chi(1)}, \dots, a_{\chi(m + n)})\text{ and} \\
\K_{m + n}(\underbrace{a_\ell, \dots, a_\ell}_{m\,\mathrm{times}}, \underbrace{a_r, \dots, a_r}_{n\,\mathrm{times}}) &= \K_\chi(a_{\chi(1)}, \dots, a_{\chi(m + n)})
\end{align*}
\end{cor}

Another consequence of Lemmata \ref{lem:interchange-cumulant} and \ref{lem:swap-cumulant} is that if one wants to check the c-bi-free independence of $\{(\A_{k, \ell}, \A_{k, r})\}_{k \in K}$, one can often enlarge $\{\A_{k, \ell}\}_{k \in K}$ and verify its c-free independence.

\begin{thm}
\label{thm:commute-bad}
If $\{(\A_{k, \ell}, \A_{k, r})\}_{k \in K}$ is a family of pairs of algebras in $(\A, \varphi, \psi)$ such that 
\begin{enumerate}[$\qquad(1)$]
\item $\A_{m, \ell}$ and $\A_{n, r}$ commute for all $m, n \in K$, and
\item for every $b \in \A_{k, r}$ there exists an $a \in A_{k, \ell}$ such that $\varphi(ca) = \varphi(cb)$ and $\psi(ca) = \psi(cb)$ for all $c \in \A$,
\end{enumerate}
then $\{(\A_{k, \ell}, \A_{k, r})\}_{k \in K}$ is c-bi-free with respect to $(\varphi, \psi)$ if and only if $\{\A_{k, \ell}\}_{k \in K}$ is c-free with respect to $(\varphi, \psi)$. Therefore, if $\{\A_{k, \ell}\}_{k \in K}$ is c-free with respect to $(\varphi, \psi)$, then $\{\A_{k, r}\}_{k \in K}$ is c-free with respect to $(\varphi, \psi)$.
\end{thm}

\begin{proof}
Clearly, the c-bi-free independence of $\{(\A_{k, \ell}, \A_{k, r})\}_{k \in K}$ implies the c-free independence of $\{\A_{k, \ell}\}_{k \in K}$ and $\{\A_{k, r}\}_{k \in K}$.

The converse amounts to the fact that the c-free independence of $\{\A_{k, \ell}\}_{k \in K}$ implies the vanishing of mixed $(\ell, r)$- and c-$(\ell, r)$-cumulants. As the analogues of the previous two lemmata also hold with $\kappa$ replacing $\K$ (see \cite{S2015}*{Lemmata 2.16 and 2.17}), one may reduce each mixed $\kappa$ or $\K$ involving elements from $\{(\A_{k, \ell}, \A_{k, r})\}_{k \in K}$ to a mixed $\kappa$ or $\K$ involving elements from $\{A_{k, \ell}\}_{k \in K}$ by the changing and swapping operations.  Since $\{\A_{k, \ell}\}_{k \in K}$ is assumed to be c-free with respect to $(\varphi, \psi)$, the proof is complete.
\end{proof}

To end this section, we analyze c-$(\ell,r)$-cumulants of products.  Given $\chi: [n] \to \{\ell, r\}$, $\pi \in \BNC(\chi)$, and $q \in [n]$, denote by $\chi|_{\backslash\,q}$ the restriction of $\chi$ to the set $[n]\,\backslash\,\{q\}$.  If $q \neq n$, define $\pi|_{q = q+1} \in \BNC(\chi|_{\backslash\,q})$ to be the bi-non-crossing partition which results from identifying $q$ and $q+1$ in $\pi$ (i.e., if $q$ and $q+1$ are in the same block as $\pi$, then $\pi|_{q=q+1}$ is obtained from $\pi$ by just removing $q$ from the block in which $q$ occurs, while if $q$ and $q+1$ are in different blocks, $\pi|_{q=q+1}$ is obtained from $\pi$ by merging the two blocks and then removing $q$).

\begin{lem}
\label{lem:multiplication-in-cumulants-weak}
Let $(\A, \varphi, \psi)$ be a two-state non-commutative probability space. If $a_1, \dots, a_n \in \A$ and $q \in [n-1]$, then
\[
\K_{\pi}(a_1, \dots, a_{q-1}, a_qa_{q+1}, a_{q+2}, \dots a_n) = \sum_{\substack{\sigma \in \BNC(\chi) \\ \sigma|_{q = q+1} = \pi}} \K_\sigma(a_1, \dots, a_n)
\]
for all $\pi \in \BNC(\chi|_{\backslash\,q})$.
\end{lem}

\begin{proof}
We proceed by induction on $n$.  If $n = 1$, there is nothing to check.
If $n = 2$, then
\[
\K_{{0_\chi}|_{1=2}}(a_1a_2) = \K_{{1_\chi}|_{1=2}}(a_1a_2) = \varphi_{{1_\chi}|_{1=2}}(a_1a_2) = \varphi_{1_\chi}(a_1, a_2) = \K_{0_\chi}(a_1, a_2) + \K_{1_\chi}(a_1, a_2)
\]
as required.

Suppose the assertion holds for $n-1$.  Using the induction hypothesis and the analogous result for $\kappa$/$\psi$ from \cite{CNS2015-2}*{Theorem 6.3.5}, we see for all $\pi \in \BNC(\chi|_{\backslash\,q})\,\backslash\,\{1_{\chi|_{\backslash\,q}}\}$ that
\[
\K_{\pi}(a_1, \dots, a_{q-1}, a_qa_{q+1}, a_{q+2}, \dots, a_n) = \sum_{\substack{\sigma \in \BNC(\chi) \\ \sigma|_{q = q+1} = \pi}} \K_\sigma(a_1, \dots, a_n).
\]
Hence
\begin{align*}
\K_{1_{\chi|_{\backslash\,q}}}&(a_1, \dots, a_{q-1}, a_qa_{q+1}, a_{q+2}, \dots, a_n) \\
&= \varphi_{1_{\chi|_{\backslash\,q}}}(a_1, \dots, a_{q-1}, a_qa_{q+1}, a_{q+2}, \dots, a_n) - \sum_{\substack{\pi \in \BNC(\chi|_{\backslash\,q}) \\  \pi \neq 1_{\chi|_{\backslash\,q}}}}   \K_{\pi}(a_1, \dots, a_{q-1}, a_qa_{q+1}, a_{q+2}, \dots, a_n) \\
&= \varphi_{1_{\chi}}(a_1, \dots, a_n) - \sum_{\substack{\pi \in \BNC(\chi|_{\backslash\,q}) \\  \pi \neq 1_{\chi|_{\backslash\,q}}}}\sum_{\substack{\sigma \in \BNC(\chi) \\ \sigma|_{q = q+1} = \pi}} \K_\sigma(a_1, \dots, a_n)  \\
&= \sum_{\sigma \in \BNC(\chi)}\K_{\sigma}(a_1, \dots, a_n) - \sum_{\substack{\sigma \in \BNC(\chi) \\ \sigma|_{q = q+1} \neq 1_{\chi|_{\backslash\,q}}}}\K_\sigma(a_1, \dots, a_n)\\
&= \sum_{\substack{\sigma \in \BNC(\chi) \\ \sigma|_{q = q+1} = 1_{\chi|_{\backslash\,q}}}}\K_\sigma(a_1, \dots, a_n). \qedhere
\end{align*}
\end{proof}

By recursively applying Lemma \ref{lem:multiplication-in-cumulants-weak}, it is possible to obtain a stronger result.  Given two partitions $\pi, \sigma \in \BNC(\chi)$, let $\pi \vee \sigma$ denote the smallest element of $\BNC(\chi)$ greater than $\pi$ and $\sigma$. Furthermore, suppose $m, n \in \mathbb{N}$ with $m < n$ are fixed, and consider a sequence of integers 
\[
k(0) = 0 < k(1) < \cdots < k(m) = n.
\]
For $\chi: [m] \to \{\ell, r\}$, define $\hat{\chi}: [n] \to \{\ell, r\}$ via
\[
\hat{\chi}(q) = \chi(p_q)
\]
where $p_q$ is the unique element of $[m]$ such that $k(p_q-1) < q \leq k(p_q)$. Let $\widehat{0_\chi}$ be the partition of $[n]$ with blocks $\{(k(p-1)+1, \dots, k(p))  \}_{p=1}^m$.  Recursively applying Lemma \ref{lem:multiplication-in-cumulants-weak} yields the following.

\begin{thm}
Let $(\A, \varphi,\psi)$ be a two-state non-commutative probability space. With the above notation
\[
\K_{1_\chi}\left(a_1\cdots a_{k(1)}, a_{k(1)+1}\cdots a_{k(2)}, \dots, a_{k(m-1)+1}\cdots a_{k(m)}\right) = \sum_{\substack{\sigma \in \BNC(\hat{\chi})\\ \sigma \vee \widehat{0_\chi} = 1_{\hat{\chi}}}}\K_\sigma(a_1, \dots, a_n).
\]
\end{thm}

\section{The conditionally bi-free partial $\R$-transform}
\label{sec:R-transform}

Let $(a_\ell, a_r)$ be a two-faced pair in a two-state non-commutative probability space $(\A, \varphi, \psi)$. The goal of this section is to define the c-bi-free partial $\R$-transform $\R^{\rc}_{(a_\ell, a_r)}$ of $(a_\ell, a_r)$ and derive a functional equation involving $\R^{\rc}_{(a_\ell, a_r)}$ via combinatorics such that the two-variable Cauchy transform of $(a_\ell, a_r)$ with respect to $\varphi$ can be obtained from the said functional equation.

\subsection{Single-variable transforms}

We begin by recalling some notation and single-variable results.

\begin{defn}
Let $a$ be a random variable in a two-state non-commutative probability space $(\A, \varphi, \psi)$. For $m \geq 1$, let $\kappa_m(a)$ and $\K_m(a)$ denote the $m^{\mathrm{th}}$ free and c-free cumulants of $a$ respectively; that is, in the notation of $(\ell, r)$- and c-$(\ell, r)$-cumulants, $\kappa_m(a) = \kappa_\chi(a, \dots, a)$ and $\K_m(a) = \K_\chi(a, \dots, a)$ where $\chi: [m] \to \{\ell, r\}$ is constant. 
\begin{enumerate}[$\qquad(1)$]
\item The \emph{$\psi$-moment  and $\varphi$-moment series} of $a$ are respectively
\[
M_a^\psi(z) = 1 + \sum_{m \geq 1}\psi(a^m)z^m\qand M_a^\varphi(z) = 1 + \sum_{m \geq 1}\varphi(a^m)z^m.
\]

\item The \emph{free and c-free $\R$-transforms} of $a$ are respectively
\[
\R_a(z) = \sum_{m \geq 0}\kappa_{m + 1}(a)z^m \qand \R^{\rc}_a(z) = \sum_{m \geq 0}\K_{m + 1}(a)z^m.
\]

\item The \emph{free and c-free cumulant series} of $a$ are respectively
\[
C_a(z) = 1 + \sum_{m \geq 1}\kappa_m(a)z^m = 1 + z\R_a(z) \qand C^{\rc}_a(z) = 1 + \sum_{m \geq 1}\K_m(a)z^m = 1 + z\R^{\rc}_a(z).
\]
\end{enumerate}
\end{defn}

Note that if $a_1$ and $a_2$ are c-free with respect to $(\varphi, \psi)$, then
\[
\R_{a_1 + a_2}(z) = \R_{a_1}(z) + \R_{a_2}(z)\qand \R^{\rc}_{a_1 + a_2}(z) = \R^{\rc}_{a_1}(z) + \R^{\rc}_{a_2}(z).
\]
Moreover, the following relations are well-known (see \cite{S1994}):
\[
C_a(zM_a^\psi(z)) = M_a^\psi(z)\qand M_a^\psi\left(\frac{z}{C_a(z)}\right) = C_a(z).
\]

We also have the following additional relations.

\begin{lem}\label{CFreeSeries}
Let $a$ be a random variable in a two-state non-commutative probability space $(\A, \varphi, \psi)$. Then
\[
C_a^\rc(zM_a^\psi(z)) = 1 + M_a^\psi(z) - \frac{M_a^\psi(z)}{M_a^\varphi(z)}\qand M_a^\varphi\left(\frac{z}{C_a(z)}\right) = \frac{C_a(z)}{1 + C_a(z) - C_a^\rc(z)}.
\]
\end{lem}

\begin{proof}
The first equation is part of \cite{BLS1996}*{Theorem 5.1} under different notation. Replacing $z$ with $\frac{z}{C_a(z)}$ in the first equation produces
\[
C_a^\rc(z) = 1 + C_a(z) - \frac{C_a(z)}{M_a^\varphi\left(\frac{z}{C_a(z)}\right)},
\]
which is the second equation.
\end{proof}

\subsection{Two-variable transforms}

Note that all series below are in two commuting variables.

\begin{defn}
Let $(a_\ell, a_r)$ be a two-faced pair in a two-state non-commutative probability space $(\A, \varphi, \psi)$. The \emph{bi-free} and \emph{c-bi-free partial $\R$-transforms} of $(a_\ell, a_r)$ are defined respectively as
\[
\R_{(a_\ell, a_r)}(z, w) = \sum_{\substack{m, n \geq 0\\m + n \geq 1}}\kappa_{m, n}(a_\ell, a_r)z^mw^n \qand \R^{\rc}_{(a_\ell, a_r)}(z, w) = \sum_{\substack{m, n \geq 0\\m + n \geq 1}}\K_{m, n}(a_\ell, a_r)z^mw^n,
\]
where $\kappa_{m, n}(a_\ell, a_r)$ and $\K_{m, n}(a_\ell, a_r)$ denote the $(\ell, r)$- and c-$(\ell, r)$-cumulants $\kappa_{\chi_{m, n}}(a_{\chi_{m, n}(1)}, \dots, a_{\chi_{m, n}(m + n)})$ and $\K_{\chi_{m, n}}(a_{\chi_{m, n}(1)}, \dots, a_{\chi_{m, n}(m + n)})$ respectively with $\chi_{m, n}: [m + n] \to \{\ell, r\}$ and $\chi_{m, n}^{-1}(\{\ell\}) = [m]$.
\end{defn}

The c-bi-free partial $\R$-transform plays a similar role as the c-free $\R$-transform when it comes to the additive c-bi-free convolution. Indeed, if $(a_{1, \ell}, a_{1, r})$ and $(a_{2, \ell}, a_{2, r})$ are c-bi-free with respect to $(\varphi, \psi)$, then
\[
\R_{(a_{1, \ell} + a_{2, \ell}, a_{1, r} + a_{2, r})}(z, w) = \R_{(a_{1, \ell}, a_{1, r})}(z, w) + \R_{(a_{2, \ell}, a_{2, r})}(z, w)
\]
and
\[
\R^{\rc}_{(a_{1, \ell} + a_{2, \ell}, a_{1, r} + a_{2, r})}(z, w) = \R^{\rc}_{(a_{1, \ell}, a_{1, r})}(z, w) + \R^{\rc}_{(a_{2, \ell}, a_{2, r})}(z, w).
\]

\begin{defn}
Let $(a_\ell, a_r)$ be a two-faced pair in a two-state non-commutative probability space $(\A, \varphi, \psi)$.  
\begin{enumerate}[$\qquad(1)$]
\item The \emph{left-then-right $\psi$-moment and $\varphi$-moment series} of $(a_\ell, a_r)$ are respectively
\[
M^{\psi}_{(a_\ell, a_r)}(z, w) = 1 + \sum_{\substack{m, n \geq 0\\m + n \geq 1}}\psi(a_{\ell}^ma_{r}^n)z^mw^n\qand M^{\varphi}_{(a_\ell, a_r)}(z, w) = 1 + \sum_{\substack{m, n \geq 0\\m + n \geq 1}}\varphi(a_{\ell}^ma_{r}^n)z^mw^n.
\]

\item The \emph{two-variable $\psi$-Cauchy and $\varphi$-Cauchy transforms} of $(a_\ell, a_r)$ are respectively
\[
G_{(a_\ell, a_r)}^\psi(z, w) = \psi((z - a_\ell)^{-1}(w - a_r)^{-1}) = \frac{1}{zw} + \sum_{\substack{m, n \geq 0\\m + n \geq 1}}\psi(a_\ell^ma_r^n)\frac{1}{z^{m + 1}}\frac{1}{w^{n + 1}}
\]
and
\[
G_{(a_\ell, a_r)}^\varphi(z, w) = \varphi((z - a_\ell)^{-1}(w - a_r)^{-1}) = \frac{1}{zw} + \sum_{\substack{m, n \geq 0\\m + n \geq 1}}\varphi(a_\ell^ma_r^n)\frac{1}{z^{m + 1}}\frac{1}{w^{n + 1}}.
\]

\item The \emph{left-then-right c-$(\ell, r)$-cumulant series} of $(a_\ell, a_r)$ is
\[
C^{\rc}_{(a_\ell, a_r)}(z, w) = 1 + \sum_{\substack{m, n \geq 0\\m + n \geq 1}}\K_{m, n}(a_\ell, a_r)z^mw^n = 1 + \R^{\rc}_{(a_\ell, a_r)}(z, w).
\]
\end{enumerate}
\end{defn}

It is easy to verify that
\[
G_{(a_\ell, a_r)}^\psi(z,w) = \frac{1}{zw}M_{(a_\ell, a_r)}^\psi\left(\frac{1}{z}, \frac{1}{w}\right)\qand G_{(a_\ell, a_r)}^\varphi(z,w) = \frac{1}{zw}M_{(a_\ell, a_r)}^\varphi\left(\frac{1}{z}, \frac{1}{w}\right).
\]

We now derive a functional equation for $\R^{\rc}_{(a_\ell, a_r)}$ analogous to the functional equation for the bi-free partial $\R$-transform $\R_{(a_\ell, a_r)}$ of $(a_\ell, a_r)$ given in \cite{V2016}*{Theorem 2.4}. Our proof is in the spirit of the combinatorial proof given in \cite{S2016}*{Theorem 7.2.4}. First, recall the following definition from \cite{S2016}*{Definition 3.2.1}.

\begin{defn}
Given $n \geq 1$ and $\chi: [n] \to \{\ell, r\}$, a bi-non-crossing partition $\pi \in \BNC(\chi)$ is said to be \emph{vertically split} if whenever $V$ is a block of $\pi$, either $V \subset \chi^{-1}(\{\ell\})$ or $V \subset \chi^{-1}(\{r\})$. The set of vertically split bi-non-crossing partitions in $\BNC(\chi)$ is denoted by $\BNC_{\mathrm{vs}}(\chi)$.
\end{defn}

\begin{thm}
Let $(a_\ell, a_r)$ be a two-faced pair in a two-state non-commutative probability space $(\A, \varphi, \psi)$. Then
\begin{align}
C_{(a_\ell, a_r)}^{\rc} & (zM_{a_\ell}^\psi(z), wM_{a_r}^\psi(w)) +1 \nonumber  \\  
 &= M_{a_\ell}^\psi(z) + M_{a_r}^\psi(w)  + \left(1 - \frac{M_{a_\ell}^\psi(z)}{M_{a_\ell}^\varphi(z)}\right) + \left(1- \frac{M_{a_r}^\psi(w)}{M_{a_r}^\varphi(w)}\right) - \frac{M_{a_\ell}^\psi(z)M_{a_r}^\psi(w)}{M_{(a_\ell, a_r)}^{\psi}(z, w)}\left(1 - \frac{M_{(a_\ell, a_r)}^{\varphi}(z, w)}{M_{a_\ell}^\varphi(z)M_{a_r}^\varphi(w)}\right). \label{CBifreeCumulantSeries}
\end{align}
\end{thm}

\begin{proof}
For $m, n \geq 0$ with $m + n \geq 1$, notice
\begin{align*}
\varphi(a_\ell^ma_r^n) &= \sum_{\pi \in \BNC_{\mathrm{vs}}(\chi_{m, n})}\K_\pi(\underbrace{a_\ell, \dots, a_\ell}_{m\,\mathrm{times}}, \underbrace{a_r, \dots, a_r}_{n\,\mathrm{times}}) + \sum_{\substack{\pi \in \BNC(\chi_{m, n})\\\pi \notin \BNC_{\mathrm{vs}}(\chi_{m, n})}}\K_\pi(\underbrace{a_\ell, \dots, a_\ell}_{m\,\mathrm{times}}, \underbrace{a_r, \dots, a_r}_{n\,\mathrm{times}})\\
&= \varphi(a_\ell^m)\varphi(a_r^n) + \Theta_{m, n},
\end{align*}
where $\Theta_{m, n}$ denotes the sum
\[
\sum_{\substack{\pi \in \BNC(\chi_{m, n})\\\pi \notin \BNC_{\mathrm{vs}}(\chi_{m, n})}}\K_\pi(\underbrace{a_\ell, \dots, a_\ell}_{m\,\mathrm{times}}, \underbrace{a_r, \dots, a_r}_{n\,\mathrm{times}}).
\]
Note that $\Theta_{m, 0} = \Theta_{0, n} = 0$ for all $m, n \geq 1$. For every partition $\pi \in \BNC(\chi_{m, n})\,\backslash\,\BNC_{\mathrm{vs}}(\chi_{m, n})$, let $V_\pi$ denote the block of $\pi$ with both left and right indices such that $\min(V_\pi)$ is the smallest among all blocks of $\pi$ with this property. Note that $V_\pi$ is necessarily an exterior block.

For every block $V$ of $\pi$, let $V_\ell = V \cap [m]$ and $V_r = V \cap [m + n]\,\backslash\,[m]$. Rearrange the sum in $\Theta_{m, n}$ by first choosing $s \in [m]$, $t \in [n]$, and $V \subset [m + n]$ such that
\[
V_\ell = \{u_1 < \cdots < u_s\}\qand V_r = \{v_1 < \cdots < v_t\},
\]
and then summing over all $\pi \in \BNC(\chi_{m, n})\,\backslash\,\BNC_{\mathrm{vs}}(\chi_{m, n})$ such that $V_\pi = V$.  The result is
\begin{align*}
&\Theta_{m, n} \\
&= \sum_{s = 1}^m\sum_{t = 1}^n\sum_{\substack{V\\V_\ell = \{u_1 < \cdots < u_s\}\\V_r = \{v_1 < \cdots < v_t\}}}\K_{s, t}(a_\ell, a_r)\psi(a_\ell^{m - u_s}a_r^{n - v_t})\varphi(a_\ell^{u_1 - 1})\varphi(a_r^{v_1 - 1})\prod_{k = 2}^{s}\psi(a_\ell^{u_{k} - u_{k - 1} - 1})\prod_{k = 2}^{t}\psi(a_r^{v_{k} - v_{k - 1} - 1})\\
&= \sum_{\substack{1 \leq s \leq m\\0 \leq i_0, i_1, \dots, i_s \leq m\\i_0 + i_1 + \cdots + i_s = m - s}}\sum_{\substack{1 \leq t \leq n\\0 \leq j_0, j_1, \dots, j_t \leq n\\j_0 + j_1 + \cdots + j_t = n - t}}\K_{s, t}(a_\ell, a_r)\psi(a_\ell^{i_0}a_r^{j_0})\varphi(a_\ell^{i_1})\varphi(a_r^{j_1})\prod_{k = 2}^{s}\psi(a_\ell^{i_k})\prod_{k = 2}^{t}\psi(a_r^{j_k}).
\end{align*}
Using the above equation and the first equation from Lemma \ref{CFreeSeries}, we obtain
\begin{align*}
\sum_{m, n \geq 1}\Theta_{m, n}z^mw^n &= \sum_{s, t \geq 1}\K_{s, t}(a_\ell, a_r)M_{(a_\ell, a_r)}^{\psi}(z, w)M_{a_\ell}^\varphi(z)M_{a_r}^\varphi(w)(M_{a_\ell}^\psi(z))^{s - 1}(M_{a_r}^\psi(w))^{t - 1}z^sw^t\\
&= \left(\sum_{s, t \geq 1}\K_{s, t}(a_\ell, a_r)(zM_{a_\ell}^\psi(z))^s(wM_{a_r}^\psi(w))^t\right)\frac{M_{a_\ell}^\varphi(z)M_{a_r}^\varphi(w)M_{(a_\ell, a_r)}^{\psi}(z, w)}{M_{a_\ell}^\psi(z)M_{a_r}^\psi(w)}\\
&= \left(C_{(a_\ell, a_r)}^\rc(zM_{a_\ell}^\psi(z), wM_{a_r}^\psi(w)) - C_{a_\ell}^\rc(zM_{a_\ell}^\psi(z)) - C_{a_r}^\rc(wM_{a_r}^\psi(w)) + 1\right)\\
& \quad \times \frac{M_{a_\ell}^\varphi(z)M_{a_r}^\varphi(w)M_{(a_\ell, a_r)}^{\psi}(z, w)}{M_{a_\ell}^\psi(z)M_{a_r}^\psi(w)}\\
&= \left(C_{(a_\ell, a_r)}^\rc(zM_{a_\ell}^\psi(z), wM_{a_r}^\psi(w)) - M_{a_\ell}^\psi(z) + \frac{M_{a_\ell}^\psi(z)}{M_{a_\ell}^\varphi(z)} - M_{a_r}^\psi(w) + \frac{M_{a_r}^\psi(w)}{M_{a_r}^\varphi(w)} - 1\right)\\
& \quad \times \frac{M_{a_\ell}^\varphi(z)M_{a_r}^\varphi(w)M_{(a_\ell, a_r)}^{\psi}(z, w)}{M_{a_\ell}^\psi(z)M_{a_r}^\psi(w)}.
\end{align*}
On the other hand, we also have
\begin{align*}
M_{(a_\ell, a_r)}^{\varphi}(z, w) &= 1 + \sum_{m \geq 1}\varphi(a_\ell^m)z^m + \sum_{n \geq 1}\varphi(a_r^n)w^n + \sum_{m, n \geq 1}\varphi(a_\ell^ma_r^n)z^mw^n\\
&= 1 + \sum_{m \geq 1}\varphi(a_\ell^m)z^m + \sum_{n \geq 1}\varphi(a_r^n)w^n + \sum_{m, n \geq 1}\varphi(a_\ell^m)\varphi(a_r^n)z^mw^n + \sum_{m, n \geq 1}\Theta_{m, n}z^mw^n\\
&= M_{a_\ell}^{\varphi}(z)M_{a_r}^{\varphi}(w) + \sum_{m, n \geq 1}\Theta_{m, n}z^mw^n.
\end{align*}
Combining these equations, we have
\begin{align*}
&C_{(a_\ell, a_r)}^\rc(zM_{a_\ell}^\psi(z), wM_{a_r}^\psi(w)) - M_{a_\ell}^\psi(z) + \frac{M_{a_\ell}^\psi(z)}{M_{a_\ell}^\varphi(z)} - M_{a_r}^\psi(w) + \frac{M_{a_r}^\psi(w)}{M_{a_r}^\varphi(w)} - 1\\
&= - \frac{M_{a_\ell}^\psi(z)M_{a_r}^\psi(w)}{M_{(a_\ell, a_r)}^{\psi}(z, w)}\left(1 - \frac{M_{(a_\ell, a_r)}^{\varphi}(z, w)}{M_{a_\ell}^\varphi(z)M_{a_r}^\varphi(w)}\right),
\end{align*}
which is the desired formula.
\end{proof}

\begin{cor}\label{CBFR-trans}
Let $(a_\ell, a_r)$ be a two-faced pair in a two-state non-commutative probability space $(\A, \varphi, \psi)$. The c-bi-free partial $\R$-transform $\R^{\rc}_{(a_\ell, a_r)}$ of $(a_\ell, a_r)$ is given by
\[
\R^{\rc}_{(a_\ell, a_r)}(z, w) = z\R_{a_\ell}^{\rc}(z) + w\R_{a_r}^{\rc}(w) + \widetilde{\R}^{\rc}_{(a_\ell, a_r)}(z, w),
\]
where
\begin{align*}
\widetilde{\R}^{\rc}_{(a_\ell, a_r)}& (z, w) \\
&= \frac{zw\left(-1 + (\R_{a_\ell}(z) + \frac{1}{z} - \R_{a_\ell}^{\rc}(z))(\R_{a_r}(w) + \frac{1}{w} - \R_{a_r}^{\rc}(w))G_{(a_\ell, a_r)}^\varphi\left(\R_{a_\ell}(z) + \frac{1}{z}, \R_{a_r}(w) + \frac{1}{w}\right)\right)}{ G_{(a_\ell, a_r)}^\psi\left(\R_{a_\ell}(z) + \frac{1}{z}, \R_{a_r}(w) + \frac{1}{w}\right)}.
\end{align*}
\end{cor}

\begin{proof}
Replacing $z$ and $w$ with $\frac{z}{C_{a_\ell}(z)}$ and $\frac{w}{C_{a_r}(w)}$ respectively, in equation (\ref{CBifreeCumulantSeries}) and using the second equation from Lemma \ref{CFreeSeries}, we obtain that
\begin{align*}
C_{(a_\ell, a_r)}^\rc(z, w) &= -1 + C_{a_\ell}^\rc(z) + C_{a_r}^\rc(w) - \frac{C_{a_\ell}(z)C_{a_r}(w)}{M_{(a_\ell, a_r)}^{\psi}\left(\frac{z}{C_{a_\ell}(z)}, \frac{w}{C_{a_r}(w)}\right)}\left (1 - \frac{M_{(a_\ell, a_r)}^{\varphi}\left(\frac{z}{C_{a_\ell}(z)}, \frac{w}{C_{a_r}(w)}\right)}{\frac{C_{a_\ell}(z)}{1 + C_{a_\ell}(z) - C_{a_\ell}^\rc(z)}\frac{C_{a_r}(w)}{1 + C_{a_r}(w) - C_{a_r}^\rc(w)}}     \right)       \\
&= -1 + C_{a_\ell}^\rc(z) + C_{a_r}^\rc(w) - \frac{zw}{G_{(a_\ell, a_r)}^{\psi}\left(\frac{C_{a_\ell}(z)}{z}, \frac{C_{a_r}(w)}{w}\right)}\left(1-  \frac{G_{(a_\ell, a_r)}^{\varphi}\left(\frac{C_{a_\ell}(z)}{z}, \frac{C_{a_r}(w)}{w}\right)}{\frac{z}{1 + C_{a_\ell}(z) - C_{a_\ell}^\rc(z)}\frac{w}{1 + C_{a_r}(w) - C_{a_r}^\rc(w)}} \right).
\end{align*}
Using the equations $C_{a_\ell}(z) = 1 + z\R_{a_\ell}(z)$, $C_{a_\ell}^\rc(z) = 1 + z\R_{a_\ell}^\rc(z)$, $C_{a_r}(w) = 1 + w\R_{a_r}(w)$, $C_{a_r}^\rc(w) = 1 + w\R_{a_r}^\rc(w)$, and subtracting $1$ from both sides of the above equation, the result follows.
\end{proof}

In view of Corollary \ref{CBFR-trans}, we define the c-bi-free partial $\R$-transforms of pairs of Borel probability measures on $\bR^2$ as follows. Note that for a Borel probability measure $\sigma$ on $\bR^2$, let $\sigma^{(1)}$ and $\sigma^{(2)}$ be the marginal distributions defined by
\[
\sigma^{(1)}(B) = \sigma(B \times \bR)\qand \sigma^{(2)}(B) = \sigma(\bR \times B)
\]
for all Borel sets $B$ on $\bR$.

\begin{defn}\label{CBF-R}
Let $(\mu, \nu)$ be a pair of Borel probability measures on $\mathbb{R}^2$.  The \emph{c-bi-free partial $\R$-transform} of $(\mu, \nu)$ is defined for $(z, w) \neq (0, 0)$ in some neighbourhood of $(0, 0)$ by
\[
\R_{(\mu, \nu)}(z, w) = z\R_{(\mu^{(1)}, \nu^{(1)})}(z) + w\R_{(\mu^{(2)}, \nu^{(2)})}(w) + \widetilde{R}_{(\mu, \nu)}(z, w),
\]
 where $\R_{(\mu^{(1)}, \nu^{(1)})}$ and $\R_{(\mu^{(2)}, \nu^{(2)})}$ are the c-free $\R$-transforms of the marginal pairs $(\mu^{(1)}, \nu^{(1)})$ and $(\mu^{(2)}, \nu^{(2)})$ respectively, and $\widetilde{R}_{(\mu, \nu)}(z, w)$ equals
\[
\frac{zw\left( -1 +    (\R_{\nu^{(1)}}(z) + \frac{1}{z} - \R_{(\mu^{(1)}, \nu^{(1)})}(z))(\R_{\nu^{(2)}}(w) + \frac{1}{w} - \R_{(\mu^{(2)}, \nu^{(2)})}(w))G_\mu(\R_{\nu^{(1)}}(z) + \frac{1}{z}, \R_{\nu^{(2)}}(w) + \frac{1}{w})\right)}{G_\nu(\R_{\nu^{(1)}}(z) + \frac{1}{z}, \R_{\nu^{(2)}}(w) + \frac{1}{w})}.
\]
The function $\widetilde{R}_{(\mu, \nu)}$ will be referred to as the \emph{reduced c-bi-free partial $\R$-transform} of $(\mu, \nu)$.
\end{defn}

As in the single-variable case, it is sometimes more convenient to consider the function $\Phi_{(\mu, \nu)}$ defined by $\Phi_{(\mu, \nu)}(z, w) = \R_{(\mu, \nu)}(1/z, 1/w)$ instead. If $\mu$ and $\nu$ are compactly supported, then $\Phi_{(\mu, \nu)}$ can be written as
\[
\Phi_{(\mu, \nu)}(z, w) = \sum_{\substack{m, n \geq 0\\m + n \geq 1}}\K_{m, n}(\mu, \nu)\frac{1}{z^m}\frac{1}{w^n}
\]
for $|z|$, $|w|$ sufficiently large. We shall refer to $\Phi_{(\mu, \nu)}$ as the \emph{c-bi-free partial Voiculescu transform} of the pair $(\mu, \nu)$. In the next section, we will give the precise domain and an alternative definition of $\Phi_{(\mu, \nu)}$ in terms of analytic functions when studying limit theorems and infinite divisibility from an analytic point of view.

\section{Additive limit theorems and infinite divisibility}\label{SecLimitThm}
\label{sec:additive-limit-theorems}

In this section, limit theorems and infinite divisibility for  the additive c-bi-free convolution are studied

\subsection{Combinatorial aspects of the additive c-bi-free convolution}

Most of the results below are c-bi-free analogues of known results in free and/or bi-free probability theories. The first result is analogous to Voiculescu's algebraic bi-free central limit theorem \cite{V2014}*{Theorem 7.9}. In view of \cite{V2014}*{Definition 7.3 and Theorem 7.4}, it is clear that the following definition is the natural c-bi-free analogue of bi-free central limit distribution.

\begin{defn}\label{CBFGauss}
A two-faced family $\hat{a} =((a_i)_{i \in I}, (a_j)_{j \in J})$ in a two-state non-commutative probability space $(\A, \varphi, \psi)$ is said to have a \emph{c-bi-free central limit distribution} (or \emph{centred c-bi-free Gaussian distribution}) with covariance matrices $C^\varphi = (C^\varphi_{k, \ell})_{k, \ell \in I \sqcup J}$ and $C^\psi = (C^\psi_{k, \ell})_{k, \ell \in I \sqcup J}$ such that
\begin{enumerate}[$\qquad(1)$]
\item $\varphi(a_{\alpha(1)}a_{\alpha(2)}) = C^\varphi_{\alpha(1), \alpha(2)}$ for all $\alpha: [2] \to I \sqcup J$,

\item $\psi(a_{\alpha(1)}a_{\alpha(2)}) = C^\psi_{\alpha(1), \alpha(2)}$ for all $\alpha: [2] \to I \sqcup J$,

\item $\kappa_{\chi_\alpha}(a_{\alpha(1)}, \dots, a_{\alpha(d)}) = \K_{\chi_\alpha}(a_{\alpha(1)}, \dots, a_{\alpha(d)}) = 0$ for all $d \neq 2$ and $\alpha: [d] \to I \sqcup J$ where $\chi_\alpha: [d] \to \{\ell, r\}$ such that $\chi_\alpha(s) = \ell$ if $\alpha(s) \in I$ and $\chi_\alpha(s) = r$ if $\alpha(s) \in J$ for $1 \leq s \leq d$.
\end{enumerate}
\end{defn}

\begin{thm}[\textbf{The algebraic c-bi-free central limit theorem}]
\label{AlgCBFCLT}
Let $\{((a_{n, i})_{i \in I}, (a_{n, j})_{j \in J})\}_{n = 1}^\infty$ be a sequence of c-bi-free two-faced families in a two-state non-commutative probability space $(\A, \varphi, \psi)$ such that
\begin{enumerate}[$\qquad(1)$]
\item $\varphi(a_{n, k}) = \psi(a_{n, k}) = 0$ for all $n \geq 1$ and $k \in I \sqcup J$,

\item $\sup_{n \geq 1}|\varphi(a_{n, k_1}\cdots a_{n, k_m})| < \infty$ and $\sup_{n \geq 1}|\psi(a_{n, k_1}\cdots a_{n, k_m})| < \infty$ for all $k_1, \dots, k_m \in I \sqcup J$,

\item $\lim_{N \to \infty}\frac{1}{N}\sum_{n = 1}^N\varphi(a_{n, k}a_{n, \ell}) = C_{k, \ell}^\varphi$ and $\lim_{N \to \infty}\frac{1}{N}\sum_{n = 1}^N\psi(a_{n, k}a_{n, \ell}) = C_{k, \ell}^\psi$ for all $k, \ell \in I \sqcup J$ where $C^\varphi = (C^\varphi_{k, \ell})_{k, \ell \in I \sqcup J}$ and $C^\psi = (C^\psi_{k, \ell})_{k, \ell \in I \sqcup J}$ are complex matrices.
\end{enumerate}
For $N \in \mathbb{N}$, let $\hat{S}_N = ((S_{N, i})_{i \in I}, (S_{N, j})_{j \in J})$ be the two-faced family defined by
\[
S_{N, k} = \frac{1}{\sqrt{N}}\sum_{n = 1}^Na_{n, k},\quad k \in I \sqcup J.
\]
Then $\hat{S}_N$ converges in distribution as $N \to \infty$ to a centred c-bi-free Gaussian distribution with covariance matrices $(C^\varphi, C^\psi)$.
\end{thm}

\begin{proof}
The fact that the corresponding $(\ell, r)$-cumulants, and hence the $\psi$-moments, converge is precisely the content of \cite{V2014}*{Theorem 7.9}. On the other hand, since the $\varphi$-moments are polynomials in $(\ell, r)$- and c-$(\ell, r)$-cumulants, it suffices to show
\[
\lim_{N \to \infty}\K_{\chi_\alpha}(S_{N, \alpha(1)}, S_{N, \alpha(2)}) = C_{\alpha(1), \alpha(2)}^\varphi
\]
for all $\alpha: [2] \to I \sqcup J$, and
\[
\lim_{N \to \infty}\K_{\chi_\alpha}(S_{N, \alpha(1)}, \dots, S_{N, \alpha(d)}) = 0
\]
for all $d \neq 2$ and $\alpha: [d] \to I \sqcup J$. In view of assumptions $(1)$ to $(3)$, the additivity and multilinearity properties of the c-$(\ell, r)$-cumulants, along with the definition that $\K_{\chi_\alpha} = \varphi$ for all $\chi_\alpha: [1] \to I \sqcup J$, the arguments for these limits are exactly the same as the ones presented in the proof of \cite{V2014}*{Theorem 7.9}.
\end{proof}

The following result is a c-bi-free version of the Kac/Loeve theorem, which we state for the sake of completeness as its proof is same as the bi-free Kac/Loeve theorem \cite{S2016}*{Theorem 3.3.1}.

\begin{thm}[\textbf{The c-bi-free Kac/Loeve theorem}]
Let $(a_{1, \ell}, a_{1, r})$ and $(a_{2, \ell}, a_{2, r})$ be c-bi-free two-faced pairs in a two-state non-commutative probability space $(\A, \varphi, \psi)$ such that $\varphi(a_{n, k}) = \psi(a_{n, k}) = 0$ for $n \in \{1, 2\}$ and $k \in \{\ell, r\}$. Let $\theta \in (0, \pi/2)$, $a_{3, k} = \cos(\theta)a_{1, k} + \sin(\theta)a_{2, k}$, and $a_{4, k} = -\sin(\theta)a_{1, k} + \cos(\theta)a_{2, k}$ for $k \in \{\ell, r\}$. The two-faced pairs $(a_{3, \ell}, a_{3, r})$ and $(a_{4, \ell}, a_{4, r})$ are c-bi-free with respect to $(\varphi, \psi)$ if and only if the two-faced pairs $(a_{1, \ell}, a_{1, r})$ and $(a_{2, \ell}, a_{2, r})$ are identically distributed with a centred c-bi-free Gaussian distribution.
\end{thm}

In what follows, we prove a general c-bi-free limit theorem analogous to \cite{NS2006}*{Theorem 13.1} for free probability theory and \cite{G2016}*{Theorem 2.3 and Corollary 2.4} for bi-free probability theory.

\begin{lem}\label{LimitLem}
For every $N \in \mathbb{N}$, let $((a_{N, i})_{i \in I}, (a_{N, j})_{j \in J})$ be a two-faced family in a two-state non-commutative probability space $(\A_N, \varphi_N, \psi_N)$. The following are equivalent:
\begin{enumerate}[$\qquad(1)$]
\item For all $d \geq 1$ and $\alpha: [d] \to I \sqcup J$, the limits
\[
\lim_{N \to \infty}N\cdot\psi_N(a_{N, \alpha(1)}\cdots a_{N, \alpha(d)})\qand \lim_{N \to \infty}N\cdot\varphi_N(a_{N, \alpha(1)}\cdots a_{N, \alpha(d)})
\]
exist.

\item For all $d \geq 1$ and $\alpha: [d] \to I \sqcup J$, the limits
\[
\lim_{N \to \infty}N\cdot\kappa_{\chi_\alpha}^N(a_{N, \alpha(1)}, \dots, a_{N, \alpha(d)})\qand \lim_{N \to \infty}N\cdot\K_{\chi_\alpha}^N(a_{N, \alpha(1)}, \dots, a_{N, \alpha(d)})
\]
exist.
\end{enumerate}
Moreover, if these assertions hold, then
\[
\lim_{N \to \infty}N\cdot\psi_N(a_{N, \alpha(1)}\cdots a_{N, \alpha(d)}) = \lim_{N \to \infty}N\cdot\kappa_{\chi_\alpha}^N(a_{N, \alpha(1)}, \dots, a_{N, \alpha(d)})
\]
and
\[
\lim_{N \to \infty}N\cdot\varphi_N(a_{N, \alpha(1)}\cdots a_{N, \alpha(d)}) = \lim_{N \to \infty}N\cdot\K_{\chi_\alpha}^N(a_{N, \alpha(1)}, \dots, a_{N, \alpha(d)})
\]
for all $d \geq 1$ and $\alpha: [d] \to I \sqcup J$.
\end{lem}

\begin{proof}
The equivalence and equality of the corresponding limits of $\psi$-moments and $(\ell, r)$-cumulants is the content of \cite{G2016}*{Lemma 2.2}, thus we only have to take the limits of $\varphi$-moments and c-$(\ell, r)$-cumulants into account. Suppose assertion $(2)$ holds.  Then
\begin{align*}
&\lim_{N \to \infty}N\cdot\varphi_N(a_{N, \alpha(1)}\cdots a_{N, \alpha(d)})\\
&= \lim_{N \to \infty}N\cdot\sum_{\pi \in \BNC(\chi_\alpha)}\K_{\pi}^N(a_{N, \alpha(1)}, \dots, a_{N, \alpha(d)})\\
&= \lim_{N \to \infty}\left(N\cdot\K_{\chi_\alpha}^N(a_{N, \alpha(1)}, \dots, a_{N, \alpha(d)}) + O(1/N)\right)
\end{align*}
for all $d \geq 1$ and $\alpha: [d] \to I \sqcup J$. Conversely, suppose assertion $(1)$ holds.  Then
\begin{align*}
&\lim_{N \to \infty}N\cdot\K_{\chi_\alpha}^N(a_{N, \alpha(1)}, \dots, a_{N, \alpha(d)})\\
&= \lim_{N \to \infty}\left(N\cdot\varphi_N(a_{N, \alpha(1)}\cdots a_{N, \alpha(d)}) -  N\cdot\sum_{\substack{\pi \in \BNC(\chi_\alpha)\\\pi \neq 1_{\chi_\alpha}}}\K_{\pi}^N(a_{N, \alpha(1)}, \dots, a_{N, \alpha(d)})\right)\\
&= \lim_{N \to \infty}\left(N\cdot\varphi_N(a_{N, \alpha(1)}\cdots a_{N, \alpha(d)}) + O(1/N)\right)
\end{align*}
for all $d \geq 1$ and $\alpha: [d] \to I \sqcup J$.
\end{proof}

\begin{thm}
For every $N \in \mathbb{N}$ let $\{((a_{N, n, i})_{i \in I}, (a_{N, n, j})_{j \in J})\}_{n = 1}^N$ be c-bi-free and identically distributed two-faced families in a two-state non-commutative probability space $(\A_N, \varphi_N, \psi_N)$.  Let $\hat{S}_N = ((S_{N, i})_{i \in I}, (S_{N, j})_{j \in J})$ be the two-faced family defined by
\[
S_{N, k} = \sum_{n = 1}^Na_{N, n, k},\quad k \in I \sqcup J.
\]
The following are equivalent:
\begin{enumerate}[$\qquad(1)$]
\item There exists a two-faced family $\hat{s} = ((s_i)_{i \in I}, (s_j)_{j \in J})$ in a two-state non-commutative probability space $(\A, \varphi, \psi)$ such that $\hat{S}_N$ converges in distribution  to $\hat{s}$ as $N \to \infty$.

\item For all $d \geq 1$ and $\alpha: [d] \to I \sqcup J$, the limits
\[
\lim_{N \to \infty}N\cdot\psi_N(a_{N, n, \alpha(1)}\cdots a_{N, n, \alpha(d)})\qand \lim_{N \to \infty}N\cdot\varphi_N(a_{N, n, \alpha(1)}\cdots a_{N, n, \alpha(d)})
\]
exist and are independent of $n$.
\end{enumerate}
Moreover, if these assertions hold, then the free and c-$(\ell, r)$-cumulants of $\hat{s}$ are given by
\[
\kappa_{\chi_\alpha}(s_{\alpha(1)}, \dots, s_{\alpha(d)}) = \lim_{N \to \infty}N\cdot\psi_N(a_{N, n, \alpha(1)}\cdots a_{N, n, \alpha(d)})
\]
and
\[
\K_{\chi_\alpha}(s_{\alpha(1)}, \dots, s_{\alpha(d)}) = \lim_{N \to \infty}N\cdot\varphi_N(a_{N, n, \alpha(1)}\cdots a_{N, n, \alpha(d)})
\]
for all $d \geq 1$ and $\alpha: [d] \to I \sqcup J$.
\end{thm}

\begin{proof}
Suppose assertion $(1)$ holds. For $d \geq 1$ and $\alpha: [d] \to I \sqcup J$, the limit
\[
\lim_{N \to \infty}N\cdot\psi_N(a_{N, n, \alpha(1)}\cdots a_{N, n, \alpha(d)}),
\]
which is independent of $n$ by the assumption of identical distribution, exists and is equal to $\kappa_{\chi_\alpha}(s_{\alpha(1)}, \dots, s_{\alpha(d)})$ by \cite{G2016}*{Corollary 2.4}. On the other hand, we have
\begin{align*}
\varphi(s_{\alpha(1)}\cdots s_{\alpha(d)}) &= \lim_{N \to \infty}\varphi_N(S_{N, \alpha(1)}\cdots S_{N, \alpha(d)})\\
&= \lim_{N \to \infty}\sum_{t(1), \dots, t(d) = 1}^N\varphi_N(a_{N, t(1), \alpha(1)}\cdots a_{N, t(d), \alpha(d)})\\
&= \lim_{N \to \infty}\sum_{t(1), \dots, t(d) = 1}^N\sum_{\pi \in \BNC(\chi_\alpha)}\K_\pi^N(a_{N, t(1), \alpha(1)}, \dots, a_{N, t(d), \alpha(d)})\\
&= \lim_{N \to \infty}\sum_{\pi \in \BNC(\chi_\alpha)}N^{|\pi|}\cdot\K_\pi^N(a_{N, n, \alpha(1)}, \dots, a_{N, n, \alpha(d)}),
\end{align*}
where the last expression, which is independent of $n$, follows from the assumptions of c-bi-free independence and identical distribution. Since $\psi(s_{\alpha(1)}\cdots s_{\alpha(d)})$ and $\varphi(s_{\alpha(1)}\cdots s_{\alpha(d)})$ exist for all $d \geq 1$ and $\alpha: [d] \to I \sqcup J$, an easy induction argument shows that
\[
\lim_{N \to \infty}N^{|\pi|}\cdot\K_\pi^N(a_{N, n, \alpha(1)}, \dots, a_{N, n, \alpha(d)})
\]
exist for all $\pi \in \BNC(\chi_\alpha)$. In particular, choose $\pi = 1_{\chi_\alpha}$ and apply Lemma \ref{LimitLem}, assertion $(2)$ follows.

Conversely, suppose assertion $(2)$ holds. By \cite{G2016}*{Corollary 2.4}, the limits
\begin{equation}\label{LimitPsi}
\lim_{N \to \infty}\psi_N(S_{N, \alpha(1)}\cdots S_{N, \alpha(d)})
\end{equation}
exist for all $d \geq 1$ and $\alpha: [d] \to I \sqcup J$. On the other hand, the limits
\[
\lim_{N \to \infty}N^{|\pi|}\cdot\K_\pi^N(a_{N, n, \alpha(1)}, \dots, a_{N, n, \alpha(d)})
\]
exist for all $d \geq 1$, $\alpha: [d] \to I \sqcup J$, and $\pi \in \BNC(\chi_\alpha)$ by Lemma \ref{LimitLem}, thus the limits
\begin{equation}\label{LimitPhi}
\lim_{N \to \infty}\varphi_N(S_{N, \alpha(1)}\cdots S_{N, \alpha(d)})
\end{equation}
exist as well. One concludes assertion $(1)$ by abstractly constructing a two-faced family $\hat{s} = ((s_i)_{i \in I}, (s_j)_{j \in J})$ in a two-state non-commutative probability space $(\A, \varphi, \psi)$ and defining $\psi(s_{\alpha(1)}\cdots s_{\alpha(d)})$ and $\varphi(s_{\alpha(1)}\cdots s_{\alpha(d)})$ to be the corresponding limit in equation \eqref{LimitPsi} and equation \eqref{LimitPhi}  respectively.

Finally, for $d \geq 1$ and $\alpha: [d] \to I \sqcup J$, 
\begin{align*}
\sum_{\pi \in \BNC(\chi_\alpha)}\K_\pi(s_{\alpha(1)}, \dots, s_{\alpha(d)}) &= \varphi(s_{\alpha(1)}\cdots s_{\alpha(d)})\\
&= \sum_{\pi \in \BNC(\chi_\alpha)}\lim_{N \to \infty}N^{|\pi|}\cdot\K_\pi^N(a_{N, n, \alpha(1)}, \dots, a_{N, n, \alpha(d)}),
\end{align*}
and hence
\[
\K_\pi(s_{\alpha(1)}, \dots, s_{\alpha(d)}) = \lim_{N \to \infty}N^{|\pi|}\cdot\K_\pi^N(a_{N, n, \alpha(1)}, \dots, a_{N, n, \alpha(d)})
\]
for all $\pi \in \BNC(\chi_\alpha)$ since the c-$(\ell, r)$-cumulants are uniquely determined by equation \eqref{CBiFreeMomentCumulant}. In particular, for $\pi = 1_{\chi_\alpha}$, we have
\[
\K_{\chi_\alpha}(s_{\alpha(1)}, \dots, s_{\alpha(d)}) = \lim_{N \to \infty}N\cdot\K_{\chi_\alpha}^N(a_{N, n, \alpha(1)}, \dots, a_{N, n, \alpha(d)}) = \lim_{N \to \infty}N\cdot\varphi_N(a_{N, n, \alpha(1)}\cdots a_{N, n, \alpha(d)})
\]
by Lemma \ref{LimitLem}.
\end{proof}

\subsection{Examples of c-bi-free distributions}

If $(\mu, \nu)$ is a pair of compactly supported Borel probability measures on $\bR^2$, then $\mu$ and $\nu$ can be identified as states on the $*$-algebra $\bC[X, Y]$ where $X^* = X$ and $Y^* = Y$, via
\[
\mu(X^mY^n) = \int_{\bR^2}s^mt^n\,d\mu(s, t)\qand \nu(X^mY^n) = \int_{\bR^2}s^mt^n\,d\nu(s, t)
\]
for $m, n \geq 0$ with $m + n \geq 1$. In this case, we denote the $(m, n)^{\mathrm{th}}$ moment of $\mu$ and $\nu$ by $\mathbb{M}_{m, n}(\mu) = \mu(X^mY^n)$ and $\mathbb{M}_{m, n}(\nu) = \nu(X^mY^n)$ respectively. Moreover, the $(m, n)^{\mathrm{th}}$ free cumulant of $\nu$ is denoted by
\[
\kappa_{m, n}(\nu) = \kappa_{m + n}(\underbrace{X, \dots, X}_{m\,\mathrm{times}}, \underbrace{Y, \dots, Y}_{n\,\mathrm{times}})
\]
and the $(m, n)^{\mathrm{th}}$ c-free cumulant of $(\mu, \nu)$ is denoted by
\[
\K_{m, n}(\mu, \nu) = \K_{m + n}(\underbrace{X, \dots, X}_{m\,\mathrm{times}}, \underbrace{Y, \dots, Y}_{n\,\mathrm{times}}).
\]

\begin{defn}
Let $\nu$ be a compactly supported Borel probability measure on $\mathbb{R}^2$. A Borel probability measure $\mu$ on $\mathbb{R}^2$ is said to have a \emph{c-bi-free Gaussian distribution} with marginal means $(\eta_1, \eta_2)$, covariance matrix
\[
\begin{pmatrix}
a & c\\
c & b
\end{pmatrix},\quad a, b \geq 0,\quad |c| \leq \sqrt{ab},
\]
and accompanying distribution $\nu$ if the c-bi-free partial Voiculescu transform of the pair $(\mu, \nu)$ is given by
\[
\Phi_{(\mu, \nu)}(z, w) = \frac{\eta_1}{z} + \frac{\eta_2}{w} + \frac{a}{z^2} + \frac{b}{w^2} + \frac{c}{zw},\quad (z, w) \in (\mathbb{C}\,\backslash\,\mathbb{R})^2.
\]
In other words, the only non-vanishing c-free cumulants of $(\mu, \nu)$ are $\K_{1, 0}(\mu, \nu) = \eta_1$, $\K_{0, 1}(\mu, \nu) = \eta_2$, $\K_{2, 0}(\mu, \nu) = a$, $\K_{0, 2}(\mu, \nu) = b$, and $\K_{1, 1}(\mu, \nu) = c$. If $\eta_1 = \eta_2 = 0$, then $\mu$ is said to be \emph{centred}. If, in addition, $a = b = 1$ and $c \in [-1, 1]$, then $\mu$ is said to be \emph{standard}. In this case, we also denote this distribution by $\mu_{(a, b, c), \nu}$.
\end{defn}

The following corollary is an immediate consequence of Theorem \ref{AlgCBFCLT} applied to centred Borel probability measures on $\mathbb{R}^2$. For $\lambda > 0$, let $D_\lambda\nu$ denote the dilation of the measure $\nu$ by the factor $\lambda$; that is,
\[
D_\lambda\nu(B) = \nu(\lambda^{-1}B)
\]
for all Borel sets $B$ on $\bR^2$.

\begin{cor}[\textbf{The probabilistic c-bi-free central limit theorem}]
Let $(\mu, \nu)$ be a pair of compactly supported Borel probability measures on $\mathbb{R}^2$ with zero mean marginal distributions and covariance matrices
\[
\begin{pmatrix}
a & c\\
c & b
\end{pmatrix}\qand \begin{pmatrix}
a' & c'\\
c' & b'
\end{pmatrix}
\]
respectively.  Then
\[
\lim_{N \to \infty}\underbrace{D_{1/\sqrt{N}}\nu \boxplus\boxplus \cdots \boxplus\boxplus D_{1/\sqrt{N}}\nu}_{N\,\mathrm{times}} = \nu_{(a', b', c')}
\]
and
\[
\lim_{N \to \infty}\underbrace{(D_{1/\sqrt{N}}\mu, D_{1/\sqrt{N}}\nu) \boxplus\boxplus_{\rc} \cdots \boxplus\boxplus_{\rc} (D_{1/\sqrt{N}}\mu, D_{1/\sqrt{N}}\nu)}_{N\,\mathrm{times}} = (\mu_{(a, b, c), \nu_{(a', b', c')}}, \nu_{(a', b', c')}),
\]
where $\nu_{(a', b', c')}$ denotes the centred bi-free Gaussian distribution with covariance matrix
\[
\begin{pmatrix}
a' & c'\\
c' & b'
\end{pmatrix}
\]
introduced in \cite{V2014}*{Definition 7.3}.
\end{cor}

Note also that the first assertion is a special probabilistic version of the algebraic bi-free central limit theorem \cite{V2014}*{Theorem 7.9}.

\begin{defn}
Let $\lambda \geq 0$, $(0, 0) \neq (\alpha, \beta) \in \mathbb{R}^2$, and $\nu$ be a compactly supported Borel probability measure on $\mathbb{R}^2$. A Borel probability measure $\pi$ on $\mathbb{R}^2$ is said to have a \emph{c-bi-free Poisson distribution} with rate $\lambda$, jump size $(\alpha, \beta)$, and accompanying distribution $\nu$ if the c-free cumulants of the pair $(\pi, \nu)$ are given by
\[
\K_{m, n}(\pi, \nu) = \lambda\alpha^m\beta^n
\]
for all $m, n \geq 0$ with $m + n \geq 1$. This distribution is denoted $\pi_{\lambda, (\alpha, \beta), \nu}$.
\end{defn}

For such a pair $(\pi, \nu)$, the c-bi-free partial Voiculescu transform is given by
\begin{align*}
\Phi_{(\pi, \nu)}(z, w) &= \sum_{\substack{m, n \geq 0\\m + n \geq 1}}\lambda\alpha^m\beta^n\frac{1}{z^m}\frac{1}{w^n}\\
&= \frac{\lambda\alpha}{z - \alpha} + \frac{\lambda\beta}{w - \beta} + \frac{\lambda\alpha\beta}{(z - \alpha)(w - \beta)}\\
&= \frac{1}{z}\Phi_{(\pi^{(1)}, \nu^{(1)})}(z) + \frac{1}{w}\Phi_{(\pi^{(2)}, \nu^{(2)})}(w) + \frac{\lambda\alpha\beta}{(z - \alpha)(w - \beta)}
\end{align*}
for $|z|, |w|$ sufficiently large where $\Phi_{(\pi^{(1)}, \nu^{(1)})}$ and $\Phi_{(\pi^{(2)}, \nu^{(2)})}$ denote the c-free Voiculescu transforms (see the next subsection for an analytic review) of the marginal pairs $(\pi^{(1)}, \nu^{(1)})$ and $(\pi^{(2)}, \nu^{(2)})$ respectively. Note that $\mu^{(1)}$ and $\mu^{(2)}$ have c-free Poisson distributions with rate $\lambda$, jump sizes $\alpha$ and $\beta$, and accompanying distributions $\nu^{(1)}$ and $\nu^{(2)}$ respectively.

\begin{thm}[\textbf{The c-bi-free Poisson limit theorem}]
\label{CBFPoisson}
Let $\lambda, \lambda' \geq 0$ and $(0, 0) \neq (\alpha, \beta), (0, 0) \neq (\alpha', \beta') \in \mathbb{R}^2$. For $N \in \mathbb{N}$, let
\[
\mu_N = \left(1 - \frac{\lambda}{N}\right)\delta_{(0, 0)} + \frac{\lambda}{N}\delta_{(\alpha, \beta)}\qand \nu_N = \left(1 - \frac{\lambda'}{N}\right)\delta_{(0, 0)} + \frac{\lambda'}{N}\delta_{(\alpha', \beta')}.
\]
If $\pi_{\lambda', (\alpha', \beta')}$ is the bi-free Poisson distribution with rate $\lambda'$ and jump size $(\alpha', \beta')$ (see \cite{GHM2016}*{Example 3.13}), then
\[
\lim_{N \to \infty}\underbrace{\nu_N \boxplus\boxplus \cdots \boxplus\boxplus \nu_N}_{N\,\mathrm{times}} = \pi_{\lambda', (\alpha', \beta')}
\]
and
\[
\lim_{N \to \infty}\underbrace{(\mu_N, \nu_N) \boxplus\boxplus_{\rc} \cdots \boxplus\boxplus_{\rc} (\mu_N, \nu_N)}_{N\,\mathrm{times}} = (\pi_{\lambda, (\alpha, \beta), \pi_{\lambda', (\alpha', \beta')}}, \pi_{\lambda', (\alpha', \beta')}).
\]
\end{thm}

We omit the proof of Theorem \ref{CBFPoisson} as it will follow from Theorem \ref{thm:compound-c-b-free-Poisson} which studies the larger class of compound c-bi-free Poisson distributions.

\begin{defn}\label{CompCBFPoisson}
Let $\lambda \geq 0$, and let $\delta_{(0, 0)} \neq \sigma$ and $\nu$ be compactly supported Borel probability measures on $\mathbb{R}^2$. A Borel probability measure $\pi$ on $\mathbb{R}^2$ is said to have a \emph{compound c-bi-free Poisson distribution} with rate $\lambda$, jump distribution $\sigma$, and accompanying distribution $\nu$ if the c-free cumulants of the pair $(\pi, \nu)$ are given by
\[
\K_{m, n}(\pi, \nu) = \lambda\cdot\mathbb{M}_{m, n}(\sigma)
\]
for all $m, n \geq 0$ with $m + n \geq 1$.  This distribution is denoted by $\pi_{\lambda, \sigma, \nu}$.
\end{defn}

For such a pair $(\pi, \nu)$, the c-bi-free Voiculescu transform is given by
\begin{align*}
\Phi_{(\pi, \nu)}(z, w) &= \sum_{\substack{m, n \geq 0\\m + n \geq 1}}\lambda\int_{\mathbb{R}^2}s^mt^n\frac{1}{z^m}\frac{1}{w^n}\,d\sigma(s, t)\\
&= \lambda\int_{\mathbb{R}^2}\frac{s}{z - s}\,d\sigma(s, t) + \lambda\int_{\mathbb{R}^2}\frac{t}{w - t}\,d\sigma(s, t) + \lambda\int_{\mathbb{R}^2}\frac{st}{(z - s)(w - t)}\,d\sigma(s, t)\\
&= \frac{1}{z}\Phi_{(\pi^{(1)}, \nu^{(1)})}(z) + \frac{1}{w}\Phi_{(\pi^{(2)}, \nu^{(2)})}(w) + \lambda\int_{\mathbb{R}^2}\frac{st}{(z - s)(w - t)}\,d\sigma(s, t)
\end{align*}
for $|z|, |w|$ sufficiently large where $\Phi_{(\pi^{(1)}, \nu^{(1)})}$ and $\Phi_{(\pi^{(2)}, \nu^{(2)})}$ are the c-free Voiculescu transforms of the marginal pairs $(\pi^{(1)}, \nu^{(1)})$ and $(\pi^{(2)}, \nu^{(2)})$ respectively. Note that $\mu^{(1)}$ and $\mu^{(2)}$ have compound c-free Poisson distributions with rate $\lambda$, jump distributions $\sigma^{(1)}$ and $\sigma^{(2)}$, and accompanying distributions $\nu^{(1)}$ and $\nu^{(2)}$ respectively.

\begin{thm}[\textbf{The compound c-bi-free Poisson limit theorem}]
\label{thm:compound-c-b-free-Poisson}
Let $\lambda, \lambda' \geq 0$ and let $\delta_{(0, 0)} \neq \sigma, \delta_{(0, 0)} \neq \sigma'$ be compactly supported Borel probability measures on $\mathbb{R}^2$. For $N \in \mathbb{N}$, let
\[
\mu_N = \left(1 - \frac{\lambda}{N}\right)\delta_{(0, 0)} + \frac{\lambda}{N}\sigma\qand \nu_N = \left(1 - \frac{\lambda'}{N}\right)\delta_{(0, 0)} + \frac{\lambda'}{N}\sigma'.
\]
If $\pi_{\lambda', \sigma'}$ denotes the compound bi-free Poisson distribution with rate $\lambda'$ and jump distribution $\sigma'$ (see \cite{GHM2016}*{Example 3.13} or \cite{HW2016}*{Example 3.5}), then
\[
\lim_{N \to \infty}\underbrace{\nu_N \boxplus\boxplus \cdots \boxplus\boxplus \nu_N}_{N\,\mathrm{times}} = \pi_{\lambda', \sigma'}\qand \lim_{N \to \infty}\underbrace{(\mu_N, \nu_N) \boxplus\boxplus_{\rc} \cdots \boxplus\boxplus_{\rc} (\mu_N, \nu_N)}_{N\,\mathrm{times}} = (\pi_{\lambda, \sigma, \pi_{\lambda', \sigma'}}, \pi_{\lambda', \sigma'}).
\]
\end{thm}

\begin{proof}
The first assertion is the compound bi-free Poisson limit theorem \cite{GHM2016}*{Example 3.13}. For the second assertion, since moments are polynomials in free and c-free cumulants, and the first assertion implies the convergence of the corresponding free cumulants, it suffices to show the convergence of the corresponding c-free cumulants. Assume $N$ is large enough so that $\mu_N$ is a Borel probability measure on $\bR^2$. For $m, n \geq 0$ with $m + n \geq 1$, we have
\[
\mathbb{M}_{m, n}(\mu_N) = \frac{\lambda}{N}\int_{\mathbb{R}^2}s^mt^n\,d\sigma(s, t) = \frac{\lambda}{N}\mathbb{M}_{m, n}(\sigma).
\]
Thus
\[
\K_{m, n}(\mu_N, \nu_N) = \mathbb{M}_{m, n}(\mu_N) + O(1/N^2) = \frac{\lambda}{N}\mathbb{M}_{m, n}(\sigma) + O(1/N^2),
\]
and hence
\[
\K_{m, n}((\mu_N, \nu_N)^{\boxplus\boxplus_{\rc}N}) = \lambda\cdot\mathbb{M}_{m, n}(\sigma) + O(1/N) \to \lambda\cdot\mathbb{M}_{m, n}(\sigma)
\]
as $N \to \infty$.
\end{proof}

\subsection{Analytic aspects of the c-bi-free partial Voiculescu transform}

Given a finite positive Borel measure $\nu$ on $\bR$, its \emph{(one-dimensional) Cauchy transform} is defined by
\[
G_\nu(z) = \int_{\bR}\frac{1}{z - s}\,d\nu(s),\quad z \in (\bC\,\backslash\,\bR).
\]
Observe that $G_\nu: \bC^+ \to \bC^-$ (where $\bC^+$ and $\bC^-$ are the upper and lower half planes respectively) and $G_\nu(\overline{z}) = \overline{G_\nu(z)}$. Define $F_\nu(z) = \frac{1}{G_\nu(z)}$ for $z \in (\bC\,\backslash\,\bR)$. 

\begin{defn}
For $\alpha, \beta > 0$, the \emph{Stolz} and \emph{truncated Stolz} angels are defined by
\[
\Gamma_\alpha = \{z = x + iy \in \mathbb{C}^+ \, \mid \, |x| < \alpha y\} \qand \Gamma_{\alpha, \beta} = \{z = x + iy \in \Gamma_\alpha \, \mid \, y > \beta\}
\]
respectively. Moreover, let $\overline{\Gamma_{\alpha, \beta}} = \{\overline{z} \, \mid \, z \in \Gamma_{\alpha, \beta}\}$  and set
\[
\Omega_{\alpha, \beta} = \{(z, w) \in (\mathbb{C}\,\backslash\,\mathbb{R})^2 \, \mid \, z, w \in \Gamma_{\alpha, \beta} \cup \overline{\Gamma_{\alpha, \beta}}\}.
\]
\end{defn}

As shown in \cite{BV1993}, for every $\alpha > 0$, there exists a $\beta > 0$ such that $F_\nu^{-1}$ (the inverse under composition) is defined on $\Gamma_{\alpha, \beta} \cup \overline{\Gamma_{\alpha, \beta}}$. Define the \emph{free Voiculescu transform} of $\nu$ by
\[
\phi_\nu(z) = F_\nu^{-1}(z) - z = \R_\nu\left(\frac{1}{z}\right).
\]
Then the additive free convolution $\boxplus$ is characterized by
\[
\phi_{\nu_1 \boxplus \nu_2}(z) = \phi_{\nu_1}(z) + \phi_{\nu_2}(z)
\]
on the common domain of the involved functions.

Given a pair $(\mu, \nu)$ of Borel probability measures on $\bR$, the \emph{c-free Voiculescu transform} of $(\mu, \nu)$ is defined by
\[
\Phi_{(\mu, \nu)}(z) = F_\nu^{-1}(z) - F_\mu(F_\nu^{-1}(z))
\]
on a domain of the form $\Gamma_{\alpha, \beta} \cup \overline{\Gamma_{\alpha, \beta}}$ where $F_\nu^{-1}$ is defined. Given two pairs $(\mu_1, \nu_1)$ and $(\mu_2, \nu_2)$, their additive c-free convolution $(\mu_1, \nu_1) \boxplus_{\rc} (\mu_2, \nu_2)$ is another pair $(\mu, \nu)$ where $\nu = \nu_1 \boxplus \nu_2$ and $\mu$ is the unique measure such that
\[
\Phi_{(\mu, \nu)}(z) = \Phi_{(\mu_1, \nu_1)}(z) + \Phi_{(\mu_2, \nu_2)}(z)
\]
on the common domain of the involved functions. (see \cite{W2011}*{Proposition 2.2} or \cite{B2008}*{Corollary 4}). In the sequel, we may abuse notation and use $(\mu_1, \nu_1) \boxplus_{\rc} (\mu_2, \nu_2)$ to denote $\mu$.

The study of the analytic aspects of the additive bi-free convolution was initiated in \cite{HW2016}. Given a finite positive Borel measure $\nu$ on $\bR^2$, the \emph{(two-dimensional) Cauchy transform} of $\nu$ is defined by
\[
G_\nu(z, w) = \int_{\bR^2}\frac{1}{(z - s)(w - t)}\,d\nu(s, t),\quad (z, w) \in (\bC\,\backslash\,\bR)^2
\]
which satisfies $G_\nu(\overline{z}, \overline{w}) = \overline{G_\nu(z, w)}$. To study limit theorems and infinite divisibility, we need to discuss weak convergence of measures, which requires the notion of tightness. Following \cite{HW2016}*{Section 2}, a family $\mathcal{F}$ of finite signed Borel measures on $\mathbb{R}^2$ is said to be \emph{tight} if
\[
\lim_{N \to \infty}\sup_{\nu \in \mathcal{F}}|\nu|(\mathbb{R}^2\,\backslash\,K_N) = 0
\]
where $K_N = \{(s, t) \in \mathbb{R}^2 \, \mid \, |s| \leq N, |t| \leq N\}$. Moreover, the family $\F$ is tight if and only if the family $\{|\nu|^{(1)}, |\nu|^{(2)} \, \mid \, \nu \in \F\}$ of marginal distributions is a tight family of finite signed Borel measures on $\bR$. The following results were obtained in \cite{HW2016}*{Section 2}, which will be useful later. Note that by $z \to \infty$ non-tangentially we mean $|z| \to \infty$ but $z$ stays within a Stolz angel $\Gamma_\alpha$ for some $\alpha > 0$.

\begin{prop}\label{HWLem2.1}
Let $\{\mu_n\}_{n = 1}^\infty$ be a tight sequence of Borel probability measures on $\mathbb{R}^2$. The limits
\[
\lim_{|z| \to \infty, z \in \Gamma_\alpha}zG_{\mu_n}(z, w) = G_{\mu_n^{(2)}}(w)\qand \lim_{|w| \to \infty, w \in \Gamma_\alpha}wG_{\mu_n}(z, w) = G_{\mu_n^{(1)}}(z)
\]
hold uniformly in $n$ for all $(z, w) \in (\mathbb{C}\,\backslash\,\mathbb{R})^2$ and $\alpha > 0$.
\end{prop}

\begin{prop}\label{HWProp2.2}
Let $\{\mu_n\}_{n = 1}^\infty$ be a sequence of Borel probability measures on $\mathbb{R}^2$. The following assertions are equivalent:
\begin{enumerate}[$\qquad(1)$]
\item The sequence $\{\mu_n\}_{n = 1}^\infty$ converges weakly to a Borel probability measure $\mu$ on $\mathbb{R}^2$.

\item There exist two open sets $U \subset \mathbb{C}^{+} \times \mathbb{C}^{+}$ and $V \subset \mathbb{C}^{+} \times \mathbb{C}^{-}$ such that the pointwise limits $\lim_{n \to \infty}G_{\mu_n}(z, w) = G(z, w)$ exist for all $(z, w) \in U \cup V$, and the limit $zwG_{\mu_n}(z, w) \to 1$ holds uniformly in $n$ as $|z|, |w| \to \infty$ non-tangentially.
\end{enumerate}
Moreover, if these assertions hold, then $G = G_\mu$ on $(\mathbb{C}\,\backslash\,\mathbb{R})^2$.
\end{prop}

Recall from \cite{BV1993} that if $\nu$ is a Borel probability measure on $\mathbb{R}^2$ with marginal distributions $\nu^{(1)}$ and $\nu^{(2)}$, then for every $\alpha > 0$, there exists a $\beta > 0$ such that $F_{\nu^{(1)}}^{-1}$ and $F_{\nu^{(2)}}^{-1}$ are defined on $\Gamma_{\alpha, \beta} \cup \overline{\Gamma_{\alpha, \beta}}$, and
\[
F_{\nu^{(1)}}^{-1}(z) = z(1 + o(1)),\quad F_{\nu^{(2)}}^{-1}(w) = w(1 + o(1)),\qand G_\nu(z, w) = \frac{1}{zw}(1 + o(1))
\]
as $|z|, |w| \to \infty$ non-tangentially. By enlarging $\beta$ if necessary, we may assume $zwG_\nu(F_{\nu^{(1)}}^{-1}(z), F_{\nu^{(2)}}^{-1}(w))$ never vanishes on $\Omega_{\alpha, \beta}$. Then \cite{HW2016} defined the \emph{bi-free partial Voiculescu transform} of $\nu$ by
\[
\phi_{\nu}(z, w) = \frac{1}{z}\phi_{\nu^{(1)}}(z) + \frac{1}{w}\phi_{\nu^{(2)}}(w) + \widetilde{\phi}_{\nu}(z, w),\quad (z, w) \in \Omega_{\alpha, \beta},
\]
for some $\alpha, \beta > 0$ where $\widetilde{\phi}_{\nu}$ is the \emph{reduced bi-free partial Voiculescu transform} of $\nu$ given by
\[
\widetilde{\phi}_{\nu}(z, w) = 1 -  \frac{1}{zwG_\nu(F_{\nu^{(1)}}^{-1}(z), F_{\nu^{(2)}}^{-1}(w))}.
\]
As with the free case, the additive bi-free convolution $\boxplus\boxplus$ is characterized by
\[
\phi_{\nu_1 \boxplus\boxplus \nu_2}(z, w) = \phi_{\nu_1}(z, w) + \phi_{\nu_2}(z, w)
\]
on the common domain of the involved functions.

Note that the original linearizing transform of $\boxplus\boxplus$, introduced in \cite{V2016} and studied in \cite{HW2016}, was the bi-free partial $\R$-transform (which is defined for $|z|$, $|w|$ sufficiently small) and is related to the bi-free partial Voiculescu transform by the change of variables $(z, w) \mapsto (1/z, 1/w)$. Moreover, at the time of writing this paper, the operation $\boxplus\boxplus$ is only defined for compactly supported and/or infinitely divisible measures. Consequently, these restrictions are also in place for $\boxplus\boxplus_{\rc}$ as below. Once these operations have been extended to arbitrary measures, it is expected that the same results also hold in the general case.

\begin{defn}
Let $(\mu, \nu)$ be a pair of Borel probability measures on $\mathbb{R}^2$. The \emph{c-bi-free partial Voiculescu transform} of $(\mu, \nu)$ is defined by
\[
\Phi_{(\mu, \nu)}(z, w) = \frac{1}{z}\Phi_{(\mu^{(1)}, \nu^{(1)})}(z) + \frac{1}{w}\Phi_{(\mu^{(2)}, \nu^{(2)})}(w) + \widetilde{\Phi}_{(\mu, \nu)}(z, w),\quad (z, w) \in \Omega_{\alpha, \beta},
\]
for some $\alpha, \beta > 0$ where
\[
\widetilde{\Phi}_{(\mu, \nu)}(z, w) = \frac{F_{\mu^{(1)}}(F_{\nu^{(1)}}^{-1}(z))F_{\mu^{(2)}}(F_{\nu^{(2)}}^{-1}(w))G_\mu(F_{\nu^{(1)}}^{-1}(z), F_{\nu^{(2)}}^{-1}(w)) - 1}{zwG_\nu(F_{\nu^{(1)}}^{-1}(z), F_{\nu^{(2)}}^{-1}(w))}.
\]
The function $\widetilde{\Phi}_{(\mu, \nu)}$ will be referred to as the \emph{reduced c-bi-free partial Voiculescu transform} of $(\mu, \nu)$.
\end{defn}

By taking non-tangential limits, some basic properties of $\Phi_{(\mu, \nu)}$ immediately follow.

\begin{lem}\label{Marginals}
If $\Phi_{(\mu, \nu)}: \Omega_{\alpha, \beta} \to \mathbb{C}$ is the c-bi-free partial Voiculescu transform of some pair $(\mu, \nu)$ of Borel probability measures on $\mathbb{R}^2$, then
\[
\lim_{|z| \to \infty}\Phi_{(\mu, \nu)}(z, w) = \frac{1}{w}\Phi_{(\mu^{(2)}, \nu^{(2)})}(w),\, \lim_{|w| \to \infty}\Phi_{(\mu, \nu)}(z, w) = \frac{1}{z}\Phi_{(\mu^{(1)}, \nu^{(1)})}(z), \text{ and } \lim_{|z|, |w| \to \infty}\Phi_{(\mu, \nu)}(z, w) = 0
\]
non-tangentially.
\end{lem}

\begin{proof}
For the first limit, since $\frac{1}{z}\Phi_{(\mu^{(1)}, \nu^{(1)})}(z) \to 0$ as $|z| \to \infty$ non-tangentially (see \cite{W2011}), it is enough to show $\widetilde{\Phi}_{(\mu, \nu)}(z, w) \to 0$ as $|z| \to \infty$ non-tangentially. Since $F_{\nu^{(1)}}^{-1}(z) = z(1 + o(1))$ and $F_{\mu^{(1)}}(z) = z(1 + o(1))$ as $|z| \to \infty$ non-tangentially (see \cite{BV1993}), we have
\[
zwG_\nu(F_{\nu^{(1)}}^{-1}(z), F_{\nu^{(2)}}^{-1}(w)) \to 1\qand F_{\mu^{(1)}}(F_{\nu^{(1)}}^{-1}(z))G_\mu(F_{\nu^{(1)}}^{-1}(z), F_{\nu^{(2)}}^{-1}(w)) \to G_{\mu^{(2)}}(F_{\nu^{(2)}}^{-1}(z))
\]
as $|z| \to \infty$ non-tangentially (see \cite{HW2016}). Since $F_{\mu^{(2)}}$ is the reciprocal of $G_{\mu^{(2)}}$ by definition, the first limit follows. The second limit can be shown similarly.  Hence the third limit holds.
\end{proof}

\begin{cor}\label{Unique}
If $(\mu_1, \nu_1)$ and $(\mu_2, \nu_2)$ are two pairs of Borel probability measures on $\mathbb{R}^2$ such that $\phi_{\nu_1} = \phi_{\nu_2}$ and $\Phi_{(\mu_1, \nu_1)} = \Phi_{(\mu_2, \nu_2)}$, then $(\mu_1, \nu_1) = (\mu_2, \nu_2)$.
\end{cor}

\begin{proof}
The fact that $\phi_{\nu_1} = \phi_{\nu_2}$ implies $\nu_1 = \nu_2$ was shown in \cite{HW2016}*{Proposition 2.5}. If, in addition, $\Phi_{(\mu_1, \nu_1)} = \Phi_{(\mu_2, \nu_2)}$, then $\mu_1^{(1)} = \mu_2^{(1)}$ and $\mu_1^{(2)} = \mu_2^{(2)}$ by taking non-tangential limits as in Lemma \ref{Marginals}. By the definition of the c-bi-free partial Voiculescu transform, we have
\[
G_{\mu_1}(F_{\nu_1^{(1)}}^{-1}(z), F_{\nu_1^{(2)}}^{-1}(w)) = G_{\mu_2}(F_{\nu_2^{(1)}}^{-1}(z), F_{\nu_2^{(2)}}^{-1}(w)).
\]
Therefore $G_{\mu_1} = G_{\mu_2}$ first on an open set $\Omega_{\alpha', \beta'}$ for some $\alpha', \beta' > 0$, and then on the whole space $(\mathbb{C}\,\backslash\,\mathbb{R})^2$ by analytic continuation. Since the (two-dimensional) Cauchy transform uniquely determines the underlying measure (see \cite{HW2016}), the result follows.
\end{proof}

The next proposition is a continuity result for the c-bi-free partial Voiculescu transform, analogous to \cite{W2011}*{Proposition 2.4} for the c-free case and \cite{HW2016}*{Proposition 2.6} for the bi-free case.

\begin{prop}\label{ConvProp}
Let $\{\mu_n\}_{n = 1}^\infty$ and $\{\nu_n\}_{n = 1}^\infty$ be two sequences of Borel probability measures on $\mathbb{R}^2$ such that the sequence $\{\nu_n\}_{n = 1}^\infty$ converges weakly to a Borel probability measure $\nu$ on $\mathbb{R}^2$. The following are equivalent:
\begin{enumerate}[$\qquad(1)$]
\item The sequence $\{\mu_n\}_{n = 1}^\infty$ converges weakly to a Borel probability measure $\mu$ on $\mathbb{R}^2$.

\item There exist $\alpha, \beta > 0$ such that all $\Phi_{(\mu_n, \nu_n)}$ are defined on $\Omega_{\alpha, \beta}$, the pointwise limits $\Phi(z, w) := \lim_{n \to \infty}\Phi_{(\mu_n, \nu_n)}(z, w)$ exist for all $(z, w) \in \Omega_{\alpha, \beta}$, and the limit $\Phi_{(\mu_n, \nu_n)}(z, w) \to 0$ holds uniformly in $n$ as $|z|, |w| \to \infty$ non-tangentially.
\end{enumerate}
Moreover, if these assertions hold, then $\Phi = \Phi_{(\mu, \nu)}$ on $\Omega_{\alpha, \beta}$.
\end{prop}

\begin{proof}
 The existence of a common domain $\Omega_{\alpha, \beta}$ for all $\Phi_{(\mu_n, \nu_n)}$ is guaranteed by the assumption that the sequence $\{\nu_n\}_{n = 1}^\infty$ converges weakly to $\nu$ (see \cite{HW2016}*{Proposition 2.6}).  Suppose assertion $(1)$ holds.  Then the sequences $\{\mu_n^{(1)}\}_{n = 1}^\infty$ and $\{\mu_n^{(2)}\}_{n = 1}^\infty$ of marginal distributions converge weakly to $\mu^{(1)}$ and $\mu^{(2)}$ respectively. Therefore the weak convergences of the sequences establish the pointwise convergence $\Phi_{(\mu_n, \nu_n)} \to \Phi_{(\mu, \nu)}$ on $\Omega_{\alpha, \beta}$. Being weakly convergent, all of the mentioned sequences are tight. By \cite{W2011}*{Proposition 2.4}
 \[
\frac{1}{z}\Phi_{(\mu_n^{(1)}, \nu_n^{(1)})}(z) \to 0\qand \frac{1}{w}\Phi_{(\mu_n^{(2)}, \nu_n^{(2)})}(w) \to 0
\]
uniformly in $n$ as $|z|, |w| \to \infty$ non-tangentially. Moreover, by \cite{HW2016}*{Lemma 2.1}, we have
\[
G_{\nu_n}(z, w) = \frac{1}{zw}(1 + o(1))\qand G_{\mu_n}(z, w) = \frac{1}{zw}(1 + o(1)),
\]
and hence $\widetilde{\Phi}_{(\mu_n, \nu_n)}(z, w) \to 0$ uniformly in $n$ as $|z|, |w| \to \infty$ non-tangentially.

Conversely, suppose assertion $(2)$ holds and fix $\alpha > 0$.  Then for every $\varepsilon > 0$ there exists a $\beta > 0$ such that
\[
|\Phi_{(\mu_n, \nu_n)}(z, w)| < \varepsilon
\]
for all $n \geq 1$ and $(z, w) \in (\Gamma_{\alpha, \beta})^2$. Fixing $z$ and letting $|w| \to \infty$ non-tangentially, Lemma \ref{Marginals} implies $\lim_{n \to \infty}\Phi_{(\mu_n^{(1)}, \nu_n^{(1)})}(z)$ exists and $\frac{1}{z}\Phi_{(\mu_n^{(1)}, \nu_n^{(1)})}(z) \to 0$ uniformly in $n$ as $|z| \to \infty$ non-tangentially. Therefore by \cite{W2011}*{Proposition 2.4}, the sequence $\{\mu_n^{(1)}\}_{n = 1}^\infty$ is weakly convergent, and thus is tight. Similarly, the sequence $\{\mu_n^{(2)}\}_{n = 1}^\infty$ is tight.  Hence the sequence $\{\mu_n\}_{n = 1}^\infty$ is tight. Being probability measures, the sequence $\{\mu_n\}_{n = 1}^\infty$ has a subsequence $\{\mu_{n_j}\}_{j = 1}^\infty$ converging weakly to some Borel probability measure $\mu$ on $\mathbb{R}^2$.  Therefore
\[
\lim_{n \to \infty}\Phi_{(\mu_n, \nu_n)}(z, w) = \lim_{j \to \infty}\Phi_{(\mu_{n_j}, \nu_{n_j})}(z, w) = \Phi_{(\mu, \nu)}(z, w)
\]
for $(z, w) \in \Omega_{\alpha, \beta}$. This implies $\lim_{n \to \infty}G_{\mu_n}(z, w) = G_\mu(z, w)$ for all $(z, w) \in (\mathbb{C}\,\backslash\,\mathbb{R})^2$. Finally, the assumption $\Phi_{(\mu_n, \nu_n)} \to 0$ uniformly in $n$ as $|z|, |w| \to \infty$ non-tangentially, together with the assumption on the sequence $\{\nu_n\}_{n = 1}^\infty$ imply $zwG_{\mu_n}(z, w) \to 1$ uniformly in $n$ as $|z|, |w| \to \infty$ non-tangentially, and the result follows from Proposition \ref{HWProp2.2}.
\end{proof}

\subsection{Analytic aspects of the additive c-bi-free convolution}

We now begin the study of limit theorems and infinite divisibility with respect to the additive c-bi-free convolution $\boxplus\boxplus_{\rc}$. Like other additive convolutions, one of the main goals is to define and characterize pairs of Borel probability measures on $\bR^2$ which are $\boxplus\boxplus_{\rc}$-infinitely divisible. For this purpose, the following limit theorem is crucial.

\begin{thm}\label{LimitThm}
Let $\{\mu_n\}_{n = 1}^\infty$ and $\{\nu_n\}_{n = 1}^\infty$ be sequences of Borel probability measures on $\mathbb{R}^2$ and $\{k_n\}_{n = 1}^\infty$ a sequence of positive integers with $\lim_{n \to \infty}k_n = \infty$. Assume there is a common domain $\Omega_{\alpha, \beta}$ and a Borel probability measure $\nu$ on $\mathbb{R}^2$ such that $\lim_{n \to \infty}k_n\phi_{\nu_n} = \phi_\nu$ pointwise on $\Omega_{\alpha, \beta}$ and the sequences $\{[\nu_n^{(1)}]^{\boxplus k_n}\}_{n = 1}^\infty$ and $\{[\nu_n^{(2)}]^{\boxplus k_n}\}_{n = 1}^\infty$ converge weakly to $\nu^{(1)}$ and $\nu^{(2)}$ respectively. Assume furthermore that the sequences $\{\zeta_n^{(1)}\}_{n = 1}^\infty$ and $\{\zeta_n^{(2)}\}_{n = 1}^\infty$ defined by
\[
\zeta_n^{(1)} = \underbrace{(\mu_n^{(1)}, \nu_n^{(1)}) \boxplus_{\rc} \cdots \boxplus_{\rc} (\mu_n^{(1)}, \nu_n^{(1)})}_{k_n\,\mathrm{times}}\qand \zeta_n^{(2)} = \underbrace{(\mu_n^{(2)}, \nu_n^{(2)}) \boxplus_{\rc} \cdots \boxplus_{\rc} (\mu_n^{(2)}, \nu_n^{(2)})}_{k_n\,\mathrm{times}},\quad n \geq 1,
\]
converge weakly to some Borel probability measures $\mu^{(1)}$ and $\mu^{(2)}$ on $\mathbb{R}$ respectively. The following are equivalent:
\begin{enumerate}[$\qquad(1)$]
\item There exists a common domain $\Omega_{\alpha, \beta}$ such that the pointwise limits
\[
\Phi(z, w) := \lim_{n \to \infty}k_n\Phi_{(\mu_n, \nu_n)}(z, w)
\]
exist for all $(z, w) \in \Omega_{\alpha, \beta}$.

\item The pointwise limits
\[
\widetilde{\Phi}(z, w) := \lim_{n \to \infty}k_n\int_{\mathbb{R}^2}\frac{st}{(z - s)(w - t)}\,d\mu_n(s, t)
\]
exist for all $(z, w) \in (\mathbb{C}\,\backslash\,\mathbb{R})^2$.

\item The finite signed Borel measures
\[
d\widetilde{\rho}_n(s, t) = k_n\frac{st}{\sqrt{1 + s^2}\sqrt{1 + t^2}}\,d\mu_n(s, t)
\]
converge weakly to a finite signed Borel measure $\rho$ on $\mathbb{R}^2$.
\end{enumerate}
Moreover, if these assertions hold, then the function $\widetilde{\Phi}$ from assertion $(2)$ has a unique integral representation
\[
\widetilde{\Phi}(z, w) = \int_{\mathbb{R}^2}\frac{\sqrt{1 + s^2}\sqrt{1 + t^2}}{(z - s)(w - t)}\,d\rho(s, t)
\]
and the function $\Phi$ from assertion $(1)$ can be written as
\[
\Phi(z, w) = \frac{1}{z}\Phi_{(\mu^{(1)}, \nu^{(1)})}(z) + \frac{1}{w}\Phi_{(\mu^{(2)}, \nu^{(2)})}(w) + \widetilde{\Phi}(z, w)
\]
for all $(z, w) \in (\mathbb{C}\,\backslash\,\mathbb{R})^2$.
\end{thm}

\begin{proof}
Note first that the existence of a common domain $\Omega_{\alpha, \beta}$ for all $\Phi_{(\mu_n, \nu_n)}$ is guaranteed by the assumption that $\lim_{n \to \infty}k_n\phi_{\nu_n} = \phi_\nu$ on $\Omega_{\alpha, \beta}$. Moreover, the equivalence of assertions $(2)$ and $(3)$ was proved in \cite{HW2016}*{Theorem 3.2} under the conditions that the sequences $\{\sigma_n^{(1)}\}_{n = 1}^\infty$ and $\{\sigma_n^{(2)}\}_{n = 1}^\infty$ defined by
\[
d\sigma_n^{(1)}(s) = k_n\frac{s^2}{1 + s^2}\,d\mu_n^{(1)}(s)\qand d\sigma_n^{(2)}(t) = k_n\frac{t^2}{1 + t^2}\,d\mu_n^{(2)}(t),\quad n \geq 1,
\]
are weakly convergent, which follow since the sequences $\{[\mu_n^{(1)}]^{\boxplus k_n}\}_{n = 1}^\infty$ and $\{[\mu_n^{(2)}]^{\boxplus k_n}\}_{n = 1}^\infty$ are weakly convergent. By \cite{W2011}*{Theorem 3.5}, these conditions also follow from our assumptions that the sequences $\{\zeta_n^{(1)}\}_{n = 1}^\infty$ and $\{\zeta_n^{(2)}\}_{n = 1}^\infty$ are weakly convergent, thus assertions $(2)$ and $(3)$ are equivalent. As seen in the proof of \cite{HW2016}*{Theorem 3.2}, it is enough to assume that the function $\widetilde{\Phi}$ from assertion $(2)$ exists on $\Omega_{\alpha, \beta}$ in order to deduce the weak convergence of $\{\widetilde{\rho}_n\}_{n = 1}^\infty$ to $\rho$ and to conclude
\[
\widetilde{\Phi}(z, w) = \lim_{n \to \infty}\int_{\mathbb{R}^2}\frac{\sqrt{1 + s^2}\sqrt{1 + t^2}}{(z - s)(w - t)}\,d\widetilde{\rho}_n(s, t) = \int_{\mathbb{R}^2}\frac{\sqrt{1 + s^2}\sqrt{1 + t^2}}{(z - s)(w - t)}\,d\rho(s, t).
\]
This also ensures the uniqueness of the integral representation of $\widetilde{\Phi}$ as it is the Cauchy transform of the measure $\sqrt{1 + s^2}\sqrt{1 + t^2}\,d\rho(s, t)$. On the other hand, assertion $(3)$ implies that
\[
\lim_{n \to \infty}\int_{\mathbb{R}^2}f(s, t)\,d\widetilde{\rho}_n(s, t) = \lim_{n \to \infty}k_n\int_{\mathbb{R}^2}f(s, t)\frac{st}{\sqrt{1 + s^2}\sqrt{1 + t^2}}\,d\mu_n(s, t)
\]
exists for all bounded continuous functions $f$. Letting
\[
f(s, t) = \frac{\sqrt{1 + s^2}\sqrt{1 + t^2}}{(z - s)(w - t)}
\]
for $(z, w) \in (\mathbb{C}\,\backslash\,\mathbb{R})^2$ then extends $\widetilde{\Phi}$ from $\Omega_{\alpha, \beta}$ to $(\mathbb{C}\,\backslash\,\mathbb{R})^2$.

For the equivalence of assertions $(1)$ and $(2)$, notice the assumptions that the sequences $\{[\nu_n^{(1)}]^{\boxplus k_n}\}_{n = 1}^\infty$, $\{[\nu_n^{(2)}]^{\boxplus k_n}\}_{n = 1}^\infty$, $\{\zeta_n^{(1)}\}_{n = 1}^\infty$, and $\{\zeta_n^{(2)}\}_{n = 1}^\infty$ are weakly convergent imply the infinitesimality of the sequences $\{\nu_n^{(1)}\}_{n = 1}^\infty$, $\{\nu_n^{(2)}\}_{n = 1}^\infty$, $\{\mu_n^{(1)}\}_{n = 1}^\infty$, and $\{\mu_n^{(2)}\}_{n = 1}^\infty$ (i.e. they all converge weakly to $\delta_0$, and hence both $\{\nu_n\}_{n = 1}^\infty$ and $\{\mu_n\}_{n = 1}^\infty$ converge weakly to $\delta_{(0, 0)}$). By \cite{W2011}*{Proposition 2.4}, we have
\[
\lim_{n \to \infty}k_n\Phi_{(\mu_n^{(1)}, \nu_n^{(1)})}(z) = \lim_{n \to \infty}\Phi_{(\zeta_n^{(1)}, [\nu_n^{(1)}]^{\boxplus k_n})}(z) = \Phi_{(\mu^{(1)}, \nu^{(1)})}(z)
\]
and
\[
\lim_{n \to \infty}k_n\Phi_{(\mu_n^{(2)}, \nu_n^{(2)})}(w) = \lim_{n \to \infty}\Phi_{(\zeta_n^{(2)}, [\nu_n^{(2)}]^{\boxplus k_n})}(w) = \Phi_{(\mu^{(2)}, \nu^{(2)})}(w)
\]
for all $z, w \in \Gamma_{\alpha, \beta} \cup \overline{\Gamma_{\alpha, \beta}}$. Therefore the limit
\[
\lim_{n \to \infty}k_n\Phi_{(\mu_n, \nu_n)}(z, w) = \lim_{n \to \infty}k_n\left(\frac{1}{z}\Phi_{(\mu_n^{(1)}, \nu_n^{(1)})}(z) + \frac{1}{w}\Phi_{(\mu_n^{(2)}, \nu_n^{(2)})}(w) + \widetilde{\Phi}_{(\mu_n, \nu_n)}(z, w)\right)
\]
exists for $(z, w) \in \Omega_{\alpha, \beta}$ if and only if the limit
\[
\lim_{n \to \infty}k_n\widetilde{\Phi}_{(\mu_n, \nu_n)}(z, w) = \lim_{n \to \infty}k_n\frac{F_{\mu_n^{(1)}}(F_{\nu_n^{(1)}}^{-1}(z))F_{\mu_n^{(2)}}(F_{\nu_n^{(2)}}^{-1}(w))G_{\mu_n}(F_{\nu_n^{(1)}}^{-1}(z), F_{\nu_n^{(2)}}^{-1}(w)) - 1}{zwG_{\nu_n}(F_{\nu_n^{(1)}}^{-1}(z), F_{\nu_n^{(2)}}^{-1}(w))}
\]
exists for $(z, w) \in \Omega_{\alpha, \beta}$.  In this case, we would have
\[
\Phi(z, w) = \frac{1}{z}\Phi_{(\mu^{(1)}, \nu^{(1)})}(z) + \frac{1}{w}\Phi_{(\mu^{(2)}, \nu^{(2)})}(w) + \lim_{n \to \infty}k_n\widetilde{\Phi}_{(\mu_n, \nu_n)}(z, w)
\]
for $(z, w) \in \Omega_{\alpha, \beta}$. By the infinitesimality of the sequences $\{\nu_n\}_{n = 1}^\infty$, $\{\nu_n^{(1)}\}_{n = 1}^\infty$, $\{\nu_n^{(2)}\}_{n = 1}^\infty$, $\{\mu_n^{(1)}\}_{n = 1}^\infty$, and $\{\mu_n^{(2)}\}_{n = 1}^\infty$, we have
\[
\lim_{n \to \infty}zwG_{\nu_n}(F_{\nu_n^{(1)}}^{-1}(z), F_{\nu_n^{(2)}}^{-1}(w)) = 1,\quad \lim_{n \to \infty}F_{\mu_n^{(1)}}(F_{\nu_n^{(1)}}^{-1}(z)) = z,\qand \lim_{n \to \infty}F_{\mu_n^{(2)}}(F_{\nu_n^{(2)}}^{-1}(w)) = w.
\]
Using the definitions $\phi_{\nu_n^{(1)}}(z) = F_{\nu_n^{(1)}}^{-1}(z) - z$ and $\phi_{\nu_n^{(2)}}(w) = F_{\nu_n^{(2)}}^{-1}(w) - w$, we obtain 
\[
\frac{1}{F_{\nu_n^{(1)}}^{-1}(z) - s} = \frac{1}{z - s}\left(1 - \frac{\phi_{\nu_n^{(1)}}(z)}{F_{\nu_n^{(1)}}^{-1}(z) - s}\right)\qand \frac{1}{F_{\nu_n^{(2)}}^{-1}(w) - t} = \frac{1}{w - t}\left(1 - \frac{\phi_{\nu_n^{(2)}}(w)}{F_{\nu_n^{(2)}}^{-1}(w) - t}\right),
\]
which gives, due to the convergence of measures to $\delta_0$ and $\delta_{(0,0)}$, that
\begin{align*}
\int_{\mathbb{R}^2}  &  \frac{F_{\mu_n^{(1)}}(F_{\nu_n^{(1)}}^{-1}(z))F_{\mu_n^{(2)}}(F_{\nu_n^{(2)}}^{-1}(w))}{(F_{\nu_n^{(1)}}^{-1}(z) - s)(F_{\nu_n^{(2)}}^{-1}(w) - t)}\,d\mu_n(s, t) - 1
\\ &= -\int_{\mathbb{R}^2}\frac{\phi_{\nu_n^{(1)}}(z)F_{\mu_n^{(1)}}(F_{\nu_n^{(1)}}^{-1}(z))F_{\mu_n^{(2)}}(F_{\nu_n^{(2)}}^{-1}(w))}{(z - s)(w - t)(F_{\nu_n^{(1)}}^{-1}(z) - s)}\,d\mu_n(s, t)\\
&\quad -\int_{\mathbb{R}^2}\frac{\phi_{\nu_n^{(2)}}(w)F_{\mu_n^{(1)}}(F_{\nu_n^{(1)}}^{-1}(z))F_{\mu_n^{(2)}}(F_{\nu_n^{(2)}}^{-1}(w))}{(z - s)(w - t)(F_{\nu_n^{(2)}}^{-1}(w) - t)}\,d\mu_n(s, t)\\
 &\quad +\int_{\mathbb{R}^2}\frac{\phi_{\nu_n^{(1)}}(z)\phi_{\nu_n^{(2)}}(w)F_{\mu_n^{(1)}}(F_{\nu_n^{(1)}}^{-1}(z))F_{\mu_n^{(2)}}(F_{\nu_n^{(2)}}^{-1}(w))}{(z - s)(w - t)(F_{\nu_n^{(1)}}^{-1}(z) - s)(F_{\nu_n^{(2)}}^{-1}(w) - t)}\,d\mu_n(s, t)\\
&\quad +\int_{\mathbb{R}^2}\frac{F_{\mu_n^{(1)}}(F_{\nu_n^{(1)}}^{-1}(z))F_{\mu_n^{(2)}}(F_{\nu_n^{(2)}}^{-1}(w)) - (z - s)(w - t)}{(z - s)(w - t)}\,d\mu_n(s, t).
\end{align*}
Let $I_1$ to $I_4$ denote the above four integrals.  Note that $\lim_{n \to \infty}k_nI_j$ exists for $j = 1, 2, 3$.  Thus the equivalence of assertions $(1)$ and $(2)$ reduces to the equivalence of assertion $(1)$ and the existence of $\lim_{n \to \infty}k_nI_4$. To this end, decompose $I_4$ as
\begin{align*}
I_4 &= \int_{\mathbb{R}^2}\frac{tz + sw - st}{(z - s)(w - t)}\,d\mu_n(s, t) + \int_{\mathbb{R}^2}\frac{F_{\mu_n^{(1)}}(F_{\nu_n^{(1)}}^{-1}(z))F_{\mu_n^{(2)}}(F_{\nu_n^{(2)}}^{-1}(w)) - zw}{(z - s)(w - t)}\,d\mu_n(s, t)\\
&= \int_{\mathbb{R}}\frac{s}{z - s}\,d\mu_n^{(1)}(s) + \int_{\mathbb{R}}\frac{t}{w - t}\,d\mu_n^{(2)}(t) + \int_{\mathbb{R}^2}\frac{st}{(z - s)(w - t)}\,d\mu_n(s, t)\\
&\quad +\int_{\mathbb{R}^2}\frac{F_{\mu_n^{(1)}}(F_{\nu_n^{(1)}}^{-1}(z))F_{\mu_n^{(2)}}(F_{\nu_n^{(2)}}^{-1}(w)) - zw}{(z - s)(w - t)}\,d\mu_n(s, t).
\end{align*}
Let $I_5$ denote the last integral above. Using the definitions $\Phi_{(\mu_{n}^{(1)}, \nu_{n}^{(1)})}(z) = F_{\nu_n^{(1)}}^{-1}(z) - F_{\mu_n^{(1)}}(F_{\nu_n^{(1)}}^{-1}(z))$ and $\Phi_{(\mu_{n}^{(2)}, \nu_{n}^{(2)})}(w) = F_{\nu_n^{(2)}}^{-1}(w) - F_{\mu_n^{(2)}}(F_{\nu_n^{(2)}}^{-1}(w))$, the numerator of the integrand of $I_5$ can be written as
\[
(\phi_{\nu_n^{(1)}}(z) - \Phi_{(\mu_{n}^{(1)}, \nu_{n}^{(1)})}(z) + z)(\phi_{\nu_n^{(2)}}(w) - \Phi_{(\mu_{n}^{(2)}, \nu_{n}^{(2)})}(w) + w) - zw.
\]
Therefore $\lim_{n \to \infty}k_nI_5$ exists because after expanding out the above expression and cancelling out the $zw$ term each of the remaining eight terms has at least one $\phi$-function or $\Phi$-function as a factor. Finally, since
\[
\int_{\mathbb{R}}\frac{s}{z - s}\,d\mu_n^{(1)}(s) = \frac{1}{z}\Phi_{(\mu_n^{(1)}, \nu_n^{(1)})}(z)(1 + o(1))
\]
and
\[
\int_{\mathbb{R}}\frac{t}{w - t}\,d\mu_n^{(2)}(t) = \frac{1}{w}\Phi_{(\mu_n^{(2)}, \nu_n^{(2)})}(w)(1 + o(1))
\]
as $n \to \infty$ (see \cite{W2011}), we conclude that assertions $(1)$ and $(2)$ are equivalent. Combining everything together, we have
\begin{align*}
\lim_{n \to \infty}k_n\widetilde{\Phi}_{(\mu_n, \nu_n)}(z, w) &= -\lim_{n \to \infty}k_nI_1 - \lim_{n \to \infty}k_nI_2 + \lim_{n \to \infty}k_nI_3 + \lim_{n \to \infty}k_nI_5\\
&\quad +\lim_{n \to \infty}k_n\int_{\mathbb{R}}\frac{s}{z - s}\,d\mu_n^{(1)}(s) + \lim_{n \to \infty}k_n\int_{\mathbb{R}}\frac{t}{w - t}\,d\mu_n^{(2)}(t)\\
&\quad + \lim_{n \to \infty}k_n\int_{\mathbb{R}^2}\frac{st}{(z - s)(w - t)}\,d\mu_n(s, t)\\
&= -\frac{1}{z}\phi_{\nu^{(1)}}(z) - \frac{1}{w}\phi_{\nu^{(2)}}(w)\\
&\quad + \frac{1}{zw}\left(w\phi_{\nu^{(1)}}(z) - w\Phi_{(\mu^{(1)}, \nu^{(1)})}(z) + z\phi_{\nu^{(2)}}(w) - z\Phi_{(\mu^{(2)}, \nu^{(2)})}(w)\right)\\
&\quad +\frac{1}{z}\Phi_{(\mu^{(1)}, \nu^{(1)})}(z) + \frac{1}{w}\Phi_{(\mu^{(2)}, \nu^{(2)})}(w) + \widetilde{\Phi}(z, w)\\
&= \widetilde{\Phi}(z, w),
\end{align*}
first on the open set $\Omega_{\alpha, \beta}$, and then on the whole space $(\mathbb{C}\,\backslash\,\mathbb{R})^2$ by analytic continuation. This completes the proof.
\end{proof}

If the measures in Theorem \ref{LimitThm} are compactly supported, then $(\mu_n, \nu_n)^{\boxplus\boxplus_{\rc}k_n}$ would be the $k_n^{\mathrm{th}}$-fold additive c-bi-free convolution of $(\mu_n, \nu_n)$ which, depending on context, can be viewed either as a measure or a pair of measures where the second component is the $k_n^{\mathrm{th}}$-fold additive bi-free convolution of $\nu_n$. In this case, the following proposition provides some necessary and sufficient conditions for the weak convergence of the sequence $\{(\mu_n, \nu_n)^{\boxplus\boxplus_{\rc}k_n}\}_{n = 1}^\infty$.

\begin{prop}\label{ConvProp2}
Let $\{\nu_n\}_{n = 1}^\infty$ and $\{k_n\}_{n = 1}^\infty$ be sequences of measures and positive integers satisfying the assumptions of Theorem \ref{LimitThm} and let $\{\mu_n\}_{n = 1}^\infty$ be a sequence of compactly supported Borel probability measures on $\mathbb{R}^2$. Assume furthermore that each $\nu_n$ is compactly supported. The sequence $\{\xi_n\}_{n = 1}^\infty$ defined by
\[
\xi_n = \underbrace{(\mu_n, \nu_n) \boxplus\boxplus_{\rc} \cdots \boxplus\boxplus_{\rc} (\mu_n, \nu_n)}_{k_n\,\mathrm{times}},\quad n \geq 1,
\]
converges weakly to a Borel probability measure on $\mathbb{R}^2$ if and only if the sequences $\{\zeta_n^{(1)}\}_{n = 1}^\infty$, $\{\zeta_n^{(2)}\}_{n = 1}^\infty$, and $\{\widetilde{\rho}_n\}_{n = 1}^\infty$, as defined in Theorem \ref{LimitThm}, are weakly convergent. Furthermore, if the sequences $\{\xi_n\}_{n = 1}^\infty$, $\{\zeta_n^{(1)}\}_{n = 1}^\infty$, $\{\zeta_n^{(2)}\}_{n = 1}^\infty$, and $\{\widetilde{\rho}_n\}_{n = 1}^\infty$ converge weakly to $\mu$, $\mu^1$, $\mu^2$, and $\rho$ respectively, then $\mu^{(1)} = \mu^1$, $\mu^{(2)} = \mu^2$, and
\[
G_\mu(z, w) = \frac{1+F_{\nu^{(1)}}(z)F_{\nu^{(2)}}(w)G_\nu(z, w)G_{\rho'}(F_{\nu^{(1)}}(z), F_{\nu^{(2)}}(w))}{F_{\mu^{(1)}}(z)F_{\mu^{(2)}}(w)},
\]
for all $(z, w) \in (\mathbb{C}\,\backslash\,\mathbb{R})^2$ where $d\rho' = \sqrt{1 + s^2}\sqrt{1 + t^2}\,d\rho(s, t)$.
\end{prop}

\begin{proof}
Suppose the sequence $\{\xi_n\}_{n = 1}^\infty$ converges weakly to $\mu$.  Then the sequences $\{\xi_n^{(1)}\}_{n = 1}^\infty$ and $\{\xi_n^{(2)}\}_{n = 1}^\infty$ converge weakly to $\mu^{(1)}$ and $\mu^{(2)}$ respectively. By Lemma \ref{Marginals} and \cite{HW2016}*{Lemma 2.4}, we have $\xi_n^{(1)} = \zeta_n^{(1)}$ and $\xi_n^{(2)} = \zeta_n^{(2)}$ for $n \geq 1$. In view of the assumptions on the sequence $\{\nu_n\}_{n = 1}^\infty$, Proposition \ref{ConvProp} implies
\[
\lim_{n \to \infty}k_n\Phi_{(\mu_n, \nu_n)}(z, w) = \lim_{n \to \infty}\Phi_{(\xi_n, \nu_n^{\boxplus\boxplus k_n})}(z, w) = \Phi_{(\mu, \nu)}(z, w)
\]
on $\Omega_{\alpha, \beta}$. The weak convergence of the sequence $\{\widetilde{\rho}_n\}_{n = 1}^\infty$ then follows from Theorem \ref{LimitThm}.

Conversely, suppose the sequences $\{\zeta_n^{(1)}\}_{n = 1}^\infty$, $\{\zeta_n^{(2)}\}_{n = 1}^\infty$, and $\{\widetilde{\rho}_n\}_{n = 1}^\infty$ are weakly convergent.  Then both $\{\xi_n^{(1)}\}_{n = 1}^\infty$ and $\{\xi_n^{(2)}\}_{n = 1}^\infty$ are tight.  Thus $\{\xi_n\}_{n = 1}^\infty$ is also a tight sequence. By the assumptions on the sequence $\{\nu_n\}_{n = 1}^\infty$, we have $\Phi_{(\xi_n, \nu_n^{\boxplus\boxplus k_n})}(z, w) \to 0$ uniformly in $n$ as $|z|, |w| \to \infty$ non-tangentially. Moreover, the existence of the pointwise limits $\Phi_{(\xi_n, \nu_n^{\boxplus\boxplus k_n})}$ as $n \to \infty$ on $\Omega_{\alpha, \beta}$ is equivalent to the the weak convergence of the sequence $\{\widetilde{\rho}_n\}_{n = 1}^\infty$ by Theorem \ref{LimitThm}. Hence the sequence $\{\xi_n\}_{n = 1}^\infty$ is weakly convergent by Proposition \ref{ConvProp}.

Finally, the equation regarding $G_\mu$ follows from Theorem \ref{LimitThm} as the reduced c-bi-free partial Voiculescu transform of $(\mu, \nu)$ is given by
\[
\widetilde{\Phi}_{(\mu, \nu)}(z, w) = \int_{\mathbb{R}^2}\frac{\sqrt{1 + s^2}\sqrt{1 + t^2}}{(z - s)(w - t)}\,d\rho(s, t)
\]
for all $(z, w) \in (\mathbb{C}\,\backslash\,\mathbb{R})^2$.
\end{proof}

\begin{exam}
Let $\{\mu_n\}_{n = 1}^\infty$ and $\{\nu_n\}_{n = 1}^\infty$ be sequences of compactly supported Borel probability measures on $\bR^2$ and $\{k_n\}_{n = 1}^\infty$ a sequence of positive integers with $\lim_{n \to \infty}k_n = \infty$. Suppose the sequence $\{\nu_n^{\boxplus\boxplus k_n}\}_{n = 1}^\infty$ converges weakly to a bi-free Gaussian distribution $\nu$ (see \cite{HW2016}*{Example 3.4}) with bi-free partial Voiculescu transform
\[
\phi_\nu(z, w) = \frac{\eta_1'}{z} + \frac{\eta_2'}{w} + \frac{a'}{z^2} + \frac{b'}{w^2} + \frac{c'}{zw},
\]
where $\eta_1', \eta_2' \in \bR$, $a', b' \geq 0$, and $|c'| \leq \sqrt{a'b'}$. Suppose furthermore that $\eta_1, \eta_2 \in \bR$, $a, b \geq 0$, and $|c| \leq \sqrt{ab}$, and the sequences $\{\zeta_n^{(1)}\}_{n = 1}^\infty$, $\{\zeta_n^{(2)}\}_{n = 1}^\infty$, and $\{\widetilde{\rho}_n\}_{n = 1}^\infty$, as defined in Theorem \ref{LimitThm}, converge weakly to $\mu^1$, $\mu^2$, and $c\delta_{(0, 0)}$ respectively where the pairs $(\mu^1, \nu^{(1)})$ and $(\mu^2, \nu^{(2)})$ are $\boxplus_{\rc}$-infinitely divisible with c-free Voiculescu transforms determined by $(\eta_1, a\delta_{0})$ and $(\eta_2, b\delta_0)$ respectively, in the sense of \cite{W2011}*{Theorem 4.1}. It follows from Theorem \ref{LimitThm} and Proposition \ref{ConvProp2} that the sequence $\{\xi_n\}_{n = 1}^\infty$, as defined in Proposition \ref{ConvProp2}, converges weakly to a Borel probability measure $\mu$ on $\bR^2$ such that the c-bi-free partial Voiculescu transform of the pair $(\mu, \nu)$ is given by
\begin{align*}
&\Phi_{(\mu, \nu)}  (z, w) \\
&= \frac{1}{z}\left(\eta_1 + \int_{\mathbb{R}}\frac{(1 + sz)a}{z - s}\,d\delta_0(s)\right) + \frac{1}{w}\left(\eta_2 + \int_{\mathbb{R}}\frac{(1 + tw)b}{w - t}\,d\delta_0(t)\right) + \int_{\mathbb{R}^2}\frac{(\sqrt{1 + s^2}\sqrt{1 + t^2})c}{(z - s)(w - t)}\,d\delta_{(0, 0)}(s,t)\\
&= \frac{\eta_1}{z} + \frac{\eta_2}{w} + \frac{a}{z^2} + \frac{b}{w^2} + \frac{c}{zw}.
\end{align*}
Hence the measure $\mu$ has a c-bi-free Gaussian distribution with accompanying distribution $\nu$ as defined in Definition \ref{CBFGauss}. 

The existence of standard c-bi-free Gaussian distributions follow from a similar procedure as described in \cite{HW2016}*{Example 3.4}. Namely, for $a = b = 1$ and $c \in [-1, 1]$, let $Z_1$ and $Z_2$ be two classically independent random variables drawn from the Bernoulli distribution $\fB = (1/2)\delta_{-1} + (1/2)\delta_1$, and let $\mu_n$ be the distribution of the random vector
\[
(X_n, Y_n) = \left(\sqrt{\frac{1 + c}{2n}} \,Z_1 - \sqrt{\frac{1 - c}{2n}} \,Z_2,\sqrt{\frac{1 + c}{2n}} \,Z_1 + \sqrt{\frac{1 - c}{2n}} \,Z_2\right).
\]
Since both $\{[\mu_n^{(1)}]^{*n}\}_{n = 1}^\infty$ and $\{[\mu_n^{(2)}]^{*n}\}_{n = 1}^\infty$ converge weakly to the standard Gaussian distribution $\N(0, 1)$, and since the $*$-domain of attraction of $\N(0, 1)$ coincides with the $\uplus$-domain of attraction of $\fB$ \cite{BP1999} (where $\uplus$ denotes the additive Boolean convolution introduced by Speicher and Woroudi in \cite{SW1997}), both $\{[\mu_n^{(1)}]^{\uplus n}\}_{n = 1}^\infty$ and $\{[\mu_n^{(2)}]^{\uplus n}\}_{n = 1}^\infty$ converge weakly to $\fB$. By \cite{W2011}*{Theorem 3.5}, this is equivalent to the weak convergence of $\{\zeta_n^{(1)}\}_{n = 1}^\infty$ and $\{\zeta_n^{(2)}\}_{n = 1}^\infty$ to $\mu^1$ and $\mu^2$ respectively, such that the c-free Voiculescu transforms of the pairs $(\mu^1, \nu^{(1)})$ and $(\mu^2, \nu^{(2)})$ are given by $\frac{1}{z}$ and $\frac{1}{w}$ respectively. Moreover, the existence of the pointwise limits
\[
\widetilde{\Phi}(z, w) = \lim_{n \to \infty}n\mathbb{E}\left[\frac{X_nY_n}{(z - X_n)(w - Y_n)}\right] = \frac{c}{zw} = \int_{\mathbb{R}^2}\frac{(\sqrt{1 + s^2}\sqrt{1 + t^2})c}{(z - s)(w - t)}\,d\delta_{(0, 0)}
\]
for all $(z, w) \in (\mathbb{C}\,\backslash\,\mathbb{R})^2$ implies that the sequence $\{\widetilde{\rho}_n\}_{n = 1}^\infty$, defined by
\[
d\widetilde{\rho}_n(s, t) = n\frac{st}{\sqrt{1 + s^2}\sqrt{1 + t^2}}\,d\mu_n(s, t),\quad n \geq 1,
\]
converges weakly to $c\delta_{(0, 0)}$ by Theorem \ref{LimitThm}. Therefore the sequence $\{\xi_n\}_{n = 1}^\infty$ converges weakly to the standard c-bi-free Gaussian distribution with covariance matrix
\[
\begin{pmatrix}
1 & c\\
c & 1
\end{pmatrix}
\]
and accompanying distribution $\nu$ by Proposition \ref{ConvProp2}. The general case follows from a shifting and rescaling argument.
\end{exam}

\begin{exam}
Let $\{\nu_n\}_{n = 1}^\infty$ be a sequence of Borel probability measures on $\mathbb{R}^2$ such that both $\{[\nu_n^{(1)}]^{\boxplus n}\}_{n = 1}^\infty$ and $\{[\nu_n^{(2)}]^{\boxplus n}\}_{n = 1}^\infty$ are weakly convergent and $\lim_{n \to \infty}n\phi_{\nu_n} = \phi_\nu$ on $\Omega_{\alpha, \beta}$ for some $\alpha, \beta > 0$. Moreover, let $\lambda \geq 0$ and $\delta_{(0, 0)} \neq \sigma$ be a Borel probability measure on $\mathbb{R}^2$. For $n \geq 1$ consider
\[
\mu_n = \left(1 - \frac{\lambda}{n}\right)\delta_{(0, 0)} + \frac{\lambda}{n}\sigma.
\]
The sequences $\{[(\mu_n^{(1)}, \nu_n^{(1)})]^{\boxplus_{\rc}n}\}_{n = 1}^\infty$ and $\{[(\mu_n^{(2)}, \nu_n^{(2)})]^{\boxplus_{\rc}n}\}_{n = 1}^\infty$ converge weakly to compound c-free Poisson distributions with rate $\lambda$, jump distributions $\sigma^{(1)}$ and $\sigma^{(2)}$, and accompanying distributions $\nu^{(1)}$ and $\nu^{(2)}$ respectively (see \cite{K2007}*{Proposition 6.4}). Moreover, the sequence $\{\widetilde{\rho}_n\}_{n = 1}^\infty$ defined by
\[
d\widetilde{\rho}_n(s, t) = n\frac{st}{\sqrt{1 + s^2}\sqrt{1 + t^2}}\,d\mu_n(s, t),\quad n \geq 1,
\]
converges weakly to the finite signed Borel measure
\[
d\rho(s, t) = \lambda\frac{st}{\sqrt{1 + s^2}\sqrt{1 + t^2}}\,d\sigma(s, t).
\]
By Theorem \ref{LimitThm}, the pointwise limits $\lim_{n \to \infty}n\Phi_{(\mu_n, \nu_n)}(z, w) = \Phi(z, w)$ exist for all $(z, w) \in \Omega_{\alpha, \beta}$ and the function $\Phi$ can be written as
\begin{align*}
\Phi(z, w) &= \lambda\int_{\mathbb{R}^2}\frac{s}{z - s}\,d\sigma(s, t) + \lambda\int_{\mathbb{R}^2}\frac{t}{w - t}\,d\sigma(s, t) + \lambda\int_{\mathbb{R}^2}\frac{st}{(z - s)(w - t)}\,d\sigma(s, t)\\
&= \lambda\int_{\mathbb{R}^2}\frac{-(z - s)(w - t) + zw}{(z - s)(w - t)}\,d\sigma(s, t)\\
&= -\lambda + \lambda\int_{\mathbb{R}^2}\frac{zw}{(z - s)(w - t)}\,d\sigma(s, t)\\
&= \lambda(zwG_\sigma(z, w) - 1)
\end{align*}
for all $(z, w) \in (\mathbb{C}\,\backslash\,\mathbb{R})^2$. Note that the function $\Phi$ is same as the c-bi-free partial Voiculescu transform of the pair $(\pi_{\lambda, \sigma, \nu}, \nu)$ where $\pi_{\lambda, \sigma, \nu}$ denotes the compound c-bi-free Poisson distribution with rate $\lambda$, jump distribution $\sigma$, and accompanying distribution $\nu$ as defined in Definition \ref{CompCBFPoisson}. The existence of such a distribution can be shown analytically by the same truncation method and limiting process used in \cite{HW2016}*{Example 3.5} to show the existence of the compound bi-free Poisson distribution with rate $\lambda$ and jump distribution $\sigma$.
\end{exam}

\subsection{Conditionally bi-free additively infinitely divisible distributions}

As mentioned earlier, the operations $\boxplus\boxplus$ and $\boxplus\boxplus_{\rc}$ have not been defined for all Borel probability measures on $\bR^2$ yet. In order to discuss infinite divisibility, we take the idea from \cite{HW2016}*{Definition 3.7} and define it in terms of the corresponding linearizing transforms. To this end, we need the following result.

\begin{thm}\label{Existence}
Let $\{\mu_n\}_{n = 1}^\infty$, $\{\nu_n\}_{n = 1}^\infty$, and $\{k_n\}_{n = 1}^\infty$ be as in Theorem \ref{LimitThm}. If the pointwise limits $\lim_{n \to \infty}k_n\Phi_{(\mu_n, \nu_n)}(z, w) = \Phi(z, w)$ exist for all $(z, w) \in \Omega_{\alpha, \beta}$, then there exists a unique Borel probability measure $\mu$ on $\mathbb{R}^2$ such that $\Phi = \Phi_{(\mu, \nu)}$ on $\Omega_{\alpha, \beta}$.
\end{thm}

\begin{proof}
The uniqueness part follows from Corollary \ref{Unique}.  Thus we show the existence part. As shown in the proof of \cite{HW2016}*{Theorem 3.2}, the sequences $\{\rho_{1; n}\}_{n = 1}^\infty$ and $\{\rho_{2; n}\}_{n = 1}^\infty$, defined by
\[
d\rho_{1; n}(s, t) = k_n\frac{s^2}{1 + s^2}\,d\mu_n(s, t)\qand d\rho_{2; n}(s, t) = k_n\frac{t^2}{1 + t^2}\,d\mu_n(s, t),\quad n \geq 1,
\]
are tight and uniformly bounded sequences of finite positive Borel measures on $\mathbb{R}^2$. By dropping to subsequences if necessary, we may assume they are both weakly convergent. Let
\[
\eta_{1; n} = k_n\int_{\mathbb{R}^2}\frac{s}{1 + s^2}\,d\mu_n(s, t)\qand \eta_{2; n} = k_n\int_{\mathbb{R}^2}\frac{t}{1 + t^2}\,d\mu_n(s, t),\quad n \geq 1.
\]
By \cite{W2011}*{Theorem 3.5} the assumptions that the sequences $\{(\mu_n^{(1)}, \nu_n^{(1)})^{\boxplus_{\rc}k_n}\}_{n = 1}^\infty$ and $\{(\mu_n^{(2)}, \nu_n^{(2)})^{\boxplus_{\rc}k_n}\}_{n = 1}^\infty$ converge weakly to $(\mu^{(1)}, \nu^{(1)})$ and $(\mu^{(2)}, \nu^{(2)})$ respectively implies that
\[
\Phi_{(\mu^{(1)}, \nu^{(1)})}(z) = \lim_{n \to \infty}\left(\eta_{1; n} + \int_{\mathbb{R}^2}\frac{1 + sz}{z - s}\,d\rho_{1; n}(s, t)\right)
\]
and
\[
\Phi_{(\mu^{(2)}, \nu^{(2)})}(w) = \lim_{n \to \infty}\left(\eta_{2; n} + \int_{\mathbb{R}^2}\frac{1 + tw}{w - t}\,d\rho_{2; n}(s, t)\right),
\]
as well as the existence of the limits $\lim_{n \to \infty}\eta_{1; n}$ and $\lim_{n \to \infty}\eta_{2; n}$. The existence of the pointwise limits $\lim_{n \to \infty}k_n\Phi_{(\mu_n, \nu_n)}(z, w) = \Phi(z, w)$ on $\Omega_{\alpha, \beta}$ is equivalent to the weak convergence of the sequence $\{\widetilde{\rho}_n\}_{n = 1}^\infty$ by Theorem \ref{LimitThm}. Hence if $\pi_{k_n, \mu_n, \nu}$ denotes the compound c-bi-free Poisson distribution with rate $k_n$, jump distribution $\mu_n$, and accompanying distribution $\nu$, then
\begin{align*}
\Phi(z, w) &= \frac{1}{z}\Phi_{(\mu^{(1)}, \nu^{(1)})}(z) + \frac{1}{w}\Phi_{(\mu^{(2)}, \nu^{(2)})}(w) + \lim_{n \to \infty}k_n\int_{\mathbb{R}^2}\frac{st}{(z - s)(w - t)}\,d\mu_n(s, t)\\
&= \frac{1}{z}\left(\lim_{n \to \infty}k_n\int_{\mathbb{R}^2}\frac{s}{1 + s^2}\,d\mu_n(s, t) + \lim_{n \to \infty}k_n\int_{\mathbb{R}^2}\frac{(1 + sz)s^2}{(z - s)(1 + s^2)}\,d\mu_n(s, t)\right)\\
&\quad +\frac{1}{w}\left(\lim_{n \to \infty}k_n\int_{\mathbb{R}^2}\frac{t}{1 + t^2}\,d\mu_n(s, t) + \lim_{n \to \infty}k_n\int_{\mathbb{R}^2}\frac{(1 + tw)t^2}{(w - t)(1 + t^2)}\,d\mu_n(s, t)\right)\\
&\quad +\lim_{n \to \infty}k_n\int_{\mathbb{R}^2}\frac{st}{(z - s)(w - t)}\,d\mu_n(s, t)\\
&= \lim_{n \to \infty}k_n\left(\int_{\mathbb{R}^2}\frac{s}{z - s}\,d\mu_n(s, t) + \int_{\mathbb{R}^2}\frac{t}{w - t}\,d\mu_n(s, t) + \int_{\mathbb{R}^2}\frac{st}{(z - s)(w - t)}\,d\mu_n(s, t)\right)\\
&= \lim_{n \to \infty}k_n\left(\int_{\mathbb{R}^2}\frac{zw}{(z - s)(w - t)}\,d\mu_n(s, t) - 1\right)\\
&= \lim_{n \to \infty}k_n(zwG_{\mu_n}(z, w) - 1)\\
&= \lim_{n \to \infty}\Phi_{(\pi_{k_n, \mu_n, \nu}, \nu)}(z, w)
\end{align*}
for all $(z, w) \in (\mathbb{C}\,\backslash\,\mathbb{R})^2$.  By the same estimates used in the proof of \cite{HW2016}*{Theorem 3.6} to show that the family of compound bi-free Poisson distributions with rates $k_n$ and jump distributions $\mu_n$ is tight, the same property also holds for the family $\{\pi_{k_n, \mu_n, \nu}\}_{n = 1}^\infty$ and $\Phi_{(\pi_{k_n, \mu_n, \nu}, \nu)}(z, w) \to 0$ uniformly in $n$ as $|z|, |w| \to \infty$ non-tangentially. Therefore by Proposition \ref{ConvProp} (with $\{\nu_n\}_{n = 1}^\infty$ in the assumption being the constant sequence $\{\nu\}$), the sequence $\{\pi_{k_n, \mu_n, \nu}\}_{n = 1}^\infty$ converges weakly to a Borel probability measure $\mu$ on $\mathbb{R}^2$ and $\Phi = \Phi_{(\mu, \nu)}$ on $(\mathbb{C}\,\backslash\,\mathbb{R})^2$.
\end{proof}

\begin{defn}
A pair $(\mu, \nu)$ of Borel probability measures on $\mathbb{R}^2$ is said to be \emph{$\boxplus\boxplus_{\rc}$-infinitely divisible} if for every $n \in \mathbb{N}$ there exists a pair $(\mu_n, \nu_n)$ of Borel probability measures on $\mathbb{R}^2$ such that $\phi_\nu = n\phi_{\nu_n}$ and $\Phi_{(\mu, \nu)} = n\Phi_{(\mu_n, \nu_n)}$ on a common domain where the involved functions are defined. Note that $\nu$ is $\boxplus\boxplus$-infinitely divisible in this case.
\end{defn}

Let $\nu$ be a $\boxplus\boxplus$-infinitely divisible Borel probability measure on $\bR^2$. In view of the definition above, it is easy to see that the pair $(\mu, \nu)$ is $\boxplus\boxplus_{\rc}$-infinitely divisible if $\mu$ is a point mass or $\mu$ is the product measure of its marginal distributions $\mu^{(1)}$ and $\mu^{(2)}$ such that $(\mu^{(1)}, \nu^{(1)})$ and $(\mu^{(2)}, \nu^{(2)})$ are $\boxplus_{\rc}$-infinitely divisible. More generally, Theorem \ref{Existence} implies that limits of $k_n\Phi_{(\mu_n, \nu_n)}$ is the class of c-bi-free partial Voiculescu transforms of $\boxplus\boxplus_{\rc}$-infinitely divisible pairs.

\begin{thm}\label{CBFInfDiv}
Let $(\mu, \nu)$ be a pair of Borel probability measures on $\mathbb{R}^2$ such that $\nu$ is $\boxplus\boxplus$-infinitely divisible. The pair $(\mu, \nu)$ is $\boxplus\boxplus_{\rc}$-infinitely divisible if and only if there exist a sequence $\{\mu_n\}_{n = 1}^\infty$ of Borel probability measures on $\mathbb{R}^2$ and a sequence $\{k_n\}_{n = 1}^{\infty}$ of positive integers with $\lim_{n \to \infty}k_n = \infty$ such that the sequences $\{\zeta_n^{(1)}\}_{n = 1}^\infty$ and $\{\zeta_n^{(2)}\}_{n = 1}^\infty$, as defined in Theorem \ref{LimitThm}, converge weakly to $\mu^{(1)}$ and $\mu^{(2)}$ respectively, and $\lim_{n \to \infty}k_n\Phi_{(\mu_n, \nu_n)} = \Phi_{(\mu, \nu)}$ on $\Omega_{\alpha, \beta}$.
\end{thm}

\begin{proof}
If the pair $(\mu, \nu)$ is $\boxplus\boxplus_{\rc}$-infinitely divisible, then for every $n \in \mathbb{N}$, there exists a pair $(\widetilde{\mu}_n, \widetilde{\nu}_n)$ of Borel probability measures on $\mathbb{R}^2$ such that $\Phi_{(\mu, \nu)} = n\Phi_{(\widetilde{\mu}_n, \widetilde{\nu}_n)}$ on $\Omega_{\alpha, \beta}$ for some $\alpha, \beta > 0$. For $n \geq 1$, simply choose $\mu_n = \widetilde{\mu}_n$, $\nu_n = \widetilde{\nu}_n$, and $k_n = n$.

Conversely, the assumptions on $\nu$, $\mu^{(1)}$, and $\mu^{(2)}$ imply $\nu^{(1)}$ and $\nu^{(2)}$ are $\boxplus$-infinitely divisible, $(\mu^{(1)}, \nu^{(1)})$ and $(\mu^{(2)}, \nu^{(2)})$ are $\boxplus_{\rc}$-infinitely divisible. For $m \geq 1$, let $\nu_m^{(1)}$, $\nu_{m}^{(2)}$, $\mu_m^{(1)}$, and $\mu_m^{(2)}$ be Borel probability measures on $\mathbb{R}$ such that
\[
\nu^{(i)} = \underbrace{\nu_{m}^{(i)} \boxplus \cdots \boxplus \nu_{m}^{(i)}}_{m\,\mathrm{times}} \qand (\mu^{(i)}, \nu^{(i)}) = \underbrace{(\mu_m^{(i)}, \nu_m^{(i)}) \boxplus_{\rc} \cdots \boxplus_{\rc} (\mu_m^{(i)}, \nu_m^{(i)})}_{m\,\mathrm{times}},\quad i = 1, 2.
\]
Consider the sequence $\{j_n\}_{n = 1}^\infty$ of positive integers defined by $j_n = [k_n/m]$ where $[x]$ denotes the integer part of $x$.  Then $\lim_{n \to \infty}j_n = \infty$ and the sequences $\{[\nu_n^{(1)}]^{\boxplus j_n}\}_{n = 1}^\infty$, $\{[\nu_n^{(2)}]^{\boxplus j_n}\}_{n = 1}^\infty$, $\{[(\mu_n^{(1)}, \nu_n^{(1)})]^{\boxplus_{\rc}j_n}\}_{n = 1}^\infty$, and $\{[(\mu_n^{(2)}, \nu_n^{(2)})]^{\boxplus_{\rc}j_n}\}_{n = 1}^\infty$ converge weakly to $\nu_m^{(1)}$, $\nu_m^{(2)}$, $(\mu_m^{(1)}, \nu_m^{(1)})$, and $(\mu_m^{(2)}, \nu_m^{(2)})$ respectively. Since the pointwise limits
\[
\lim_{n \to \infty}j_n\Phi_{(\mu_n, \nu_n)}(z, w) = \lim_{n \to \infty}[k_n/m]\Phi_{(\mu_n, \nu_n)}(z, w) = \frac{1}{m}\Phi_{(\mu, \nu)}(z, w)
\]
exist for all $(z, w) \in (\mathbb{C}\,\backslash\,\mathbb{R})^2$ by assumption, Theorem \ref{Existence} and the assumption that $\nu$ is $\boxplus\boxplus$-infinitely divisible together imply the existence of a pair $(\mu_m, \nu_m)$ of Borel probability measures on $\mathbb{R}^2$ such that $\phi_\nu = m\phi_{\nu_m}$ and $\Phi_{(\mu, \nu)} = m\Phi_{(\mu_m, \nu_m)}$ on $\Omega_{\alpha, \beta}$.
\end{proof}

As seen in the proof of Theorem \ref{Existence}, we also have the following Poisson type characterization of $\boxplus\boxplus_{\rc}$-infinite divisibility.

\begin{prop}
Let $(\mu, \nu)$ be a pair of Borel probability measures on $\bR^2$ such that $\nu$ is $\boxplus\boxplus$-infinitely divisible. The following are equivalent:
\begin{enumerate}[$\qquad(1)$]
\item The pair $(\mu, \nu)$ is $\boxplus\boxplus_{\rc}$-infinitely divisible.

\item There exist a sequence $\{\mu_n\}_{n = 1}^\infty$ of Borel probability measures on $\bR^2$ and a sequence $\{k_n\}_{n = 1}^\infty$ of positive integers with $\lim_{n \to \infty}k_n = \infty$ such that the sequence of compound c-bi-free Poission distributions with rate $k_n$, jump distribution $\mu_n$, and accompanying distribution $\nu$ converge weakly to $\mu$.
\end{enumerate}
\end{prop}

\begin{proof}
If assertion $(2)$ holds, then $(\mu, \nu)$ is $\boxplus\boxplus_{\rc}$-infinitely divisible as $(\pi_{k_n, \mu_n, \nu}, \nu)$ is $\boxplus\boxplus_{\rc}$-infinitely divisible for all $n \geq 1$. On the other hand, the converse follows from Theorem \ref{CBFInfDiv} and the proof of Theorem \ref{Existence}, along with the discussion in \cite{HW2016} preceding \cite{HW2016}*{Proposition 3.11} that the c-bi-free Poisson approximation to $\mu$ holds without passing to subsequences.
\end{proof}

\subsection{A c-bi-free L\'{e}vy-Hin\v{c}in formula}

If $(\mu, \nu)$ is a $\boxplus\boxplus_{\rc}$-infinitely divisible pair, then the marginal pairs $(\mu^{(1)}, \nu^{(1)})$ and $(\mu^{(2)}, \nu^{(2)})$ are $\boxplus_{\rc}$-infinitely divisible. Recall from \cite{W2011}*{Theorem 4.1} that there exist real numbers $\eta_1$, $\eta_2$ and finite positive Borel measures $\rho_1$, $\rho_2$ on $\bR$ such that $\Phi_{(\mu^{(1)}, \nu^{(1)})}$ and $\Phi_{(\mu^{(2)}, \nu^{(2)})}$ admit c-free L\'{e}vy-Hin\v{c}in representations determined by $(\eta_1, \rho_1)$ and $(\eta_2, \rho_2)$ respectively. On the other hand, we have
\[
\Phi_{(\mu, \nu)}(z, w) = \frac{1}{z}\Phi_{(\mu^{(1)}, \nu^{(1)})}(z) + \frac{1}{w}\Phi_{(\mu^{(2)}, \nu^{(2)})}(w) + \int_{\bR^2}\frac{\sqrt{1 + s^2}\sqrt{1 + t^2}}{(z - s)(w - t)}\,d\rho(s, t),
\]
where $\rho$ is the weak limit of the sequence $\{\widetilde{\rho}_n\}_{n = 1}^\infty$ as defined in Theorem \ref{LimitThm}. Thus, for a fixed $\boxplus\boxplus$-infinitely divisible measure $\nu$, every $\boxplus\boxplus_{\rc}$-infinitely divisible pair $(\mu, \nu)$ has a unique quintuple $(\eta_1, \eta_2, \rho_1, \rho_2, \rho)$ associated to $\mu$. In fact, as we will see in the next result, $\rho_1$, $\rho_2$, and $\rho$ cannot be arbitrary. In particular, if $(\mu', \nu')$ and $(\mu'', \nu'')$ are $\boxplus\boxplus_{\rc}$-infinitely divisible pairs with quintuples $(\eta_1', \eta_2', \rho_1', \rho_2', \rho')$ and $(\eta_1'', \eta_2'', \rho_1'', \rho_2'', \rho'')$ associated to $\mu'$ and $\mu''$ respectively, then we may define $(\mu', \nu') \boxplus\boxplus_{\rc} (\mu'', \nu'')$ to be the pair $(\mu, \nu)$ where $\nu$ is the additive bi-free convolution of $\nu_1$ and $\nu_2$ as defined in \cite{HW2016}*{Proposition 3.12}, and $\mu$ is the Borel probability measure on $\bR^2$ with quintuple $(\eta_1' + \eta_1'', \eta_2' + \eta_2'', \rho_1' + \rho_1', \rho_2' + \rho_2'', \rho' + \rho'')$ associated to it.

\begin{prop}\label{CBFLevyHincin}
Let $\nu$ be a $\boxplus\boxplus$-infinitely divisible Borel probability measure on $\mathbb{R}^2$ and $\Phi$ be an analytic function on $\Omega_{\alpha, \beta}$ for some $\alpha, \beta > 0$. The function $\Phi$ is the c-bi-free partial Voiculescu transform of some $\boxplus\boxplus_{\rc}$-infinitely divisible pair $(\mu, \nu)$ of Borel probability measures on $\mathbb{R}^2$ if and only if there exists a unique quintuple $(\eta_1, \eta_2, \rho_1, \rho_2, \rho)$ where $\eta_1$ and $\eta_2$ are real numbers, $\rho_1$ and $\rho_2$ are finite positive Borel measures on $\mathbb{R}^2$, $\rho$ is a finite signed Borel measure on $\mathbb{R}^2$ such that
\begin{enumerate}[$\qquad(a)$]
\item $\frac{t}{\sqrt{1 + t^2}}\,d\rho_1(s, t) = \frac{s}{\sqrt{1 + s^2}}\,d\rho(s, t)$,

\item $\frac{s}{\sqrt{1 + s^2}}\,d\rho_2(s, t) = \frac{t}{\sqrt{1 + t^2}}\,d\rho(s, t)$,

\item $|\rho(\{(0, 0)\})|^2 \leq \rho_1(\{(0, 0)\})\rho_2(\{(0, 0)\})$,
\end{enumerate}
and the function $\Phi$ can be continued analytically to $(\mathbb{C}\,\backslash\,\mathbb{R})^2$ via
\begin{align*}
\Phi(z, w) &= \frac{1}{z}\left(\eta_1 + \int_{\mathbb{R}^2}\frac{1 + sz}{z - s}\,d\rho_1(s, t)\right) + \frac{1}{w}\left(\eta_2 + \int_{\mathbb{R}^2}\frac{1 + tw}{w - t}\,d\rho_2(s, t)\right) + \int_{\mathbb{R}^2}\frac{\sqrt{1 + s^2}\sqrt{1 + t^2}}{(z - s)(w - t)}\,d\rho(s, t).
\end{align*}
Moreover, the marginal pairs $(\mu^{(1)}, \nu^{(1)})$ and $(\mu^{(2)}, \nu^{(2)})$ are $\boxplus_{\rc}$-infinitely divisible with c-free Voiculescu transforms
\[
\Phi_{(\mu^{(1)}, \nu^{(1)})}(z) = \eta_1 + \int_{\mathbb{R}}\frac{1 + sz}{z - s}\,d\rho_1^{(1)}(s)\qand \Phi_{(\mu^{(2)}, \nu^{(2)})}(w) = \eta_2 + \int_{\mathbb{R}}\frac{1 + tw}{w - t}\,d\rho_2^{(2)}(t)
\]
for all $z, w \in \mathbb{C}\,\backslash\,\mathbb{R}$.
\end{prop}

\begin{proof}
If $\Phi = \Phi_{(\mu, \nu)}$ for some $\boxplus\boxplus_{\rc}$-infinitely divisible pair $(\mu, \nu)$, then the assertion follows from Theorems \ref{LimitThm} and \ref{CBFInfDiv}. The proof of the converse is practically the same as the proof of \cite{HW2016}*{Theorem 4.3} by decomposing the function $\Phi$ into the sum of three c-bi-free partial Voiculescu transforms of $\boxplus\boxplus_{\rc}$-infinitely divisible pairs, thus we only sketch the main arguments.

As seen in the proof of \cite{HW2016}*{Theorem 4.3}, the bi-free partial Voiculescu transform $\phi_\nu$ of $\nu$, which has a similar form as $\Phi$, can be written as $\phi_\nu = \phi_{\nu_1} + \phi_{\nu_2} + \phi_{\nu_3}$ where $\nu_1$, $\nu_2$, and $\nu_3$ are $\boxplus\boxplus$-infinitely divisible. More precisely, $\nu_1$ has a centred bi-free Gaussian distribution, $\nu_2$ is the product measure of two $\boxplus$-infinitely divisible measures, and $\nu_3$ is the weak limit of a $\boxplus\boxplus$-infinitely divisible sequence $\{\nu_{3, n}\}_{n = 1}^\infty$. Let
\[
S = \{(s, 0) \in \mathbb{R}^2 : s \neq 0\},\quad T = \{(0, t) \in \mathbb{R}^2 : t \neq 0\},\qand U = \mathbb{R}^2\,\backslash\,(S \cup T \cup \{(0, 0)\})
\]
be a decomposition of $\mathbb{R}^2\,\backslash\,\{(0, 0)\}$ into three Borel sets. Moreover, decompose the measures $\rho$, $\rho_1$, and $\rho_2$ by
\[
\rho = \rho(\{(0, 0)\})\delta_{(0, 0)} + \rho|_S + \rho|_T + \rho|_U\qand \rho_i = \rho_i(\{(0, 0)\})\delta_{(0, 0)} + \rho_i|_S + \rho_i|_T + \rho_i|_U,\quad i = 1, 2.
\]
One then uses the conditions $(a)$, $(b)$, and $(c)$ to check that $\rho_1|_T$, $\rho_2|_S$, $\rho|_S$, and $\rho|_T$ are in fact the zero measure. Therefore the function $\Phi$ can be decomposed as
\begin{align*}
\Phi(z, w) &= \frac{\rho_1(\{(0, 0)\})}{z^2} + \frac{\rho_2(\{(0, 0)\})}{w^2} + \frac{\rho(\{(0, 0)\})}{zw}\\
&\quad + \frac{1}{z}\left(\eta_1 + \int_{\mathbb{R}\,\backslash\,\{0\}}\frac{1 + sz}{z - s}\,d\rho_1|_S(s)\right) + \frac{1}{w}\left(\eta_2 + \int_{\mathbb{R}\,\backslash\,\{0\}}\frac{1 + tw}{w - t}\,d\rho_2|_T(t)\right)\\
&\quad + \int_{U}\frac{1 + sz}{z^2 - sz}\,d\rho_1|_U(s, t) + \int_{U}\frac{1 + tw}{w^2 - tw}\,d\rho_2|_U(s, t) + \int_{U}\frac{\sqrt{1 + s^2}\sqrt{1 + t^2}}{(z - s)(w - t)}\,d\rho|_U(s, t)
\end{align*}
on $(\mathbb{C}\,\backslash\,\mathbb{R})^2$. In the decomposition above, the first line corresponds to the c-bi-free partial Voiculescu transform of the centred c-bi-free Gaussian distribution with covariance matrix
\[
\begin{pmatrix}
\rho_1(\{(0, 0)\}) & \rho(\{(0, 0)\})\\
\rho(\{(0, 0)\}) & \rho_2(\{(0, 0)\})
\end{pmatrix}
\]
and accompanying distribution $\nu_1$.  The second line corresponds to the c-bi-free partial Voiculescu transform of $(\mu_2, \nu_2)$ where $\mu_2$ is the product measure of its marginal distributions $\mu_2^{(1)}$ and $\mu_2^{(2)}$ such that the marginal pairs $(\mu_2^{(1)}, \nu_2^{(1)})$ and $(\mu_2^{(2)}, \nu_2^{(2)})$ are $\boxplus_{\rc}$-infinitely divisible pairs determined by $(\eta_1, \rho_1|_S)$ and $(\eta_2, \rho_2|_T)$ respectively.  Finally the third line corresponds to the c-bi-free partial Voiculescu transform of $(\mu_3, \nu_3)$ where $\mu_3$ is the weak limit of a sequence $\{\mu_{3, n}\}_{n = 1}^\infty$ such that $(\mu_{3, n}, \nu_{3, n})$ is $\boxplus\boxplus_{\rc}$-infinitely divisible for all $n \geq 1$. For more details, see the proof of \cite{HW2016}*{Theorem 4.3}.
\end{proof}

\begin{thm}\label{Unbound}
Let $(\mu, \nu)$ be a pair of Borel probability measures on $\mathbb{R}^2$ such that $\nu$ is $\boxplus\boxplus$-infinitely divisible. The following are equivalent:
\begin{enumerate}[$\qquad(1)$]
\item The pair $(\mu, \nu)$ is $\boxplus\boxplus_{\rc}$-infinitely divisible.

\item There exists a unique quintuple $(\eta_1, \eta_2, \rho_1, \rho_2, \rho)$ where $\eta_1$ and $\eta_2$ are real numbers, $\rho_1$ and $\rho_2$ are finite positive Borel measures on $\mathbb{R}^2$, $\rho$ is a finite signed Borel measure on $\mathbb{R}^2$ such that
\begin{enumerate}[$\quad(a)$]
\item $\frac{t}{\sqrt{1 + t^2}}\,d\rho_1(s, t) = \frac{s}{\sqrt{1 + s^2}}\,d\rho(s, t)$,

\item $\frac{s}{\sqrt{1 + s^2}}\,d\rho_2(s, t) = \frac{t}{\sqrt{1 + t^2}}\,d\rho(s, t)$,

\item $|\rho(\{(0, 0)\})|^2 \leq \rho_1(\{(0, 0)\})\rho_2(\{(0, 0)\})$,
\end{enumerate}
and the c-bi-free partial Voiculescu transform $\Phi_{(\mu, \nu)}$ of $(\mu, \nu)$ can be continued analytically to $(\mathbb{C}\,\backslash\,\mathbb{R})^2$ via
\begin{align}\label{CBFLHRep}
\begin{split}
\Phi_{(\mu, \nu)}(z, w) &= \frac{1}{z}\left(\eta_1 + \int_{\mathbb{R}^2}\frac{1 + sz}{z - s}\,d\rho_1(s, t)\right) + \frac{1}{w}\left(\eta_2 + \int_{\mathbb{R}^2}\frac{1 + tw}{w - t}\,d\rho_2(s, t)\right)\\
&+ \int_{\mathbb{R}^2}\frac{\sqrt{1 + s^2}\sqrt{1 + t^2}}{(z - s)(w - t)}\,d\rho(s, t).
\end{split}
\end{align}
Moreover, the marginal pairs $(\mu^{(1)}, \nu^{(1)})$ and $(\mu^{(2)}, \nu^{(2)})$ are $\boxplus_{\rc}$-infinitely divisible with c-free Voiculescu transforms
\[
\Phi_{(\mu^{(1)}, \nu^{(1)})}(z) = \eta_1 + \int_{\mathbb{R}}\frac{1 + sz}{z - s}\,d\rho_1^{(1)}(s)\qand \Phi_{(\mu^{(2)}, \nu^{(2)})}(w) = \eta_2 + \int_{\mathbb{R}}\frac{1 + tw}{w - t}\,d\rho_2^{(2)}(t)
\]
for all $z, w \in \mathbb{C}\,\backslash\,\mathbb{R}$.

\item There exists a weakly continuous $\boxplus\boxplus_{\rc}$-semigroup $\{(\mu_x, \nu_x)\}_{x \geq 0}$ of pairs of Borel probability measures on $\mathbb{R}^2$ such that $(\mu_0, \nu_0) = (\delta_{(0, 0)}, \delta_{(0, 0)})$ and $(\mu_1, \nu_1) = (\mu, \nu)$.
\end{enumerate}
\end{thm}

\begin{proof}
The fact that assertion $(1)$ implies assertion $(2)$ follows from Theorems \ref{LimitThm} and \ref{CBFInfDiv}.  The fact that assertion $(2)$ implies assertion $(1)$ follows from Corollary \ref{Unique} and Proposition \ref{CBFLevyHincin}. It is also clear that assertion $(3)$ implies assertion $(1)$. To finish the proof, suppose assertions $(1)$ and $(2)$ hold. As shown in \cite{HW2016}, there exists a weakly continuous $\boxplus\boxplus$-semigroup $\{\nu_x\}_{x \geq 0}$ of Borel probability measures on $\mathbb{R}^2$ such that $\nu_0 = \delta_{(0, 0)}$ and $\nu_1 = \nu$. Fix $x > 0$ and let $\Phi_x$ be the function defined by
\begin{align*}
\Phi_x(z, w) &= \frac{1}{z}\left(\eta_{1; x} + \int_{\mathbb{R}^2}\frac{1 + sz}{z - s}\,d\rho_{1; x}(s, t)\right) + \frac{1}{w}\left(\eta_{2; x} + \int_{\mathbb{R}^2}\frac{1 + tw}{w - t}\,d\rho_{2; x}(s, t)\right)\\
&+ \int_{\mathbb{R}^2}\frac{\sqrt{1 + s^2}\sqrt{1 + t^2}}{(z - s)(w - t)}\,d\rho_x(s, t),\quad (z, w) \in (\mathbb{C}\,\backslash\,\mathbb{R})^2,
\end{align*}
where $(\eta_{1; x}, \eta_{2; x}, \rho_{1; x}, \rho_{2; x}, \rho_x) = x\cdot(\eta_1, \eta_2, \rho_1, \rho_2, \rho)$ with component-wise multiplication. By Proposition \ref{CBFLevyHincin}, there exists a Borel probability measure $\mu_x$ on $\mathbb{R}^2$ such that $\Phi_{(\mu_x, \nu_x)} = \Phi_x = x\cdot\Phi_{(\mu_1, \nu_1)}$. It is easy to check that the $\boxplus\boxplus_{\rc}$-semigroup $\{(\mu_x, \nu_x)\}_{x \geq 0}$ obtained this way has the desired properties.
\end{proof}

In comparison with other additive convolutions, especially the bi-free case (see \cite{HW2016}*{Theorem 4.3}), it makes sense to refer to equation \eqref{CBFLHRep} as the c-bi-free L\'{e}vy-Hin\v{c}in representation of the $\boxplus\boxplus_{\rc}$-infinitely divisible pair $(\mu, \nu)$. On the other hand, if both $\mu$ and $\nu$ are compactly supported, then we may also characterize its c-bi-free partial $\R$-transform $\R_{(\mu, \nu)}$ by analyzing the Cauchy transforms of the corresponding $\boxplus\boxplus_{\rc}$-semigroup $\{(\mu_x, \nu_x)\}_{x \geq 0}$.

\begin{thm}\label{CompSupp}
Let $(\mu, \nu)$ be a pair of compactly supported $\boxplus\boxplus_{\rc}$-infinitely divisible Borel probability measures on $\bR^2$ and let $\{(\mu_x, \nu_x)\}_{x \geq 0}$ be the $\boxplus\boxplus_{\rc}$-semigroup generated by $(\mu, \nu)$. There exist a unique triple $(\rho_1, \rho_2, \rho)$ where $\rho_1$ and $\rho_2$ are compactly supported finite positive Borel measures on $\bR^2$, $\rho$ is a compactly supported finite signed Borel measure on $\bR^2$ such that the following hold:
\begin{enumerate}[$\qquad(1)$]
\item The sequences of measures $(s^2/\varepsilon)\,d\mu_\varepsilon(s, t)$, $(t^2/\varepsilon)\,d\mu_\varepsilon(s, t)$, and $(st/\varepsilon)\,d\mu_\varepsilon(s, t)$ converge weakly to $\rho_1$, $\rho_2$, and $\rho$ respectively, as $\varepsilon \to 0^+$, and
\[
\lim_{\varepsilon \to 0^+}\frac{1}{\varepsilon}\int_{\bR^2}s\,d\mu_\varepsilon(s, t) = \K_{1, 0}(\mu, \nu)\qand \lim_{\varepsilon \to 0^+}\frac{1}{\varepsilon}\int_{\bR^2}t\,d\mu_\varepsilon(s, t) = \K_{0, 1}(\mu, \nu).
\]
\item The c-bi-free partial $\R$-transform of $(\mu, \nu)$ is given by
\begin{align*}
\R_{(\mu, \nu)}(z, w) &= z\left(\K_{1, 0}(\mu, \nu) + \int_{\bR^2}\frac{z}{1 - sz}\,d\rho_1(s, t)\right) + w\left(\K_{0, 1}(\mu, \nu) + \int_{\bR^2}\frac{w}{1 - tw}\,d\rho_2(s, t)\right)\\
&+ \int_{\bR^2}\frac{zw}{(1 - sz)(1 - tw)}\,d\rho(s, t)
\end{align*}
for all $(z, w) \in (\bC\,\backslash\,\bR)^2 \cup \{(0, 0)\}$.

\item The measures $\rho_1$, $\rho_2$, and $\rho$ satisfy
\[
t\,d\rho_1(s, t) = s\,d\rho(s, t),\quad s\,d\rho_2(s, t) = t\,d\rho(s, t),\quad |\rho(\{(0, 0)\})|^2 \leq \rho_1(\{(0, 0)\})\rho_2(\{(0, 0)\}),
\]
and the total masses are given by
\[
\rho_1(\bR^2) = \int_{\bR}s^2\,d\mu^{(1)}(s),\quad \rho_2(\bR^2) = \int_{\bR}t^2\,d\mu^{(2)}(t),\quad \rho(\bR^2) = \int_{\bR^2}st\,d\mu(s, t).
\]
\end{enumerate}
\end{thm}

\begin{proof}
Note that the statements are exactly the same as the bi-free case (\cite{HW2016}*{Theorem 4.2}) for $\nu$ a compactly supported $\boxplus\boxplus$-infinitely divisible Borel probability measure on $\bR^2$ with corresponding $\boxplus\boxplus$-semigroup $\{\nu_x\}_{x \geq 0}$ generated by $\nu$. Moreover, the proof of \cite{HW2016}*{Theorem 4.2} starts with the result of \cite{HW2016}*{Proposition 4.1}, which states that the limit
\[
\lim_{t \to 0^+}\frac{G_{\nu_t}(1/z, 1/w) - G_{\nu_0}(1/z, 1/w)}{t} = zw\R_\nu(z, w)
\]
for $(z, w)$ in some punctured bi-disk $\Omega^*$ around $(0, 0)$. If we can show the c-bi-free analogue of the above limit (that is,
\[
\lim_{t \to 0^+}\frac{G_{\mu_t}(1/z, 1/w) - G_{\mu_0}(1/z, 1/w)}{t} = zw\R_{(\mu, \nu)}(z, w)
\]
for $(z, w) \in \Omega^*$), then the rest of the proof would be identical.

For notational simplicity, denote $R_{1, 1} = \R_{\nu^{(1)}}(z)$, $R_{1, 2} = \R_{\nu^{(2)}}(w)$, $R_{2, 1} = \R_{(\mu^{(1)}, \nu^{(1)})}(z)$, $R_{2, 2} = \R_{(\mu^{(2)}, \nu^{(2)})}(w)$, $R_1 = \R_\nu(z, w)$, $R_2 = \R_{(\mu, \nu)}(z, w)$,
\[
G(t, z, w) = G_{\mu_t}(z, w),\quad K_{\nu_t^{(1)}}(z) = tR_{1,1} + \frac{1}{z},\qand K_{\nu_t^{(2)}}(w) = wR_{1,2} + \frac{1}{w}.
\]
By Definition \ref{CBF-R}, and the fact that $\mu$ and $\nu$ are compactly supported, there exists $t' > 0$ such that
\begin{align*}
G(t, K_{\nu_t^{(1)}}(z), K_{\nu_t^{(2)}}(w)) = \frac{(t R_2 - tzR_{2,1} - twR_{2,2})G_{\nu_t}(K_{\nu_t^{(1)}}(z), K_{\nu_t^{(2)}}(w)) + zw}{(1 + tzR_{1,1} - tzR_{2,1})(1 + twR_{1,2} - twR_{2,2})}
\end{align*}
for $0 \leq t \leq t'$ and $(z, w) \in \Omega^*$.

Replacing $t'$ with a smaller value if necessary, and using the fact that $\nu$ is $\boxplus\boxplus$-infinitely divisible, we have (see \cite{HW2016}*{Section 4})
\[
G_{\nu_t}(K_{\nu_t^{(1)}}(z), K_{\nu_t^{(2)}}(w)) = \frac{zw}{1 + tzR_{1,1} + twR_{1,2} - tR_1}.
\]
Hence
\begin{align}
\label{Cauchytimet}
G(t, K_{\nu_t^{(1)}}(z), K_{\nu_t^{(2)}}(w)) = \frac{zw(1 + tzR_{1,1} - tzR_{2,1} + twR_{1,2} - twR_{2,2} - tR_1 + tR_2)}{(1 + tzR_{1,1} - tzR_{2,1})(1 + twR_{1,2} - twR_{2,2})(1 + tzR_{1,1} + twR_{1,2} - tR_1)}
\end{align}
for $0 \leq t \leq t'$ and $(z, w) \in \Omega^*$. For $(z, w) \in \Omega^*$, differentiating both sides of equation \eqref{Cauchytimet} with respect to $t$ produces
\begin{align*}
&\partial_tG(t, K_{\nu_t^{(1)}}(z), K_{\nu_t^{(2)}}(w)) + R_{1,1} \partial_zG(t, K_{\nu_t^{(1)}}(z), K_{\nu_t^{(2)}}(w)) + R_{1,2}\partial_wG(t, K_{\nu_t^{(1)}}(z), K_{\nu_t^{(2)}}(w))\\
&= \frac{d}{dt}\frac{zw(1 + tzR_{1,1} - tzR_{2,1} + twR_{1,2} - twR_{2,2} - tR_1 + tR_2)}{(1 + tzR_{1,1} - tzR_{2,1})(1 + twR_{1,2} - twR_{2,2})(1 + tzR_{1,1} + twR_{1,2} - tR_1)}\\
&= \frac{F}{(1 + tzR_{1, 1} - tzR_{2, 1})^2(1 + twR_{1, 2} - twR_{2, 2})^2(1 + tzR_{1, 1} + twR_{1, 2} - tR_1)^2}
\end{align*}
where $F$ is a function in $\{z, w, R_{1, 1}, R_{1, 2}, R_{2, 1}, R_{2, 2}, R_1, R_2\}$ obtained from the quotient rule. By the discussion preceding \cite{HW2016}*{Proposition 4.1}, we may take the derivative at $t = 0$ in the above equation, and arrive after some simplification at
\begin{align*}
&\lim_{t \to 0^+}\frac{G_{\mu_t}(1/z, 1/w) - G_{\mu_0}(1/z, 1/w)}{t} - z^2wR_{1, 1} - zw^2R_{1, 2}\\
&= zw(zR_{1, 1} - zR_{2, 1} + wR_{1, 2} - wR_{2, 2} - R_1 + R_2) \\
& \quad - zw(zR_{1, 1} - zR_{2, 1} + wR_{1, 2} - wR_{2, 2} + zR_{1, 1} + wR_{1, 2} - R_1)\\
&= zwR_2 - z^2wR_{1, 1} - zw^2R_{1, 2}
\end{align*}
as required.
\end{proof}

Note that assertion $(2)$ of Theorem \ref{Unbound} is equivalent to assertion $(2)$ of Theorem \ref{CompSupp} when $\mu$ and $\nu$ are compactly supported via the change of variables $(z, w) \mapsto (1/z, 1/w)$ and a substitution similar to the bi-free case given at the end of \cite{HW2016}*{Section 4}.

\begin{rem}
Given two pairs $(\mu_1, \nu_1)$ and $(\mu_2, \nu_2)$ of Borel probability measures on $\bR$, an interesting special case of the additive c-free convolution occurs when $\nu_1 = \nu_2 = \delta_0$. In this case, we have $(\mu_1, \delta_0) \boxplus_{\rc} (\mu_2, \delta_0) = (\mu, \delta_0)$ where $\mu = \mu_1 \uplus \mu_2$ is the additive Boolean convolution of $\mu_1$ and $\mu_2$ corresponding to the Boolean independence (see \cite{SW1997}). On the combinatorial level, this happens when the state $\psi$ of the two-state non-commutative probability $(\A, \varphi, \psi)$ is the delta state, i.e., $\psi(\alpha) = \alpha$ for all $\alpha \in \mathbb{C}$ and $\psi(a) = 0$ for all $a \notin \mathbb{C}$, and thus only outer blocks survive in the definition of the c-free cumulants. On the analytic level, the c-free Voiculescu transform of the pair $(\mu, \delta_0)$ becomes
\[
\Phi_{(\mu, \delta_0)}(z) = F_{\delta_0}^{-1}(z) - F_\mu(F_{\delta_0}^{-1}(z)) = z - F_\mu(z),
\]
which is the Boolean self-energy $E_\mu$ of the measure $\mu$ (see \cite{SW1997}*{Section 3}). Analogously, if $\psi$ is the delta state as above, then only exterior blocks survive in the definition of the c-$(\ell, r)$-cumulants. This leads to a Boolean type additive convolution for two-faced families, which we call the additive bi-Boolean convolution and denote by $\uplus\uplus$. On the analytic level, given two Borel probability measures $\mu_1$ and $\mu_2$ on $\bR^2$, the measure $\mu$ defined by $(\mu, \delta_{(0, 0)}) = (\mu_1, \delta_{(0, 0)}) \boxplus\boxplus_{\rc} (\mu_2, \delta_{(0, 0)})$ is exactly $\mu_1 \uplus\uplus \mu_2$. Moreover, the c-bi-free partial Voiculescu transform $\Phi_{(\mu, \delta_{(0, 0)})}$ of the pair $(\mu, \delta_{(0, 0)})$ becomes
\[
\Phi_{(\mu, \delta_{(0, 0)})}(z, w) = \frac{1}{z}E_{\mu^{(1)}}(z) + \frac{1}{w}E_{\mu^{(2)}}(w) + \frac{G_\mu(z, w)}{G_{\mu^{(1)}}(z)G_{\mu^{(2)}}(w)} - 1 := E_\mu(z, w),
\]
which would be the linearizing transform of the measure $\mu$ with respect to $\uplus\uplus$. We intend to investigate this topic further in a forthcoming paper.
\end{rem}

\end{document}